\documentclass[11pt]{article}

\usepackage[a4paper,margin=1in]{geometry}
\usepackage[T1]{fontenc}
\usepackage[utf8]{inputenc}
\usepackage{lmodern}
\usepackage{microtype}
\usepackage{amsmath,amssymb,amsfonts,amsthm,mathtools}
\usepackage{mathrsfs}
\usepackage{enumitem}
\usepackage{booktabs}
\usepackage{tocloft}
\usepackage[colorlinks=true,linkcolor=blue,citecolor=blue,urlcolor=blue]{hyperref}
\usepackage{aliascnt}
\usepackage[capitalize,nameinlink]{cleveref}

\numberwithin{equation}{section}

\newtheorem{theorem}{Theorem}[section]

\newaliascnt{proposition}{theorem}
\newtheorem{proposition}[proposition]{Proposition}
\aliascntresetthe{proposition}

\newaliascnt{lemma}{theorem}
\newtheorem{lemma}[lemma]{Lemma}
\aliascntresetthe{lemma}

\newaliascnt{corollary}{theorem}
\newtheorem{corollary}[corollary]{Corollary}
\aliascntresetthe{corollary}

\theoremstyle{definition}
\newaliascnt{definition}{theorem}
\newtheorem{definition}[definition]{Definition}
\aliascntresetthe{definition}

\newaliascnt{assumption}{theorem}
\newtheorem{assumption}[assumption]{Assumption}
\aliascntresetthe{assumption}

\newaliascnt{convention}{theorem}
\newtheorem{convention}[convention]{Convention}
\aliascntresetthe{convention}

\newaliascnt{problem}{theorem}
\newtheorem{problem}[problem]{Problem}
\aliascntresetthe{problem}

\newaliascnt{example}{theorem}

\aliascntresetthe{example}

\theoremstyle{remark}
\newaliascnt{remark}{theorem}
\newtheorem{remark}[remark]{Remark}
\aliascntresetthe{remark}

\newaliascnt{warning}{theorem}

\aliascntresetthe{warning}

\crefname{assumption}{Assumption}{Assumptions}
\Crefname{assumption}{Assumption}{Assumptions}
\crefname{definition}{Definition}{Definitions}
\crefname{convention}{Convention}{Conventions}
\Crefname{convention}{Convention}{Conventions}
\Crefname{definition}{Definition}{Definitions}
\crefname{lemma}{Lemma}{Lemmas}
\Crefname{lemma}{Lemma}{Lemmas}
\crefname{proposition}{Proposition}{Propositions}
\Crefname{proposition}{Proposition}{Propositions}
\crefname{corollary}{Corollary}{Corollaries}
\Crefname{corollary}{Corollary}{Corollaries}
\crefname{theorem}{Theorem}{Theorems}
\Crefname{theorem}{Theorem}{Theorems}
\crefname{warning}{Warning}{Warnings}
\Crefname{warning}{Warning}{Warnings}
\crefname{problem}{Problem}{Problems}
\Crefname{problem}{Problem}{Problems}

\DeclareMathOperator*{\esssup}{ess\,sup}
\DeclareMathOperator{\dist}{dist}
\DeclareMathOperator{\Range}{Range}

\DeclareMathOperator{\Tr}{Tr}

\newcommand{\R}{\mathbb R}
\newcommand{\T}{\mathbb T}

\newcommand{\eps}{\varepsilon}

\newcommand{\CKN}{\mathrm{CKN}}

\newcommand{\loc}{\mathrm{loc}}
\newcommand{\cl}{\mathrm{cl}}

\newcommand{\Err}{\mathsf{Err}}

\newcommand{\Rep}{\mathsf{Rep}}
\newcommand{\Prof}{\mathsf{P}}
\newcommand{\Dist}{\operatorname{dist}}
\newcommand{\Image}{\operatorname{Im}}
\newcommand{\dx}{\,dx}

\newcommand{\dxdt}{\,dx\,dt}
\newcommand{\nabh}{\nabla_h}
\newcommand{\divh}{\nabla_h\cdot}
\newcommand{\norm}[2]{\left\|#1\right\|_{#2}}

\newcommand{\pair}[2]{\left\langle #1,#2\right\rangle}
\newcommand{\calC}{\mathcal C}
\newcommand{\calD}{\mathcal D}
\newcommand{\calE}{\mathcal E}
\newcommand{\calF}{\mathcal F}

\newcommand{\calO}{\mathcal O}

\newcommand{\calU}{\mathcal U}
\newcommand{\calY}{\mathcal Y}
\newcommand{\Q}{Q}

\title{Finite-Window Singularity Audits and Local-to-Clean Defect Transfer for Navier--Stokes}
\author{Runlong Yu\\
The University of Alabama, Tuscaloosa, AL, USA\\
\texttt{ryu5@ua.edu}}
\date{}

\begin{document}
\maketitle

\begin{abstract}
We develop a finite-window audit and transfer framework for the local regularity problem of the three-dimensional incompressible Navier--Stokes equations.  The paper has two logically separate components.  First, a finite-scale critical ledger estimate shows that a persistent non-CKN branch cannot survive without cumulative untaxed critical supply or accumulated localization leakage.  Second, a local-to-clean transfer theorem shows how a clean finite-window anti-phantom gap gives a localized coercive estimate for Navier--Stokes defect packages, provided three structural inputs hold: quotient-lifting stability, componentwise detection comparison, and a normalized residual budget. The technical contribution is an explicit bookkeeping of the residuals that can obstruct this transfer.  Pressure tails, cutoff leakage, truncation loss, nonlinear cutoff mismatch, reproduction drift, gauge mismatch, and profit discrepancy are separated as finite-window ledger entries and then assembled into a concrete absorption criterion.  The main transfer theorem is algebraic once these entries satisfy the stated normalized bounds; the PDE-facing work left open is precisely to prove the corresponding component compatibility estimates against the localized quotient distance.  Thus the article is a conditional reduction. Its conclusion is structural: any surviving obstruction must either defeat the stated compatibility/growth controls or appear as an NS-realizable, cleaned, scale-critical, combined-invisible, profitable, reproducible defect cascade.
\end{abstract}



\tableofcontents

\section{Introduction}\label{sec:introduction}

The local regularity problem for the three-dimensional incompressible
Navier--Stokes equations is usually phrased in terms of scale-critical
quantities.  The weak-solution background goes back to Leray--Hopf theory,
while the modern local partial-regularity viewpoint begins with Scheffer and
Caffarelli--Kohn--Nirenberg and was refined in later proofs and expositions
\cite{Leray1934,Hopf1951,Scheffer1976,Scheffer1977,CKN1982,Lin1998,SereginLectureNotes}.
If a suitable weak solution enters a Caffarelli--Kohn--Nirenberg smallness
regime at some scale, then local regularity follows.  Critical-norm,
pressure-sensitive, and quantitative regularity perspectives provide further
context for the finite-scale formulation used here
\cite{SohrWahl1986,SereginSverak2002,ESS2003,BarkerPrange2021,AlbrittonBarkerPrange2023}.
Thus a possible singular point must avoid that decay basin through an infinite
sequence of parabolic windows.  This paper asks a more structural question:
\begin{quote}
\emph{If a non-CKN branch persists across scales, what precise mechanism pays for
that persistence, hides from the available finite-window observations, reproduces
from scale to scale, and remains realizable by the Navier--Stokes equation?}
\end{quote}

The answer is formulated as a finite-window singularity audit coupled to a
local-to-clean transfer theorem.  The audit side shows that persistent critical
badness cannot survive for free: along a finite chain of windows it must generate
untaxed critical supply or accumulated leakage.  The transfer side then asks
whether a clean finite-window anti-phantom gap can be transported back to a
localized Navier--Stokes package without hiding pressure tails, cutoff leakage,
truncation error, reproduction drift, gauge mismatch, or profit discrepancy.

The result is not a proof of Navier--Stokes regularity and not a construction of
a singular solution.  Its contribution is a theorem-level reduction: a surviving
singular branch must either defeat finite-window observability and residual
absorption, or organize itself as an NS-realizable, cleaned, scale-critical,
combined-invisible, profitable, reproducible defect cascade.

\subsection{Main local-to-clean theorem, informally}

The central theorem of the paper is the local-to-clean transfer theorem proved in
\Cref{sec:main-reduction}.  Let \(\mathfrak D\in\calD^{\loc}_\Lambda\) be a
localized finite-window package and let \(\Theta_\Lambda\mathfrak D\) be its clean
charted image.  Suppose that the clean model has a quotient anti-phantom gap
\[
    \|\calF^{\cl}_\Lambda(d)\|_{\cl}
    \ge
    c^{\cl}_\Lambda
    \Dist_{\cl}(d,\Image G^{\cl}_\Lambda).
\]
Assume also that clean quotient distance lifts back to localized quotient
distance and that the total residual ledger has the normalized form
\[
    \Err_\Lambda(\mathfrak D)
    \le
    \eta_\Lambda
    \Dist_{\loc}(\mathfrak D,\Image G^{\loc}_\Lambda)
    +
    \Delta_\Lambda,
    \qquad
    \eta_\Lambda<c^{\cl}_\Lambda(1-\varepsilon_G).
\]
Then
\[
\begin{aligned}
    &\|O^{\loc}_\Lambda \mathfrak D\|
    +C_E\|\calE^{\loc}_\Lambda(\mathfrak D)\|
    +C_R\Rep^{\loc}_\Lambda(\mathfrak D)
    +[\Prof^{\loc}_\Lambda(\mathfrak D)]_+  \\
    &\qquad\ge
    c^{\loc}_\Lambda
    \Dist_{\loc}(\mathfrak D,\Image G^{\loc}_\Lambda)
    -
    \Delta'_\Lambda,
\end{aligned}
\]
where \(c^{\loc}_\Lambda=c^{\cl}_\Lambda(1-\varepsilon_G)-\eta_\Lambda>0\).  In
words, a non-gauge localized defect cannot remain invisible after transfer
unless it is paid for by a displayed residual, a reproduction failure, a positive
ledger-profit term, or an error budget.

\subsection{What is proved, assumed, and left open}

The paper is organized to keep the mathematical status of each layer explicit.
The following table is the intended dependency map.
\begin{center}
\renewcommand{\arraystretch}{1.18}
\begin{tabular}{p{0.24\textwidth}p{0.32\textwidth}p{0.34\textwidth}}
\toprule
\textbf{Layer} & \textbf{What is proved here} & \textbf{What remains an input or target} \\
\midrule
Finite-scale ledger &
The local energy inequality, interpolation, and pressure decay imply a finite-scale survival alternative: persistent reservoir badness forces cumulative supply or leakage. &
Uniform taxation or observable depletion of all critical supply along moving windows. \\
\addlinespace
Finite-window audit &
NS-derived packages satisfy coarse momentum, pressure compatibility, Reynolds positivity, and an energy-flux identity; fixed-window invisibility is a finite-dimensional kernel question. &
Exclusion of nonzero combined-invisible directions in the closure of cleaned NS-realizable packages, with moving-window constants controlled. \\
\addlinespace
Local-to-clean transfer &
A clean anti-phantom gap transfers to a localized coercive estimate once quotient lifting, detection comparison, and residual normalization hold. &
Concrete Navier--Stokes estimates proving those three structural inputs uniformly enough for a singular branch. \\
\addlinespace
Residual absorption &
Pressure, localization, truncation, nonlinear, reproduction, gauge, and profit discrepancies are decomposed into explicit finite-window budgets. &
Component compatibility bounds of the form
\(\Err_\bullet\le \eta_\bullet \Dist_{\loc}+\Delta_\bullet\), with total coefficient below the clean gap margin. \\
\bottomrule
\end{tabular}
\end{center}

This separation is part of the statement of the paper.  The finite-scale ledger
is the unconditional PDE input.  The main local-to-clean theorem is an algebraic
transfer result under explicit structural hypotheses.  The normalized absorption
problem identifies the first genuine remaining blocker.

\subsection{Main innovations}

The paper has four connected innovations.
\begin{enumerate}[label=\textbf{I\arabic*.},leftmargin=3.2em]
\item \textbf{Finite-scale payment ledger.}  Persistent reservoir badness is
converted into a quantitative alternative: cumulative untaxed supply or leakage
must appear along any long non-CKN branch.
\item \textbf{NS-realizable defect audit.}  A dangerous obstruction is not an
arbitrary kernel vector.  It must satisfy coarse momentum, pressure compatibility,
Reynolds positivity, and the energy-flux identity inherited from suitable weak
solutions.
\item \textbf{Local-to-clean transfer.}  A clean anti-phantom gap is transferred
to a localized finite-window model under explicit quotient-lifting, component
comparison, and normalized error-budget assumptions.
\item \textbf{Residual-budget closure map.}  Pressure, localization, truncation,
nonlinear, reproduction, gauge, and profit errors are placed into a single
concrete ledger.  The final obstruction is isolated as a normalized absorption
problem against the localized quotient distance.
\end{enumerate}

\subsection{Boundary of the paper}

The paper deliberately does not claim a new proof of Navier--Stokes regularity.
It proves a finite-scale ledger estimate and an algebraic local-to-clean transfer
theorem, then records the exact compatibility estimates that would be needed to
turn this finite-window mechanism into a regularity argument.  Conversely, any
counterexample mechanism surviving this framework would have to satisfy the same
compatibility, cleaning, observability, profitability, and reproduction tests.

\subsection{Organization}

\Cref{sec:prelim-audit,sec:payment,sec:hiding,sec:realization,sec:reproduction}
formulate the singularity audit: payment, hiding, realization, and reproduction.
\Cref{sec:audit-to-transfer} explains the transition from the audit language to
the local-to-clean transfer problem.  \Cref{sec:lct-preliminaries,sec:localized-packages,sec:chart,sec:detection-error,sec:main-reduction}
develop the transfer theorem and its residual budgets.  \Cref{sec:absorption-audit}
assembles the concrete absorption criterion, and \Cref{sec:final-synthesis}
records the resulting conditional synthesis and closure targets.

\section{Suitable weak solutions, CKN coordinates, and bad branches}\label{sec:prelim-audit}

For a space-time point $z_0=(x_0,t_0)$, write
\[
Q_r(z_0)=B_r(x_0)\times(t_0-r^2,t_0).
\]
When $z_0=(0,0)$, write $Q_r=Q_r(0,0)$.

A pair $(u,p)$ is a suitable weak solution if
\[
u\in L^\infty_tL^2_x\cap L^2_tH^1_x,\qquad p\in L^{3/2},
\]
solves (NS) distributionally, is divergence-free, and satisfies the local energy inequality: for every nonnegative $\phi\in C_c^\infty$,
\begin{align}
\int |u(x,t)|^2\phi(x,t)\dx+2\int_{-\infty}^t\int |\nabla u|^2\phi\dx\,ds
&\le \int_{-\infty}^t\int |u|^2(\partial_s\phi+\Delta\phi)\dx\,ds\notag\\
&\quad+\int_{-\infty}^t\int (|u|^2+2p)u\cdot\nabla\phi\dx\,ds.
\label{eq:local-energy}
\end{align}

The scale-critical quantities are
\begin{align*}
A(z_0,r)&=r^{-1}\esssup_{t_0-r^2<t<t_0}\int_{B_r(x_0)}|u(x,t)|^2\dx,\\
E(z_0,r)&=r^{-1}\int_{Q_r(z_0)}|\nabla u|^2\dxdt,\\
C(z_0,r)&=r^{-2}\int_{Q_r(z_0)}|u|^3\dxdt,\\
D(z_0,r)&=r^{-2}\int_{Q_r(z_0)}|p-(p)_{B_r(x_0)}(t)|^{3/2}\dxdt.
\end{align*}
We also set
\[
\Psi(z_0,r)=C(z_0,r)+D(z_0,r),\qquad
\Phi(z_0,r)=A(z_0,r)+E(z_0,r)+C(z_0,r)+D(z_0,r).
\]

The CKN criterion gives a universal $\eps_{\CKN}>0$ such that
\[
\Psi(z_0,r)\le \eps_{\CKN}
\quad\Longrightarrow\quad
u \text{ is regular in a smaller cylinder.}
\]
Therefore a non-CKN branch is a sequence of scales $r_k\downarrow0$ for which $\Psi(z_0,r_k)>\eps_{\CKN}$ for all relevant $k$.

\begin{definition}[Admissible scale-window chain]
Fix $0<\theta<1$.  A sequence
\[
r_k=\theta^k r_0,
\qquad
Q_k=Q_{r_k}(z_k),
\]
is called an admissible scale-window chain if $Q_{k+1}\subset Q_k$ and the centers have controlled parabolic drift:
\[
|x_{k+1}-x_k|\le \eta r_k,
\qquad
|t_{k+1}-t_k|\le \eta r_k^2.
\]
\end{definition}

The reservoir badness used by the ledger is
\[
B_k=A_k+C_k+D_k,
\]
where $A_k,C_k,D_k$ denote the scale-critical quantities on $Q_k$.  The dissipation $E_k$ is not placed in the reservoir; it appears mainly as a tax.

\section{Payment: the finite-scale critical ledger}\label{sec:payment}

This section records the first answer to the audit question.  If badness persists, it must be paid for.  The payment appears as untaxed critical supply or localization leakage.

Choose cutoffs $\phi_k\in C_c^\infty(Q_k)$ satisfying
\[
0\le\phi_k\le1,
\qquad
\phi_k\equiv1\text{ on }Q_{k+1},
\]
and
\[
|\nabla\phi_k|\le C_\theta r_k^{-1},
\qquad
|\partial_t\phi_k|+|\Delta\phi_k|\le C_\theta r_k^{-2}.
\]
Define transition ledger coordinates
\begin{align*}
\Phi_k^{\mathrm{flux}}&=r_k^{-1}\int_{Q_k}|u|^2|u\cdot\nabla\phi_k|\dxdt,\\
\Pi_k^{\mathrm{press}}&=r_k^{-1}\int_{Q_k}|p-(p)_{B_{r_k}(x_k)}(t)|\,|u\cdot\nabla\phi_k|\dxdt,\\
\Lambda_k&=r_k^{-1}\int_{Q_k}|u|^2(|\partial_t\phi_k|+|\Delta\phi_k|)\dxdt.
\end{align*}
The first is nonlinear flux supply, the second is pressure-transport supply, and the third is cutoff/window leakage.

\begin{lemma}[Local energy ledger]
\label[lemma]{lem:local-energy-ledger}
For every suitable weak solution and every transition $Q_k\to Q_{k+1}$,
\[
A_{k+1}+2E_{k+1}
\le
\theta^{-1}\bigl(\Lambda_k+\Phi_k^{\mathrm{flux}}+2\Pi_k^{\mathrm{press}}\bigr).
\]
\end{lemma}

\begin{proof}
Apply the local energy inequality \eqref{eq:local-energy} with the cutoff $\phi_k$ in the suitable-weak-solution sense of the CKN framework \cite{CKN1982,SereginLectureNotes}, first with a standard time cutoff that approaches the characteristic function of the terminal time interval and then pass to the usual almost-everywhere time representatives.  Taking the essential supremum over terminal times in the smaller cylinder gives the $A_{k+1}$ term, while the spacetime part gives $E_{k+1}$.  Since $\phi_k\equiv1$ on $Q_{k+1}$, the left-hand side controls the smaller-scale energy and dissipation.  The right-hand side separates into the cutoff leakage term, the nonlinear flux term, and the pressure-transport term.  The spatial average of the pressure may be subtracted because $\nabla\cdot u=0$ and, for each time, $\int u\cdot\nabla\phi_k\,dx=0$ after the standard compact-support approximation.
\end{proof}

The velocity and pressure reservoirs are closed by the standard interpolation and pressure-decay estimates:
\begin{align}
C_{k+1}&\le C_{I,\theta}\Bigl((\Phi_k^{\mathrm{flux}}+2\Pi_k^{\mathrm{press}})^{3/2}+\Lambda_k^{3/2}\Bigr),
\label{eq:cubic-closure}\\
D_{k+1}&\le C_P\theta D_k+C_P\theta^{-2}C_k.
\label{eq:pressure-decay}
\end{align}
The first follows by applying the scale-invariant interpolation inequality
\[
    C_{k+1}\lesssim A_{k+1}^{3/4}E_{k+1}^{3/4}+A_{k+1}^{3/2}
\]
to \Cref{lem:local-energy-ledger} and absorbing constants depending only on the fixed scale ratio \(\theta\).  The second is the usual local Calderon--Zygmund plus harmonic-pressure decomposition \cite{SohrWahl1986,SereginSverak2002,SereginLectureNotes}.

Fix $0<\lambda<1$ and choose $\theta$ so that
\[
\delta_D=(1-\lambda)-C_P\theta>0.
\]
Define
\begin{align*}
\mathrm{Sup}^{\mathrm{full}}_k
&=\theta^{-1}(\Phi_k^{\mathrm{flux}}+2\Pi_k^{\mathrm{press}})
+C_{I,\theta}(\Phi_k^{\mathrm{flux}}+2\Pi_k^{\mathrm{press}})^{3/2}
+C_P\theta^{-2}C_k,\\
\mathrm{Tax}^{\mathrm{full}}_k
&=2E_{k+1}+(1-\lambda)A_k+(1-\lambda)C_k+\delta_DD_k,\\
\mathrm{Leak}^{\mathrm{full}}_k
&=\theta^{-1}\Lambda_k+C_{I,\theta}\Lambda_k^{3/2}.
\end{align*}

\begin{theorem}[Full critical ledger inequality]
\label[theorem]{thm:full-critical-ledger-inequality}
Along every admissible scale-window chain,
\begin{equation}\label{eq:full-ledger}
B_{k+1}-(1-\lambda)B_k
\le
\mathrm{Sup}^{\mathrm{full}}_k-
\mathrm{Tax}^{\mathrm{full}}_k+
\mathrm{Leak}^{\mathrm{full}}_k.
\end{equation}
\end{theorem}

\begin{proof}
Add the energy transition estimate, the interpolation bound \eqref{eq:cubic-closure}, and the pressure decay estimate \eqref{eq:pressure-decay}.  Subtract $(1-\lambda)(A_k+C_k+D_k)$.  The terms are exactly those grouped into supply, tax, and leakage.
\end{proof}

\begin{theorem}[Finite-scale survival alternative]
\label[theorem]{thm:finite-scale-survival-alternative}
Assume that
\[
B_k\ge\eps,
\qquad 0\le k\le N-1.
\]
Then
\begin{equation}\label{eq:survival-alt}
\sum_{k=0}^{N-1}\bigl(\mathrm{Sup}^{\mathrm{full}}_k-
\mathrm{Tax}^{\mathrm{full}}_k\bigr)_+
\ge
\lambda\eps N-B_0-
\sum_{k=0}^{N-1}\mathrm{Leak}^{\mathrm{full}}_k.
\end{equation}
\end{theorem}

\begin{proof}
Let $M_k=B_{k+1}-(1-\lambda)B_k$.  Summing gives
\[
\sum_{k=0}^{N-1}M_k=B_N-B_0+\lambda\sum_{k=0}^{N-1}B_k
\ge -B_0+\lambda\eps N.
\]
On the other hand, \eqref{eq:full-ledger} implies
\[
\sum_{k=0}^{N-1}M_k
\le
\sum_{k=0}^{N-1}\bigl(\mathrm{Sup}^{\mathrm{full}}_k-
\mathrm{Tax}^{\mathrm{full}}_k\bigr)
+\sum_{k=0}^{N-1}\mathrm{Leak}^{\mathrm{full}}_k.
\]
Since $X\le X_+$ for every real number $X$, \eqref{eq:survival-alt} follows.
\end{proof}

The meaning of \eqref{eq:survival-alt} is the first central conclusion:
\[
\boxed{\text{long survival of critical badness forces positive average untaxed supply or leakage.}}
\]
In particular, if leakage has vanishing average along an infinite branch, then a persistent bad orbit requires positive-density untaxed supply.

\section{Hiding: NS-realizable defect packages and combined observability}\label{sec:hiding}

The ledger tells us that badness must be paid for.  The next question is how a paid mechanism can remain hidden.  The answer requires passing from raw scale quantities to equation-generated defect packages.

\subsection{Dyadic rescaling and coarse graining}

At a point $z_0=(0,0)$, let $r_n=2^{-n}$ and define the rescaled fields
\[
u^{(n)}(x,t)=r_nu(r_nx,r_n^2t),
\qquad
p^{(n)}(x,t)=r_n^2p(r_nx,r_n^2t).
\]
The CKN quantities are invariant under this scaling.

Let $S_\ell$ be an interior parabolic mollifier at scale $\ell$.  Define
\[
U_{n,\ell}=S_\ell u^{(n)},
\qquad
P_{n,\ell}=S_\ell p^{(n)},
\]
and the Reynolds defect
\[
R_{n,\ell}=S_\ell(u^{(n)}\otimes u^{(n)})-U_{n,\ell}\otimes U_{n,\ell}.
\]
Because $S_\ell$ is an averaging operator, $R_{n,\ell}$ is nonnegative as a quadratic form:
\[
R_{n,\ell}\ge0.
\]
The unresolved energy density is
\[
\kappa_{n,\ell}=\frac12\Tr R_{n,\ell}\ge0.
\]

The coarse-grained fields satisfy the exact identity
\begin{equation}\label{eq:coarse-momentum}
\partial_tU-\Delta U+\nabla\cdot(U\otimes U)+\nabla P=-\nabla\cdot R,
\qquad
\nabla\cdot U=0.
\end{equation}
Taking divergence gives pressure compatibility:
\begin{equation}\label{eq:pressure-compat}
-\Delta P=\partial_i\partial_j(U_iU_j+R_{ij}).
\end{equation}
The active pressure $P^{\mathrm{act}}$ is the chosen local solution generated by the source in \eqref{eq:pressure-compat}; the retained harmonic pressure is
\[
P^{\mathrm{har}}=P-P^{\mathrm{act}},
\qquad \Delta P^{\mathrm{har}}=0
\]
in the local pressure window.  Only functions of time are pure pressure gauge.  Spatially harmonic pressure is not discarded; it still contributes to momentum and pressure work.

The coarse-grained flux is
\[
\Pi_{n,\ell}=-R_{n,\ell}:\nabla U_{n,\ell}.
\]
Dotting \eqref{eq:coarse-momentum} with $U$ gives
\begin{align}
\partial_t\frac12|U|^2-\Delta\frac12|U|^2+|\nabla U|^2
+\nabla\cdot\left[\left(\frac12|U|^2+P\right)U+RU\right]
=-\Pi.
\label{eq:coarse-energy}
\end{align}
Thus the package
\[
D_{n,\ell}=(U_{n,\ell},P_{n,\ell};P^{\mathrm{act}}_{n,\ell},P^{\mathrm{har}}_{n,\ell},R_{n,\ell},\Pi_{n,\ell})
\]
is not an arbitrary finite-dimensional defect.  It is \emph{NS-realizable}: it comes from an actual suitable weak solution and satisfies \eqref{eq:coarse-momentum}, \eqref{eq:pressure-compat}, $R\ge0$, and \eqref{eq:coarse-energy}.  The coarse-grained flux and local energy-transfer viewpoint is closely related to the Constantin--E--Titi commutator method, Eyink's local energy-transfer formulation, and the Duchon--Robert defect-measure perspective
\cite{ConstantinETiti1994,Eyink1994,DuchonRobert2000}.

\subsection{Combined observation channels}

A finite observation window is a tuple
\[
W=(n,\ell,\Lambda,\chi,s_*),
\]
where $\Lambda$ is a finite-dimensional frequency or test window, $\chi$ is a cutoff, and $s_*$ is a selected time.  Let $Y_W$ denote the cleaned finite-dimensional constrained defect space obtained from the projected variables after imposing the finite-window versions of momentum, pressure compatibility, and flux identity, and after quotienting only exact gauge directions.

The combined observation map has four channels:
\[
\calO^{\mathrm{comb}}_W=(\calO^P_W,\calO^F_W,\calO^E_W,\calO^T_W).
\]
Here
\begin{itemize}[leftmargin=2em]
\item $\calO^P_W$ observes the active pressure source;
\item $\calO^F_W$ observes the linearized interscale flux;
\item $\calO^E_W$ observes the positive energy/dissipation channel;
\item $\calO^T_W$ observes selected-time trace directions, equivalently through the adjoint map $A_W^*$.
\end{itemize}
The hierarchy is
\[
\text{active pressure}\ \Rightarrow\ \text{flux}\ \Rightarrow\ \text{energy}\ \Rightarrow\ \text{adjoint trace}.
\]

\begin{definition}[Combined-invisible direction]
A nonzero $d\in Y_W$ is called combined-invisible in the window $W$ if
\[
\calO^P_Wd=\calO^F_Wd=\calO^E_Wd=\calO^T_Wd=0.
\]
It is called a true finite-window phantom if, in addition, it lies in the closure of cleaned projections of NS-realizable packages.
\end{definition}

This distinction is essential.  A formal finite-dimensional kernel is not automatically a Navier--Stokes obstruction.  It becomes dangerous only after intersection with the range of actual NS-derived packages.

\begin{proposition}[Fixed-window observability criterion]
For a fixed finite window $W$, the following are equivalent:
\begin{enumerate}[label=(\roman*)]
\item there is no nonzero combined-invisible direction in $Y_W$;
\item there exists a finite constant $M_W$ such that
\[
\norm{d}{Y_W}\le M_W\Bigl(
orm{\calO^P_Wd}{}+
orm{\calO^F_Wd}{}+
orm{\calO^E_Wd}{}+
orm{\calO^T_Wd}{}\Bigr)
\]
for all $d\in Y_W$.
\end{enumerate}
\end{proposition}

\begin{proof}
This is the compactness of the unit sphere in the finite-dimensional space $Y_W$.  If a nonzero invisible direction exists, the estimate fails.  If not, the sum of observation norms is a strictly positive continuous function on the unit sphere, hence has a positive minimum.
\end{proof}

The moving-window problem is the question of whether the constants $M_W$ can be controlled along a sequence of windows $W_n$ selected from a possible singular branch.  Non-effective growth of these constants is itself an obstruction.

\section{Realization: excluding formal phantoms}\label{sec:realization}

The third audit test is realizability.  Many algebraic phantoms are artifacts of the chosen finite-dimensional model.  They are not dangerous unless they can be produced by the Navier--Stokes equation.

\begin{definition}[NS-realizable finite-window direction]
A finite-window direction $d\in Y_W$ is NS-realizable if there exists a sequence of suitable weak solutions, dyadic scales, coarse-graining parameters, and cleaning choices whose projected normalized packages converge to $d$ in $Y_W$.
\end{definition}

The NS-realizability filter has three parts:
\begin{enumerate}[label=(\roman*)]
\item the direction must satisfy the constrained coarse momentum equation;
\item its pressure must be compatible with the active source and retained harmonic pressure;
\item its stress and flux must come from a nonnegative Reynolds covariance and the identity $\Pi=-R:\nabla U$.
\end{enumerate}

This is why the defect package is stronger than a formal envelope.  A formal left singular vector of an observation map might be invisible, but it may fail pressure compatibility, positivity, local energy balance, or vertical momentum compatibility.  Such a vector is a model phantom, not an NS obstruction.

\section{Reproduction: from invisible directions to cascade mechanisms}\label{sec:reproduction}

A single-scale invisible direction is still not enough.  A singularity is a scale-recursive phenomenon.  The relevant object is therefore a mechanism that reproduces.

\begin{definition}[PRV mechanism]
A finite-window mechanism is called profitable, reproducible, and verified, or PRV, if it has the following properties:
\begin{enumerate}[label=(\roman*)]
\item \textbf{Verified:} the projected package satisfies the finite-window Navier--Stokes residual constraints up to an error below the window tolerance;
\item \textbf{Cleaned:} it is not a pressure gauge, cutoff artifact, harmonic leakage artifact, or finite-projection tail;
\item \textbf{Invisible:} the combined observation strength is below the window tolerance;
\item \textbf{Profitable:} its ledger profit $\mathrm{Sup}-\mathrm{Tax}$ has positive averaged contribution after leakage is removed;
\item \textbf{Reproducible:} the package transfers from one scale to the next with small reproduction residual.
\end{enumerate}
\end{definition}

The key point is that a serious bad mechanism cannot be merely static.  It must be dynamically profitable and repeatable.  In this sense, the singularity problem becomes a recurrence problem in a scale-critical state space.

\section{The one-component test and harmonic pressure}

The one-component degeneration problem is a special but important audit test.  It is connected to one-component and one-entry regularity criteria for the three-dimensional Navier--Stokes equations
\cite{KukavicaZiane2006,CaoTiti2011,CheminZhang2016,CheminZhangZhang2017,HanLeiLiZhao2019,KangNguyen2023}.  Suppose
\[
\Phi(1)=A(1)+E(1)+C(1)+D(1)\le M
\]
and
\[
C_3(1)=\int_{Q_1}|u_3|^3\dxdt
\]
is small.  The expected limiting class has
\[
V=(V_h,0),\qquad \divh V_h=0,
\]
and solves the strict two-and-a-half-dimensional system
\[
\partial_tV_h-\Delta V_h+(V_h\cdot\nabh)V_h+\nabh Q=0,
\qquad
\partial_3Q=0.
\]
This is not the usual two-dimensional Navier--Stokes system, because $V_h$ may still depend on $x_3$ and the full Laplacian remains present.

The pressure topology is the decisive point.  One should not demand $p\to Q$ strongly.  The correct local comparison is
\[
p\approx Q+h,
\qquad
\Delta h(\cdot,t)=0.
\]
Thus the natural excess is measured in the harmonic-pressure quotient:
\[
X^{\mathrm{har}}_\rho(u,p;M)=
\inf_{(V,Q)}\inf_{h:\Delta h=0}
\left[\rho^{-2}\int_{Q_\rho}|u-V|^3\dxdt+
\rho^{-2}\int_{Q_\rho}|p-Q-h|^{3/2}\dxdt\right].
\]
The quotient is forced by local pressure decomposition: the Calderon--Zygmund part is compact under strong $L^3$ velocity convergence, while the local harmonic pressure can oscillate in time without being controlled by the same local formula \cite{SohrWahl1986,SereginLectureNotes,YuOneComponent2026}.

\begin{theorem}[Finite-scale one-component regularity; cited input from~\cite{YuOneComponent2026}]
\label[theorem]{thm:finite-scale-one-component-regularity}
For every $M\ge1$, there exist $\eps_*(M)>0$ and $\rho_*(M)>0$ such that every suitable weak solution in $Q_1$ satisfying
\[
\Phi(1)\le M,
\qquad
C_3(1)\le\eps_*(M),
\]
is regular in $Q_{\rho_*(M)}$.
\end{theorem}

This theorem is used here as an external one-component regularity input, not as a result proved in the present article.  It gives a local answer to the audit question.  If the vertical component is too small, the solution cannot sustain a fully three-dimensional bad recurrence at the unit scale.  If one seeks a quantitative logarithmic or power rate, the remaining obstruction is not compactness itself but the strict-shadow selection, Schur visibility, and vertical-duality problem; see~\cite{YuStrict2026,YuSchur2026}.

\subsection{Vertical duality as an anti-phantom principle}

In the strict-shadow reduction, finite-window trace-cost duality has the following form.  Let
\[
A:H\to Y
\]
be the active trace-defect map between finite-dimensional Hilbert spaces.  For an active residual $g\in Y$, define
\[
\mathrm{Cost}_{\mathrm{tr}}(g)=\inf\{\norm{\xi}{H}^2:A\xi=-g\}.
\]
Then
\[
\mathrm{Cost}_{\mathrm{tr}}(g)^{1/2}
=
\sup_{A^*y\ne0}\frac{|\pair{g}{y}|}{\norm{A^*y}{H}},
\]
with value $+\infty$ if $g\notin\Range A$.

Thus the active residual is removable at small trace cost precisely when
\[
|\pair{g}{y}|\le r\norm{A^*y}{H}
\qquad\text{for every }y\in Y^*.
\]
The vertical-duality target says that residuals produced by the full Navier--Stokes vertical momentum equation do not excite strict phantom quotient directions.  This is the one-component version of the broader anti-phantom problem.

\section{Old observables, no-rate branches, and Schur visibility}

A second diagnostic result is negative.  The standard old observable package--scale quantities, harmonic-pressure quotient distance, covariance stress, unresolved variance, good-time trace functionals, finite-window residuals, and trace cost--does not by itself force a logarithmic or power selected-trace rate.  The strict-shadow and Schur-visibility terminology used in this diagnostic layer follows the related one-component modules \cite{YuStrict2026,YuSchur2026}.

In an abstract old trace skeleton, one can construct branches with arbitrarily slow selected-trace convergence while all fixed finite-window residuals vanish.  This is not a Navier--Stokes counterexample.  It is an insufficiency theorem: old observables alone do not prove the desired quantitative selection estimate.

The obstruction separation is as follows:
\begin{enumerate}[label=(\roman*)]
\item parabolic trace-drop removes elementary high-frequency escape, provided drop times intersect admissible good times;
\item fixed-window analyticity removes ordinary finite-dimensional arbitrary-slow behavior through Lojasiewicz-type control;
\item the surviving obstruction is an all-order finite-mode flat branch with potentially non-summable finite-stage constants;
\item the Schur trace-projectability problem asks whether this residual branch is a true NS obstruction or only a strict-shadow phantom;
\item relaxed vertical-pressure visibility and vertical duality are the NS-specific mechanisms expected to eliminate the surviving phantom.
\end{enumerate}

This module explains why the audit needs the realization and anti-phantom filters.  Formal invisibility is not enough; the hidden object must be NS-realizable and must remain invisible after all compatibility relations among pressure, flux, energy, and trace are imposed.

\section{Conceptual singularity-audit alternative}\label{sec:conceptual-audit}

We first state the audit conclusion in a conceptual form.  This result explains
what a persistent non-CKN branch would have to become if leakage, visible supply,
and moving-window growth losses were controlled.  It is not the main algebraic
transfer theorem; the fully quantified local-to-clean estimate is proved later in
\Cref{sec:main-reduction}.  The present alternative is aligned with the
defect-cascade reduction viewpoint developed in~\cite{YuInvisible2026}; the
one-component strict-shadow and Schur-visibility modules are cited separately
in~\cite{YuStrict2026,YuSchur2026}.

\begin{assumption}[Dyadic defect extraction]
\label[assumption]{ass:dyadic-defect-extraction}
Every persistent non-CKN branch with controlled scale-critical size admits, after passing to a moving sequence of windows $W_n$, normalized cleaned NS-realizable finite-window defect directions $d_n\in Y_{W_n}$.
\end{assumption}

\begin{assumption}[Observable depletion]
\label[assumption]{ass:observable-depletion}
There is a depletion functional $\mathcal B_n\ge0$ and an observation strength $\mathfrak O_n(d_n)$ such that whenever $d_n$ is not close to the combined-invisible NS-realizable kernel,
\[
\mathcal B_n-\mathcal B_{n+1}
\ge c\,\omega_n\mathfrak O_n(d_n)-e_n,
\]
where the weights $\omega_n$ are not summably lost and $\sum e_n<\infty$ along controlled branches.
\end{assumption}

\begin{assumption}[Moving-window growth control]
\label[assumption]{ass:moving-window-growth-control}
The observability constants and cleaning constants associated with $W_n$ grow slowly enough to be absorbed by the scale gains in the finite-window argument.
\end{assumption}

\begin{assumption}[Anti-phantom or NS-realizability exclusion]
\label[assumption]{ass:anti-phantom-exclusion}
There is no nonzero cleaned NS-realizable defect cascade which is simultaneously scale-critical, combined-invisible, profitable, and reproducible along the selected moving windows.
\end{assumption}

\begin{theorem}[Conditional singularity-audit alternative]
\label[theorem]{thm:singularity-audit-alternative}
Let $(u,p)$ be a suitable weak solution near $z_0$.  Suppose a dyadic branch based at $z_0$ remains outside the CKN smallness regime for all sufficiently small scales.  Assume \Cref{ass:dyadic-defect-extraction,ass:observable-depletion,ass:moving-window-growth-control}.  Assume also that the depletion budget in \Cref{ass:observable-depletion} is bounded from below along the controlled branch, and that any moving-window subsequence whose visible observation strength tends to zero and whose reproduction residual is within the selected tolerance admits a cleaned NS-realizable cascade limit.  Then at least one of the following alternatives holds:
\begin{enumerate}[label=(\alph*)]
\item accumulated leakage is non-negligible along the branch;
\item the branch generates positive average untaxed critical supply which is not depleted by the available taxes;
\item moving-window observability is non-effective, in the sense that the finite-window constants grow too fast to close the argument;
\item there exists an NS-realizable, cleaned, scale-critical, combined-invisible, profitable and reproducible defect cascade;
\item in the one-component degeneration regime, the remaining obstruction is a vertical-duality or anti-phantom failure for the active residual quotient.
\end{enumerate}
Consequently, if leakage is negligible, visible supply is uniformly taxed or observably depleted, moving-window growth is controlled, and NS-realizable combined-invisible PRV cascades are excluded as in \Cref{ass:anti-phantom-exclusion}, then the branch must enter a CKN scale and the point is regular.
\end{theorem}

\begin{proof}
The argument is a conditional exclusion proof.  If the branch remains outside the CKN basin, then the reservoir badness $B_k$ stays bounded below along the selected subchain.  By \Cref{thm:finite-scale-survival-alternative}, cumulative untaxed supply or leakage must occur.  If leakage is non-negligible, alternative (a) holds.  If the remaining positive average supply is not taxed or depleted, then alternative (b) holds.  Otherwise the visible part of the supply produces observation strength in the pressure, flux, energy, or trace channels.  \Cref{ass:observable-depletion}, together with the lower bound on the depletion budget and the summability of the errors $e_n$, forbids a non-summable amount of visible observation along a controlled branch.  Hence, unless the moving-window constants lose effectiveness as in alternative (c), one may pass to a subsequence whose normalized packages approach the combined-invisible NS-realizable kernel.  The near-kernel extraction hypothesis in the theorem statement then produces a cleaned scale-critical cascade; the scale-transfer residual gives reproducibility, and the positive average supply inherited from the ledger gives profitability.  This is alternative (d).  In the one-component degeneration channel, failure of the corresponding vertical-duality closure is recorded separately as alternative (e).  If all alternatives (a)--(e) are excluded, the assumed non-CKN branch cannot persist, so a CKN scale occurs.
\end{proof}

This theorem gives the promised answer:
\[
\boxed{\begin{gathered}
\text{a potential singularity is either paid by untaxed supply/leakage}\\
\text{or becomes an NS-realizable invisible recurrence.}
\end{gathered}}
\]

\section{What remains open}

The reduction turns the broad question ``Can a singularity occur?'' into a shorter list of theorem targets.

\begin{problem}[Uniform taxation or observable depletion]
Show that every NS-realizable critical supply is uniformly taxed, or else produces pressure/flux/energy/trace activity which depletes a finite local budget, up to summable leakage.
\end{problem}

\begin{problem}[Finite-window NS-realizability exclusion]
In each controlled finite window, prove that the combined-invisible kernel contains no nonzero direction lying in the closure of NS-realizable cleaned defect packages.
\end{problem}

\begin{problem}[Vertical-duality and anti-phantom closure]
For one-component strict-shadow degeneration and for the more general PFET quotient, prove that active residuals generated by the Navier--Stokes equation lie in the trace-defect range with small trace cost.  Equivalently, prove the dual estimate
\[
|\pair{g}{y}|\le r\norm{A^*y}{}
\]
for all reduced active dual directions.
\end{problem}

\begin{problem}[Moving-window growth control]
Control the growth of finite-window observability, cleaning, pressure-transfer, and trace-cost constants along scale-selected windows strongly enough to prevent non-summable escape.
\end{problem}

\begin{problem}[Critical recurrence diagnostics]
Develop diagnostics for the three remaining sustaining mechanisms: vortex-stretching production, active pressure work, and coarse-grained interscale flux.  Show either that one of these must be visibly depleting, or that a sign-coherent hidden recurrence can be constructed.
\end{problem}

\section{The singularity-audit conclusion}

The main conclusion is not that singularities have been ruled out.  It is that a singularity, if it exists within this framework, must be far more organized than a generic energy or pressure concentration.  It must pay through untaxed critical supply or leakage; it must hide from active pressure, flux, energy, and trace observations; it must be realized by the exact structures of the Navier--Stokes equation; and it must reproduce across scales as a profitable recurrence.

Thus the problem has been transformed from
\[
\text{does a Navier--Stokes singularity exist?}
\]
into the more structured question
\[
\text{does there exist an NS-realizable, cleaned, combined-invisible, profitable critical recurrence?}
\]
A positive regularity route must prove that this final object is impossible.  A counterexample route must construct it while preserving pressure compatibility, Reynolds positivity, local energy inequality, and scale-to-scale reproduction.  Either direction would be highly informative.  The value of the audit framework is that it identifies exactly what remains to be proved or constructed.

\section{From singularity audit to local-to-clean transfer}\label{sec:audit-to-transfer}

The audit part of the paper identifies the last dangerous object: a cleaned,
NS-realizable, profitable, reproducible, combined-invisible defect cascade.  To
make this statement usable, one must compare two finite-window descriptions.
The localized description is the one produced directly from Navier--Stokes: it
contains cutoffs, harmonic pressure tails, projection tails, recentring choices,
and moving-window leakage.  The clean description is a finite model in which a
quotient anti-phantom gap can be tested without those local artifacts.

The rest of the paper proves the following transfer principle at the finite-window
level:
\[
\boxed{
\begin{gathered}
\text{clean quotient gap}
+\text{ quotient lifting}
+\text{ detection comparison}
+\text{ residual absorption} \\
\Longrightarrow
\text{localized non-gauge defects are quantitatively detected.}
\end{gathered}}
\]
Thus the first half of the paper answers \emph{what} a surviving obstruction must
look like, while the second half proves \emph{how} a clean anti-phantom gap would
be transported back to localized Navier--Stokes packages.  The proof is kept
finite-window throughout: no scale-uniform conclusion is claimed unless the
corresponding compatibility and growth controls are explicitly assumed.

\section{Local-to-clean finite-window preliminaries}\label{sec:lct-preliminaries-full}\label{sec:lct-preliminaries}

\subsection{Navier--Stokes scaling and local cylinders}

Let \(z_0=(x_0,t_0)\) and let
\[
    \Q_r(z_0)=B_r(x_0)\times(t_0-r^2,t_0).
\]
For a dyadic radius \(r_k=2^{-k}r_0\), define the rescaled fields
\[
    u^{(k)}(y,s)=r_k u(x_0+r_k y,t_0+r_k^2s),
    \qquad
    p^{(k)}(y,s)=r_k^2 p(x_0+r_k y,t_0+r_k^2s).
\]
The standard scale-critical quantities are invariant under this normalization.
This manuscript uses the normalization only as a finite-scale bookkeeping
device; no regularity conclusion is inferred from non-smallness alone.

\subsection{Suitable weak solutions}

Throughout, \((u,p)\) is assumed to be a suitable weak solution on the relevant
local cylinder.  Thus \(u\) is divergence-free, solves the Navier--Stokes
equations distributionally, belongs to the local energy class, and satisfies
the local energy inequality.  The pressure is understood modulo additive
functions of time, which are harmless after subtracting spatial means.

\subsection{Localized and clean defect spaces}

The localized defect space is denoted by \(\calD_\Lambda^{\loc}\).  A typical
element has the form
\[
    \mathfrak D_\Lambda
    =
    (U_k,P_k,R_k,\Phi_k,\Pi_k,L_k,s_k)_{k\in\Lambda_{\rm sc}},
\]
where \(U_k\) is the projected velocity coordinate, \(P_k\) is the projected
pressure coordinate, \(R_k\) is a Reynolds-stress or covariance coordinate,
\(\Phi_k,\Pi_k,L_k\) are flux, pressure-transport, and leakage coordinates,
and \(s_k\) is a slack vector recording finite-scale ledger inequalities.

The clean defect space is denoted by \(\calD_\Lambda^{\cl}\).  It is a finite
periodic model space in which cutoff, boundary, and harmonic-pressure artifacts
have either been removed or explicitly represented as clean gauge directions.

\subsection{Detection maps}

The localized detection map is
\[
    \calF_\Lambda^{\loc}(\mathfrak D)
    =
    \left(
    O_\Lambda^{\loc}\mathfrak D,\,
    \calE_\Lambda^{\loc}(\mathfrak D),\,
    \Rep_\Lambda^{\loc}(\mathfrak D),\,
    [\Prof_\Lambda^{\loc}(\mathfrak D)]_+
    \right).
\]
The four components measure observability, finite-window Navier--Stokes
compatibility, scale reproduction, and positive ledger profit.  The clean
detection map \(\calF_\Lambda^{\cl}\) is the corresponding map on
\(\calD_\Lambda^{\cl}\).
Throughout the manuscript the detection norms are the weighted sum norms
\[
    \|\calF_\Lambda^{\bullet}(d)\|_{\bullet}
    =
    \|O_\Lambda^{\bullet}d\|
    +
    C_E\|\calE_\Lambda^{\bullet}(d)\|
    +
    C_R\Rep_\Lambda^{\bullet}(d)
    +
    [\Prof_\Lambda^{\bullet}(d)]_+,
    \qquad
    \bullet\in\{\loc,\cl\}.
\]

\subsection{Roadmap for the transfer proof}\label{subsec:transfer-roadmap}

The transfer proof has three moving parts.  First, \Cref{sec:localized-packages}
constructs and estimates the localized residual budgets.  Second,
\Cref{sec:chart} introduces the local-to-clean chart and the quotient-lifting
condition.  Third, \Cref{sec:detection-error,sec:main-reduction} combine the
component comparisons with the clean gap.  The following table records where the
main entries of the residual ledger are introduced.
\begin{center}
\renewcommand{\arraystretch}{1.12}
\begin{tabular}{p{0.26\textwidth}p{0.36\textwidth}p{0.28\textwidth}}
\toprule
\textbf{Residual entry} & \textbf{What it measures} & \textbf{Main location} \\
\midrule
Pressure budget & harmonic tails, cutoff--Riesz commutators, active-source mismatch, projection, mean, and periodization residuals & \Cref{subsec:pressure-transfer}--\Cref{subsec:pressure-periodization} \\
Localization budget & energy, flux, pressure-transport, and momentum errors produced by the cutoff & \Cref{subsec:energy-flux-localization}--\Cref{subsec:localization-budget} \\
Truncation and nonlinear budget & finite-window truncation loss, cutoff-flux mismatch, and chosen nonlinear remainders & \Cref{subsec:truncation-leakage,subsec:nonlinear-cutoff} \\
Reproduction, gauge, and profit budget & chart-commutation drift, gauge mismatch, and scalar profit discrepancy & \Cref{subsec:reproduction-drift,sec:chart,subsec:profit-comparison} \\
\bottomrule
\end{tabular}
\end{center}

\section{Localized Defect Packages}\label{sec:localized-packages}

\subsection{Finite-window package realization}

\begin{definition}[Localized package datum]
A localized package datum consists of a finite dyadic scale window
\(\Lambda_{\rm sc}\), localized projection spaces for velocity, pressure, and
stress, cutoff functions, residual target spaces, cleaning spaces, observation
spaces, reproduction drift spaces, and profit functionals.
\end{definition}

\begin{assumption}[Localized defect-package realization]
\label[assumption]{ass:localized-package-realization}
For every suitable weak solution in the chosen local cylinder, the localized
package datum produces a finite-dimensional element
\[
    \mathfrak D_\Lambda^{\loc}
    \in
    \calD_\Lambda^{\loc}
\]
whose coordinates are defined from the rescaled fields, the localized pressure
splitting, and the finite-scale ledger variables.
\end{assumption}

\begin{remark}
This assumption is a construction hypothesis, not a smallness statement.  It
does not assert that the package is close to a clean periodic package.
\end{remark}

\subsection{Pressure splitting and leakage}

For each localized scale, the pressure is decomposed schematically as
\[
    p^{(k)}=p_k^{\rm act}+p_k^{\rm harm},
\]
where \(p_k^{\rm act}\) is generated by a localized Calderon--Zygmund pressure
solve and \(p_k^{\rm harm}\) is harmonic on the core.  The pressure error
budget contains, at a minimum, harmonic tails and cutoff--Riesz commutators:
\[
    \Err_{\rm prs}
    =
    \Err_{\rm harm}
    +
    \Err_{\rm comm}.
\]

\subsection{Pressure-transfer datum}\label{subsec:pressure-transfer}

\begin{definition}[Pressure-transfer datum]
\label{def:pressure-transfer-datum}
A pressure-transfer datum for a finite window \(\Lambda\) assigns to each
localized scale \(k\in\Lambda_{\rm sc}\) the following objects:
\begin{enumerate}[label=(\roman*),leftmargin=2em]
    \item a smooth pressure cutoff \(\eta_k\) and a core cylinder
    \(Q_{{\rm core},k}\) on which \(\eta_k\equiv1\);
    \item a normalization of pressure representatives, for instance spatial
    mean zero on the core at each time;
    \item a clean pressure source coordinate \(F_{k}^{\cl}\) and a clean
    pressure projection \(\mathsf P_{{\rm prs},k}^{\cl}\);
    \item a finite-dimensional harmonic gauge space
    \(\mathcal H_{\Lambda,k}^{M}\) on the core and a bounded projection
    \[
        \Pi_{{\rm harm},M}:H_{\Lambda,k}\to\mathcal H_{\Lambda,k}^{M};
    \]
    \item a pressure observation norm \(\|\cdot\|_{Y_{{\rm prs},k}}\), a
    pressure source norm \(\|\cdot\|_{X_{{\rm src},k}}\), a harmonic-tail
    seminorm \(\|\cdot\|_{H_{\Lambda,k}}\), and a commutator seminorm
    \(\|\cdot\|_{P_{\Lambda,k}}\).
\end{enumerate}
All of these choices are part of the datum.  In particular, the pressure
observation norm and the harmonic projection are not canonical, and no
smallness or scale-uniform estimate is assumed in the definition.
\end{definition}

\begin{lemma}[Pressure mismatch decomposition]
\label{lem:pressure-mismatch-decomposition}
Fix a pressure-transfer datum.  Let
\[
    f_{k,ij}=u_i^{(k)}u_j^{(k)},
    \qquad
    F_{k,ij}^{\cl}=U_{k,i}U_{k,j}+R_{k,ij},
\]
and suppose that the normalized localized pressure satisfies on the core
\[
    p^{(k)}
    =
    R_iR_j(\eta_k f_{k,ij})
    +
    p_k^{\rm harm},
\]
where \(p_k^{\rm harm}\) is harmonic on \(Q_{{\rm core},k}\).  Define the clean
active pressure coordinate by
\[
    p_k^{\cl}
    :=
    \mathsf P_{{\rm prs},k}^{\cl}
    R_iR_j(F_{k,ij}^{\cl}).
\]
Then on the core,
\[
\begin{aligned}
    p^{(k)}-p_k^{\cl}
    &=
    p_k^{\rm harm}
    +
    \Bigl[
    R_iR_j(\eta_k f_{k,ij})
    -
    R_iR_j(f_{k,ij})
    \Bigr]  \\
    &\quad+
    R_iR_j(f_{k,ij}-F_{k,ij}^{\cl})
    +
    \Bigl[
    R_iR_j(F_{k,ij}^{\cl})
    -
    \mathsf P_{{\rm prs},k}^{\cl}R_iR_j(F_{k,ij}^{\cl})
    \Bigr].
\end{aligned}
\]
The four terms are, respectively, the harmonic pressure tail, the cutoff--Riesz
commutator contribution, the active source residual, and the clean active
projection residual.
\end{lemma}

\begin{proof}
Starting from the pressure splitting, add and subtract
\[
    R_iR_j(f_{k,ij})
    \qquad\text{and}\qquad
    R_iR_j(F_{k,ij}^{\cl})
\]
between \(R_iR_j(\eta_k f_{k,ij})\) and
\(\mathsf P_{{\rm prs},k}^{\cl}R_iR_j(F_{k,ij}^{\cl})\).  This gives the
displayed identity.  The only PDE input used here is the stated active/harmonic
pressure decomposition on the core; the rest is algebraic bookkeeping.
\end{proof}

\begin{convention}[Default harmonic pressure gauge]
\label{conv:default-harmonic-pressure-gauge}
The default harmonic pressure gauge is the \(L^2(B_{1/2})\)-orthogonal
projection onto the finite-dimensional space of harmonic polynomials of degree
at most \(M\), restricted to \(B_{1/2}\).  More precisely, set
\[
    \mathcal H_M(B_{1/2})
    :=
    \{q|_{B_{1/2}}:\ q \text{ is a harmonic polynomial on }\R^3,
    \ \deg q\le M\}.
\]
The projected harmonic modes in \(\mathcal H_M(B_{1/2})\) are the default
finite harmonic pressure gauge directions.  The residual
\((I-\Pi_{{\rm harm},M})p^{\rm harm}\) is not discarded; it remains the
harmonic pressure-tail error.
\end{convention}

\begin{lemma}[Finite-dimensional harmonic gauge space]
\label{lem:finite-dimensional-harmonic-gauge-space}
For every integer \(M\ge0\), the space \(\mathcal H_M(B_{1/2})\) in
\Cref{conv:default-harmonic-pressure-gauge} is a finite-dimensional subspace
of \(L^2(B_{1/2})\).  In particular, it is closed in \(L^2(B_{1/2})\).
\end{lemma}

\begin{proof}
Let \(\mathcal P_M(\R^3)\) denote the vector space of polynomials in three
variables of degree at most \(M\).  This space is finite-dimensional.  The
harmonic polynomials of degree at most \(M\) form the kernel of the linear map
\[
    \Delta:\mathcal P_M(\R^3)\to \mathcal P_{M-2}(\R^3),
\]
with the evident interpretation when \(M<2\).  Hence they form a
finite-dimensional subspace of \(\mathcal P_M(\R^3)\).  Restricting such
polynomials to \(B_{1/2}\) gives a finite-dimensional subspace of
\(L^2(B_{1/2})\).  Every finite-dimensional subspace of a Hilbert space is
closed.
\end{proof}

\begin{definition}[Harmonic pressure projection]
\label{def:harmonic-pressure-projection}
Let
\[
    \Pi_{{\rm harm},M}:L^2(B_{1/2})\to\mathcal H_M(B_{1/2})
\]
be the \(L^2(B_{1/2})\)-orthogonal projection onto
\(\mathcal H_M(B_{1/2})\).
\end{definition}

\begin{lemma}[Boundedness and best approximation of the harmonic projection]
\label{lem:harmonic-projection-best-approximation}
The map \(\Pi_{{\rm harm},M}\) is a bounded linear projection on
\(L^2(B_{1/2})\) with
\[
    \|\Pi_{{\rm harm},M}h\|_{L^2(B_{1/2})}
    \le
    \|h\|_{L^2(B_{1/2})}.
\]
Moreover, for every \(h\in L^2(B_{1/2})\),
\[
    \|(I-\Pi_{{\rm harm},M})h\|_{L^2(B_{1/2})}
    =
    \inf_{q\in\mathcal H_M(B_{1/2})}
    \|h-q\|_{L^2(B_{1/2})}.
\]
\end{lemma}

\begin{proof}
Since \(\mathcal H_M(B_{1/2})\) is a closed subspace of the Hilbert space
\(L^2(B_{1/2})\), the Hilbert projection theorem gives a unique orthogonal
projection \(\Pi_{{\rm harm},M}\).  Orthogonal projections are linear,
idempotent, and have operator norm at most one, which gives the displayed
boundedness estimate.

For the best-approximation identity, decompose
\[
    h
    =
    \Pi_{{\rm harm},M}h
    +
    (I-\Pi_{{\rm harm},M})h,
\]
where the residual is orthogonal to \(\mathcal H_M(B_{1/2})\).  For any
\(q\in\mathcal H_M(B_{1/2})\), the Pythagorean identity gives
\[
    \|h-q\|_{L^2(B_{1/2})}^2
    =
    \|(I-\Pi_{{\rm harm},M})h\|_{L^2(B_{1/2})}^2
    +
    \|\Pi_{{\rm harm},M}h-q\|_{L^2(B_{1/2})}^2.
\]
The infimum is attained at \(q=\Pi_{{\rm harm},M}h\), proving the identity.
\end{proof}

\begin{lemma}[Default harmonic-gauge pressure splitting]
\label{lem:default-harmonic-gauge-pressure-splitting}
Let \(p^{\rm harm}\in L^2(B_{1/2})\) be a harmonic pressure tail on the
normalized core.  Under \Cref{conv:default-harmonic-pressure-gauge},
\[
    p^{\rm harm}
    =
    \Pi_{{\rm harm},M}p^{\rm harm}
    +
    (I-\Pi_{{\rm harm},M})p^{\rm harm}.
\]
The projected component is the default finite harmonic gauge part, and the
residual tail is measured by
\[
    \Err_{\rm harm-tail}^{(M)}
    :=
    \|(I-\Pi_{{\rm harm},M})p^{\rm harm}\|_{L^2(B_{1/2})}.
\]
It satisfies the best-approximation formula
\[
    \Err_{\rm harm-tail}^{(M)}
    =
    \inf_{q\in\mathcal H_M(B_{1/2})}
    \|p^{\rm harm}-q\|_{L^2(B_{1/2})}.
\]
\end{lemma}

\begin{proof}
The splitting is the identity
\[
    I=\Pi_{{\rm harm},M}+(I-\Pi_{{\rm harm},M})
\]
on \(L^2(B_{1/2})\).  The projected component belongs to
\(\mathcal H_M(B_{1/2})\) by definition of the projection, and
\Cref{conv:default-harmonic-pressure-gauge} declares precisely these modes to
be the default finite harmonic gauge directions.  The residual is retained as
the harmonic pressure-tail error.  The best-approximation formula follows from
\Cref{lem:harmonic-projection-best-approximation} with
\(h=p^{\rm harm}\).
\end{proof}

\begin{remark}[Status of harmonic pressure cleaning]
\Cref{lem:default-harmonic-gauge-pressure-splitting} proves only the
finite-dimensional \(L^2\)-orthogonal splitting and its best-approximation
property.  It does not prove that
\(\Err_{\rm harm-tail}^{(M)}\) is small as \(M\) is fixed, does not prove any
scale-uniform decay, and does not show that the chosen harmonic gauge is
compatible with every pressure observation channel.  If a harmonic pressure
direction represents a physical observable rather than a gauge direction, the
observation map must keep track of that direction.
\end{remark}

\begin{convention}[Default pressure observation norm]
\label{conv:default-pressure-observation-norm}
For the normalized pressure module on a finite time interval \(I\), the default
pressure observation norm is
\[
    \|q\|_{Y_{\rm prs}}
    :=
    \|q\|_{L^{3/2}(I;L^{3/2}(B_{1/2}))}.
\]
The harmonic residual norm is
\[
    \|r\|_{H_{\rm harm}}
    :=
    \|r\|_{L^2(I;L^2(B_{1/2}))}.
\]
\end{convention}

\begin{lemma}[Finite-measure pressure observation embedding]
\label{lem:finite-measure-pressure-observation-embedding}
Let \(I\subset\R\) be a finite time interval.  For every
\[
    r\in L^2(I;L^2(B_{1/2})),
\]
one has
\[
    \|r\|_{L^{3/2}(I;L^{3/2}(B_{1/2}))}
    \le
    |I|^{1/6}|B_{1/2}|^{1/6}
    \|r\|_{L^2(I;L^2(B_{1/2}))}.
\]
\end{lemma}

\begin{proof}
For almost every \(t\in I\), Holder's inequality on the finite-measure set
\(B_{1/2}\) gives
\[
    \|r(t,\cdot)\|_{L^{3/2}(B_{1/2})}
    \le
    |B_{1/2}|^{1/6}
    \|r(t,\cdot)\|_{L^2(B_{1/2})}.
\]
Taking the \(L^{3/2}(I)\) norm and applying Holder's inequality on the
finite-measure set \(I\) gives
\[
\begin{aligned}
    \|r\|_{L^{3/2}(I;L^{3/2}(B_{1/2}))}
    &\le
    |B_{1/2}|^{1/6}
    \|\,\|r(t,\cdot)\|_{L^2(B_{1/2})}\,\|_{L^{3/2}(I)} \\
    &\le
    |B_{1/2}|^{1/6}|I|^{1/6}
    \|\,\|r(t,\cdot)\|_{L^2(B_{1/2})}\,\|_{L^2(I)}.
\end{aligned}
\]
This is the claimed estimate.
\end{proof}

\begin{lemma}[Pressure observation control of the harmonic residual]
\label{lem:harmonic-residual-pressure-observation-control}
Let \(p^{\rm harm}\in L^2(I;L^2(B_{1/2}))\), and define the harmonic residual
\[
    r_M
    :=
    (I-\Pi_{{\rm harm},M})p^{\rm harm},
\]
where \(\Pi_{{\rm harm},M}\) acts on the spatial variable for almost every
\(t\).  Then
\[
    \|r_M\|_{Y_{\rm prs}}
    \le
    |I|^{1/6}|B_{1/2}|^{1/6}
    \|r_M\|_{H_{\rm harm}}.
\]
Consequently the \(L^2\)-harmonic tail residual contributes to the pressure
error budget in the default pressure observation norm through the explicit
quantity
\[
    \Err_{{\rm harm}\to{\rm prs}}^{(M)}
    :=
    |I|^{1/6}|B_{1/2}|^{1/6}
    \|(I-\Pi_{{\rm harm},M})p^{\rm harm}\|_{L^2(I;L^2(B_{1/2}))}.
\]
\end{lemma}

\begin{proof}
By \Cref{lem:finite-measure-pressure-observation-embedding}, applied to
\(r=r_M\),
\[
    \|r_M\|_{L^{3/2}(I;L^{3/2}(B_{1/2}))}
    \le
    |I|^{1/6}|B_{1/2}|^{1/6}
    \|r_M\|_{L^2(I;L^2(B_{1/2}))}.
\]
Using \Cref{conv:default-pressure-observation-norm} identifies the left-hand
side with \(\|r_M\|_{Y_{\rm prs}}\) and the right-hand norm with
\(\|r_M\|_{H_{\rm harm}}\).  The displayed pressure-error contribution is
therefore just the measured upper bound for the harmonic residual in the
default pressure observation norm.
\end{proof}

\begin{remark}[Status of the observation embedding]
\Cref{lem:harmonic-residual-pressure-observation-control} uses only finite
measure.  It does not show that the harmonic residual is small, does not prove
decay as \(M\to\infty\), and does not provide a scale-uniform constant under
dyadic rescaling.
\end{remark}

\begin{lemma}[Fixed-geometry harmonic polynomial approximation]
\label{lem:fixed-geometry-harmonic-polynomial-approximation}
Let \(M\ge0\), and let \(h\in L^2(B_{3/4})\) be harmonic in \(B_{3/4}\).  Then
\[
    \|(I-\Pi_{{\rm harm},M})h\|_{L^2(B_{1/2})}
    \le
    \left(\frac23\right)^{M+1}
    \|h\|_{L^2(B_{3/4})}.
\]
Consequently, if \(p^{\rm harm}\in L^2(I;L^2(B_{3/4}))\) and
\(p^{\rm harm}(t,\cdot)\) is harmonic in \(B_{3/4}\) for almost every \(t\in I\),
then
\[
    \|(I-\Pi_{{\rm harm},M})p^{\rm harm}\|_{L^2(I;L^2(B_{1/2}))}
    \le
    \left(\frac23\right)^{M+1}
    \|p^{\rm harm}\|_{L^2(I;L^2(B_{3/4}))}.
\]
\end{lemma}

\begin{proof}
We first prove the spatial estimate.  By the standard spherical-harmonic
expansion for \(L^2\)-harmonic functions on a ball,
\[
    h(x)=\sum_{\ell=0}^{\infty} h_\ell(x)
\]
in \(L^2(B_\rho)\) for every \(\rho<3/4\), where each \(h_\ell\) is a
homogeneous harmonic polynomial of degree \(\ell\).  The summands are
orthogonal in \(L^2(B_\rho)\), and homogeneity gives, for
\(0<\rho\le 3/4\),
\[
    \|h_\ell\|_{L^2(B_\rho)}^2
    =
    \left(\frac{\rho}{3/4}\right)^{2\ell+3}
    \|h_\ell\|_{L^2(B_{3/4})}^2 .
\]
Set
\[
    T_M h:=\sum_{\ell=0}^{M} h_\ell .
\]
Then \(T_M h\in\mathcal H_M(B_{1/2})\).  Since \(\Pi_{{\rm harm},M}\) is the
\(L^2(B_{1/2})\)-orthogonal projection onto \(\mathcal H_M(B_{1/2})\),
\[
    \|(I-\Pi_{{\rm harm},M})h\|_{L^2(B_{1/2})}
    \le
    \|h-T_Mh\|_{L^2(B_{1/2})}.
\]
Using orthogonality and the preceding scaling identity with \(\rho=1/2\),
\[
\begin{aligned}
    \|h-T_Mh\|_{L^2(B_{1/2})}^2
    &=
    \sum_{\ell>M}\|h_\ell\|_{L^2(B_{1/2})}^2  \\
    &=
    \sum_{\ell>M}
    \left(\frac23\right)^{2\ell+3}
    \|h_\ell\|_{L^2(B_{3/4})}^2 \\
    &\le
    \left(\frac23\right)^{2M+2}
    \sum_{\ell>M}\|h_\ell\|_{L^2(B_{3/4})}^2 \\
    &\le
    \left(\frac23\right)^{2M+2}
    \|h\|_{L^2(B_{3/4})}^2 .
\end{aligned}
\]
Taking square roots proves the spatial estimate.  Applying this estimate for
almost every \(t\in I\) to \(h=p^{\rm harm}(t,\cdot)\) and integrating in time
gives the stated \(L^2(I;L^2)\) bound.
\end{proof}

\begin{corollary}[Pressure observation bound for the harmonic approximation tail]
\label{cor:pressure-observation-harmonic-approximation-tail}
Under the hypotheses of
\Cref{lem:fixed-geometry-harmonic-polynomial-approximation},
\[
    \|(I-\Pi_{{\rm harm},M})p^{\rm harm}\|_{Y_{\rm prs}}
    \le
    |I|^{1/6}|B_{1/2}|^{1/6}
    \left(\frac23\right)^{M+1}
    \|p^{\rm harm}\|_{L^2(I;L^2(B_{3/4}))}.
\]
Thus, in the fixed normalized geometry, the harmonic-tail contribution to
\(\Err_{\rm prs}\) is controlled by the interior \(L^2\)-size of the harmonic
pressure tail on \(B_{3/4}\), with the explicit factor shown above.
\end{corollary}

\begin{proof}
Combine
\Cref{lem:harmonic-residual-pressure-observation-control} with the
\(L^2(I;L^2)\) estimate in
\Cref{lem:fixed-geometry-harmonic-polynomial-approximation}.
\end{proof}

\begin{remark}[Status of the harmonic approximation estimate]
\Cref{lem:fixed-geometry-harmonic-polynomial-approximation} and
\Cref{cor:pressure-observation-harmonic-approximation-tail} are fixed-geometry
statements on the nested balls \(B_{1/2}\Subset B_{3/4}\).  The estimate uses
interior harmonicity on \(B_{3/4}\) and the default
\(L^2(B_{1/2})\)-projection.  It is not a scale-uniform pressure estimate and
does not control any non-harmonic pressure component, cutoff leakage, or
active-source projection residual.
\end{remark}

\begin{lemma}[Cutoff--Riesz commutator harmonicity]
\label{lem:cutoff-riesz-commutator-harmonicity}
Let \(\eta\in C^\infty\) satisfy \(\eta\equiv1\) on a neighborhood of a spatial
core \(B_{\rm core}\).  Let \(f_{ij}\in L^{3/2}_{\rm loc}\), and assume that
\((1-\eta)f_{ij}\) is supported away from \(B_{\rm core}\).  Define
\[
    \mathcal C_\eta(f)
    :=
    R_iR_j(\eta f_{ij})
    -
    \eta R_iR_j(f_{ij}).
\]
Then \(\mathcal C_\eta(f)\) is spatially harmonic on \(B_{\rm core}\) in the
distributional sense.
\end{lemma}

\begin{proof}
On \(B_{\rm core}\), \(\eta=1\), hence
\[
    \mathcal C_\eta(f)
    =
    R_iR_j(\eta f_{ij})
    -
    R_iR_j(f_{ij})
    =
    -R_iR_j((1-\eta)f_{ij}).
\]
The distribution \((1-\eta)f_{ij}\) is supported away from \(B_{\rm core}\).
Therefore, for every test function \(\varphi\in C_c^\infty(B_{\rm core})\),
\[
    \langle -\Delta \mathcal C_\eta(f),\varphi\rangle
    =
    -\langle \partial_i\partial_j((1-\eta)f_{ij}),\varphi\rangle
    =
    0,
\]
because the supports of \((1-\eta)f_{ij}\) and \(\varphi\) are disjoint.  Thus
\(\Delta \mathcal C_\eta(f)=0\) on \(B_{\rm core}\) distributionally.
\end{proof}

\begin{remark}[What the commutator lemma does not prove]
The lemma proves harmonicity on the separated core, not smallness in any
pressure norm.  Norm bounds for \(\mathcal C_\eta(f)\) require additional
information about the pressure observation norm, the cutoff separation, and the
source norm.  If the clean pressure solve is periodized rather than whole-space,
an additional periodization residual must be included in the pressure budget.
\end{remark}

\begin{definition}[Normalized local pressure geometry]
\label{def:normalized-local-pressure-geometry}
Let
\[
    B_{1/2}\subset B_{3/4}\subset B_1\subset\R^3,
    \qquad
    A_{3/4,1}:=B_1\setminus \overline{B}_{3/4}.
\]
A normalized pressure cutoff is a function
\(\eta\in C_c^\infty(B_1)\) such that \(0\le\eta\le1\) and
\(\eta\equiv1\) on \(B_{3/4}\).  For a tensor source \(f=(f_{ij})\), define
\[
    \mathcal C_\eta(f)
    :=
    R_iR_j(\eta f_{ij})-\eta R_iR_j(f_{ij}),
\]
with summation over \(i,j\).
\end{definition}

\begin{lemma}[Normalized cutoff--Riesz commutator estimate]
\label{lem:normalized-cutoff-riesz-commutator-estimate}
Let \(I\subset\R\) be a finite time interval, let \(\eta\) be a normalized
pressure cutoff, and let \(f=(f_{ij})\) be such that the commutator in
\Cref{def:normalized-local-pressure-geometry} is defined distributionally.
Set
\[
    g_{ij}:=(1-\eta)f_{ij}.
\]
Assume that \(g_{ij}\) is supported in \(A_{3/4,1}\) for almost every
\(t\in I\), and that
\[
    g\in L^{3/2}\bigl(I;L^{3/2}(A_{3/4,1})\bigr).
\]
Then \(\mathcal C_\eta(f)\) is spatially harmonic on \(B_{1/2}\) in the
distributional sense for almost every \(t\in I\), and
\[
\begin{aligned}
    \|\mathcal C_\eta(f)\|_{L^{3/2}(I;L^{3/2}(B_{1/2}))}
    &\le
    C_{\rm nlp}
    \sum_{i,j=1}^3
    \|g_{ij}\|_{L^{3/2}(I;L^{3/2}(A_{3/4,1}))}  \\
    &=
    C_{\rm nlp}
    \sum_{i,j=1}^3
    \|(1-\eta)f_{ij}\|_{L^{3/2}(I;L^{3/2}(A_{3/4,1}))}.
\end{aligned}
\]
Here \(C_{\rm nlp}<\infty\) depends only on the fixed normalized geometry and
the Calderon--Zygmund kernel, not on \(f\).  The estimate is fixed-scale; no
smallness or scale-uniformity is asserted.
\end{lemma}

\begin{proof}
On \(B_{1/2}\), we have \(\eta=1\).  Hence, in distributions on
\(B_{1/2}\),
\[
    \mathcal C_\eta(f)
    =
    R_iR_j(\eta f_{ij})-R_iR_j(f_{ij})
    =
    -R_iR_j(g_{ij}).
\]
Since \(g_{ij}\) is supported in \(A_{3/4,1}\), its spatial support is
separated from \(B_{1/2}\).  Therefore the principal-value singularity of the
Riesz kernel is never encountered on \(B_{1/2}\).  Writing \(K_{ij}\) for the
whole-space kernel of \(R_iR_j\) away from the origin, we have, for
\(x\in B_{1/2}\),
\[
    \mathcal C_\eta(f)(x,t)
    =
    -\sum_{i,j=1}^3\int_{A_{3/4,1}}K_{ij}(x-y)g_{ij}(y,t)\,dy .
\]
The separation gives
\[
    d_0:=\operatorname{dist}(B_{1/2},A_{3/4,1})=\frac14,
\]
and hence
\[
    M_K:=\max_{1\le i,j\le3}
    \sup_{\substack{x\in B_{1/2}\\ y\in A_{3/4,1}}}
    |K_{ij}(x-y)|<\infty .
\]
It follows that, for almost every \(t\in I\),
\[
    |\mathcal C_\eta(f)(x,t)|
    \le
    M_K\sum_{i,j=1}^3
    \|g_{ij}(\cdot,t)\|_{L^1(A_{3/4,1})}.
\]
Taking the \(L^{3/2}(B_{1/2})\) norm and using Holder's inequality on the
annulus,
\[
\begin{aligned}
    \|\mathcal C_\eta(f)(\cdot,t)\|_{L^{3/2}(B_{1/2})}
    &\le
    |B_{1/2}|^{2/3}M_K
    \sum_{i,j=1}^3
    \|g_{ij}(\cdot,t)\|_{L^1(A_{3/4,1})} \\
    &\le
    |B_{1/2}|^{2/3}|A_{3/4,1}|^{1/3}M_K
    \sum_{i,j=1}^3
    \|g_{ij}(\cdot,t)\|_{L^{3/2}(A_{3/4,1})}.
\end{aligned}
\]
Taking the \(L^{3/2}(I)\) norm gives the stated bound with
\[
    C_{\rm nlp}
    =
    |B_{1/2}|^{2/3}|A_{3/4,1}|^{1/3}M_K .
\]

It remains to record harmonicity.  Let
\(\varphi\in C_c^\infty(B_{1/2})\).  Since
\(\operatorname{supp}g_{ij}(\cdot,t)\cap\operatorname{supp}\varphi=\emptyset\),
\[
    \langle -\Delta\mathcal C_\eta(f)(\cdot,t),\varphi\rangle
    =
    -\langle \partial_i\partial_j g_{ij}(\cdot,t),\varphi\rangle
    =
    -\langle g_{ij}(\cdot,t),\partial_i\partial_j\varphi\rangle
    =
    0 .
\]
Thus \(\Delta\mathcal C_\eta(f)=0\) on \(B_{1/2}\) distributionally for
almost every \(t\).  The proof uses only the fixed normalized separation and
the annular \(L^{3/2}\) source norm.
\end{proof}

\begin{lemma}[Normalized active pressure source residual]
\label{lem:normalized-active-pressure-source-residual}
Work in the normalized geometry of
\Cref{def:normalized-local-pressure-geometry}, and let \(I\subset\R\) be a
finite time interval.  Let \(f=(f_{ij})\) and \(F^{\cl}=(F_{ij}^{\cl})\) be
tensor sources and set
\[
    G_{ij}:=f_{ij}-F_{ij}^{\cl}.
\]
Assume that the zero extension of \(G\) is supported in \(B_1\) and belongs to
\[
    L^{3/2}\bigl(I;L^{3/2}(\R^3)\bigr).
\]
Then
\[
    \|R_iR_jG_{ij}\|_{L^{3/2}(I;L^{3/2}(B_{1/2}))}
    \le
    C_{\rm CZ}
    \sum_{i,j=1}^3
    \|G_{ij}\|_{L^{3/2}(I;L^{3/2}(B_1))}.
\]
If the clean active pressure \(R_iR_jF_{ij}^{\cl}\) and the clean pressure
projection \(\mathsf P_{{\rm prs},N}^{\cl}\) are fixed, define
\[
    \Err_{\rm active}
    :=
    C_{\rm CZ}
    \sum_{i,j=1}^3
    \|f_{ij}-F_{ij}^{\cl}\|_{L^{3/2}(I;L^{3/2}(B_1))}
\]
and
\[
    \Err_{\rm proj}
    :=
    \Err_{{\rm proj},N}^{\cl}\bigl(R_iR_j(F_{ij}^{\cl})\bigr).
\]
Then the active-source and clean-projection mismatch is bounded by
\[
\begin{aligned}
    &\|R_iR_j(f_{ij}-F_{ij}^{\cl})\|_{L^{3/2}(I;L^{3/2}(B_{1/2}))} \\
    &\quad+
    \bigl\|
    R_iR_j(F_{ij}^{\cl})
    -
    \mathsf P_{{\rm prs},N}^{\cl}R_iR_j(F_{ij}^{\cl})
    \bigr\|_{L^{3/2}(I;L^{3/2}(B_{1/2}))}
    \le
    \Err_{\rm active}
    +
    \Err_{\rm proj}.
\end{aligned}
\]
\end{lemma}

\begin{proof}
For \(1<p<\infty\), the Calderon--Zygmund operator \(R_iR_j\) is bounded on
\(L^p(\R^3)\).  Applying this with \(p=3/2\) at almost every time gives
\[
    \|R_iR_jG_{ij}(\cdot,t)\|_{L^{3/2}(\R^3)}
    \le
    C_{\rm CZ}
    \sum_{i,j=1}^3
    \|G_{ij}(\cdot,t)\|_{L^{3/2}(\R^3)}.
\]
Restricting the left-hand side to \(B_{1/2}\), using the support of \(G\) in
\(B_1\), and then taking the \(L^{3/2}(I)\) norm proves the first estimate.
The second displayed inequality follows by adding this estimate to the
definition of \(\Err_{\rm proj}\).  The projection residual is not estimated
away; it remains an explicit pressure error.
\end{proof}

\begin{remark}[Status of the active source estimate]
\Cref{lem:normalized-active-pressure-source-residual} proves a fixed-scale
Calderon--Zygmund estimate for the active pressure source residual in the
normalized model.  It does not prove that the clean projection residual is
small, and it does not give scale-uniform constants for an arbitrary dyadic
window.
\end{remark}

\subsection{Clean pressure projection residual}

\begin{convention}[Clean pressure projection model]
\label{conv:clean-pressure-projection-model}
On a finite time interval \(I\), let
\[
    Y_{\rm prs}=L^{3/2}(I;L^{3/2}(B_{1/2}))
\]
be the default pressure observation space from
\Cref{conv:default-pressure-observation-norm}.  For a finite integer
\(N\ge0\), fix a finite-dimensional clean pressure observation space
\[
    V_{{\rm prs},N}^{\cl}\subset Y_{\rm prs}
\]
and a linear projection
\[
    \mathsf P_{{\rm prs},N}^{\cl}:Y_{\rm prs}\to V_{{\rm prs},N}^{\cl},
    \qquad
    (\mathsf P_{{\rm prs},N}^{\cl})^2=\mathsf P_{{\rm prs},N}^{\cl}.
\]
The projection is part of the clean pressure model.  It removes only the
finite-dimensional clean pressure modes in \(V_{{\rm prs},N}^{\cl}\); the
remaining component is retained as an observable projection residual.
\end{convention}

\begin{lemma}[Bounded clean pressure projection]
\label{lem:bounded-clean-pressure-projection}
For every finite-dimensional subspace
\(V_{{\rm prs},N}^{\cl}\subset Y_{\rm prs}\), there exists a bounded linear
projection
\[
    \mathsf P_{{\rm prs},N}^{\cl}:Y_{\rm prs}\to V_{{\rm prs},N}^{\cl}.
\]
In particular, after fixing such a projection, there is a finite constant
\(C_{{\rm proj},N}<\infty\) such that
\[
    \|\mathsf P_{{\rm prs},N}^{\cl}q\|_{Y_{\rm prs}}
    \le
    C_{{\rm proj},N}\|q\|_{Y_{\rm prs}}
    \qquad
    \text{for all }q\in Y_{\rm prs}.
\]
\end{lemma}

\begin{proof}
Because \(V_{{\rm prs},N}^{\cl}\) is finite-dimensional, choose a basis
\(\psi_1,\ldots,\psi_d\) of \(V_{{\rm prs},N}^{\cl}\).  Let
\(\ell_a^V\), \(1\le a\le d\), be the coordinate functionals on
\(V_{{\rm prs},N}^{\cl}\), so that
\[
    v=\sum_{a=1}^{d}\ell_a^V(v)\psi_a
    \qquad
    \text{for }v\in V_{{\rm prs},N}^{\cl}.
\]
Each \(\ell_a^V\) is continuous on the finite-dimensional normed space
\(V_{{\rm prs},N}^{\cl}\).  By Hahn--Banach, extend \(\ell_a^V\) to a bounded
linear functional \(\ell_a\) on \(Y_{\rm prs}\).  Define
\[
    \mathsf P_{{\rm prs},N}^{\cl}q
    :=
    \sum_{a=1}^{d}\ell_a(q)\psi_a .
\]
Then \(\mathsf P_{{\rm prs},N}^{\cl}q\in V_{{\rm prs},N}^{\cl}\), and if
\(q\in V_{{\rm prs},N}^{\cl}\) then the construction gives
\(\mathsf P_{{\rm prs},N}^{\cl}q=q\).  Hence
\((\mathsf P_{{\rm prs},N}^{\cl})^2=\mathsf P_{{\rm prs},N}^{\cl}\).  Finally,
\[
    \|\mathsf P_{{\rm prs},N}^{\cl}q\|_{Y_{\rm prs}}
    \le
    \sum_{a=1}^{d}\|\ell_a\|_{(Y_{\rm prs})^*}\|\psi_a\|_{Y_{\rm prs}}
    \|q\|_{Y_{\rm prs}},
\]
which proves boundedness with
\[
    C_{{\rm proj},N}
    :=
    \sum_{a=1}^{d}\|\ell_a\|_{(Y_{\rm prs})^*}\|\psi_a\|_{Y_{\rm prs}}.
\]
\end{proof}

\begin{definition}[Clean pressure projection residual]
\label{def:clean-pressure-projection-residual}
Given the clean pressure projection model in
\Cref{conv:clean-pressure-projection-model}, define
\[
    \Err_{{\rm proj},N}^{\cl}(q)
    :=
    \|q-\mathsf P_{{\rm prs},N}^{\cl}q\|_{Y_{\rm prs}},
    \qquad q\in Y_{\rm prs}.
\]
\end{definition}

\begin{lemma}[Well-defined clean projection residual]
\label{lem:clean-projection-residual-well-defined}
For every \(q\in Y_{\rm prs}\), the quantity
\(\Err_{{\rm proj},N}^{\cl}(q)\) is a finite, well-defined pressure error.
Moreover,
\[
    \Err_{{\rm proj},N}^{\cl}(q)
    \le
    (1+C_{{\rm proj},N})\|q\|_{Y_{\rm prs}}.
\]
\end{lemma}

\begin{proof}
Since \(q\in Y_{\rm prs}\) and
\(\mathsf P_{{\rm prs},N}^{\cl}:Y_{\rm prs}\to Y_{\rm prs}\) is bounded,
\(q-\mathsf P_{{\rm prs},N}^{\cl}q\in Y_{\rm prs}\).  Therefore the norm in
\Cref{def:clean-pressure-projection-residual} is finite.  The displayed bound
follows from the triangle inequality and
\Cref{lem:bounded-clean-pressure-projection}.
\end{proof}

\begin{lemma}[Hilbert clean pressure best approximation]
\label{lem:hilbert-clean-pressure-best-approximation}
Let
\[
    Y_{\rm prs}^{(2)}:=L^2(I;L^2(B_{1/2})).
\]
Suppose \(V_{{\rm prs},N}^{\cl}\subset Y_{\rm prs}^{(2)}\) is finite-dimensional
and \(\mathsf P_{{\rm prs},N}^{\cl,2}\) is the
\(Y_{\rm prs}^{(2)}\)-orthogonal projection onto
\(V_{{\rm prs},N}^{\cl}\).  Then for every \(q\in Y_{\rm prs}^{(2)}\),
\[
    \|q-\mathsf P_{{\rm prs},N}^{\cl,2}q\|_{Y_{\rm prs}^{(2)}}
    =
    \inf_{v\in V_{{\rm prs},N}^{\cl}}
    \|q-v\|_{Y_{\rm prs}^{(2)}}.
\]
\end{lemma}

\begin{proof}
The Hilbert-space projection theorem gives the orthogonal decomposition
\[
    q
    =
    \mathsf P_{{\rm prs},N}^{\cl,2}q
    +
    (I-\mathsf P_{{\rm prs},N}^{\cl,2})q,
\]
where the residual is orthogonal to \(V_{{\rm prs},N}^{\cl}\).  For
\(v\in V_{{\rm prs},N}^{\cl}\), Pythagoras gives
\[
    \|q-v\|_{Y_{\rm prs}^{(2)}}^2
    =
    \|(I-\mathsf P_{{\rm prs},N}^{\cl,2})q\|_{Y_{\rm prs}^{(2)}}^2
    +
    \|\mathsf P_{{\rm prs},N}^{\cl,2}q-v\|_{Y_{\rm prs}^{(2)}}^2 .
\]
The infimum is attained at
\(v=\mathsf P_{{\rm prs},N}^{\cl,2}q\), proving the identity.
\end{proof}

\begin{remark}[Banach status of the clean pressure projection residual]
\label{rem:banach-clean-pressure-projection-status}
The default pressure observation norm is
\(Y_{\rm prs}=L^{3/2}(I;L^{3/2}(B_{1/2}))\), which is a Banach norm rather than
the Hilbert norm used in
\Cref{lem:hilbert-clean-pressure-best-approximation}.  In the default
\(Y_{\rm prs}\) setting, the manuscript uses only boundedness of
\(\mathsf P_{{\rm prs},N}^{\cl}\) and the residual norm
\(\Err_{{\rm proj},N}^{\cl}\).  No best-approximation identity, smallness of
\(\Err_{{\rm proj},N}^{\cl}\), or scale-uniform truncation statement is claimed
unless an additional best-approximation hypothesis is explicitly imposed.
\end{remark}

\subsection{Pressure mean normalization residual}

\begin{convention}[Default pressure mean normalization]
\label{conv:default-pressure-mean-normalization}
In the normalized pressure observation space
\[
    Y_{\rm prs}=L^{3/2}(I;L^{3/2}(B_{1/2})),
\]
the default pressure representative is the one with zero spatial mean over the
observation core \(B_{1/2}\) at almost every time.  For
\(q\in Y_{\rm prs}\), write
\[
    \langle q\rangle_{B_{1/2}}(t)
    :=
    \frac{1}{|B_{1/2}|}\int_{B_{1/2}}q(x,t)\,dx
\]
and define the mean-normalization map
\[
    \mathsf Z_{\rm mean}q
    :=
    q-\langle q\rangle_{B_{1/2}}.
\]
Here \(\langle q\rangle_{B_{1/2}}\) is regarded as a function of \((x,t)\) that
is constant in \(x\) on \(B_{1/2}\).
\end{convention}

\begin{lemma}[Bounded pressure mean residual]
\label{lem:bounded-pressure-mean-residual}
For every \(q\in Y_{\rm prs}\), the mean residual
\[
    \Err_{\rm mean}(q)
    :=
    \|\langle q\rangle_{B_{1/2}}\|_{Y_{\rm prs}}
\]
is well defined and satisfies
\[
    \Err_{\rm mean}(q)
    \le
    \|q\|_{Y_{\rm prs}}.
\]
Moreover, \(\mathsf Z_{\rm mean}\) is a bounded projection onto the closed
subspace of \(Y_{\rm prs}\) consisting of functions with zero spatial mean over
\(B_{1/2}\) for almost every \(t\), and
\[
    \|\mathsf Z_{\rm mean}q\|_{Y_{\rm prs}}
    \le
    2\|q\|_{Y_{\rm prs}}.
\]
\end{lemma}

\begin{proof}
For almost every \(t\in I\), Holder's inequality gives
\[
    |\langle q\rangle_{B_{1/2}}(t)|
    \le
    |B_{1/2}|^{-2/3}
    \|q(\cdot,t)\|_{L^{3/2}(B_{1/2})}.
\]
Since \(\langle q\rangle_{B_{1/2}}(t)\) is constant in \(x\), its spatial
\(L^{3/2}(B_{1/2})\) norm is
\[
    \|\langle q\rangle_{B_{1/2}}(t)\|_{L^{3/2}(B_{1/2})}
    =
    |B_{1/2}|^{2/3}|\langle q\rangle_{B_{1/2}}(t)|.
\]
Combining the two displays yields
\[
    \|\langle q\rangle_{B_{1/2}}(t)\|_{L^{3/2}(B_{1/2})}
    \le
    \|q(\cdot,t)\|_{L^{3/2}(B_{1/2})}.
\]
Taking the \(L^{3/2}(I)\) norm proves
\(\Err_{\rm mean}(q)\le\|q\|_{Y_{\rm prs}}\).

The identity
\[
    \langle \mathsf Z_{\rm mean}q\rangle_{B_{1/2}}=0
\]
shows that \(\mathsf Z_{\rm mean}\) maps into the zero-mean subspace, and if
\(q\) already has zero mean then \(\mathsf Z_{\rm mean}q=q\).  Thus
\(\mathsf Z_{\rm mean}^2=\mathsf Z_{\rm mean}\).  The triangle inequality and
the mean estimate give
\[
    \|\mathsf Z_{\rm mean}q\|_{Y_{\rm prs}}
    \le
    \|q\|_{Y_{\rm prs}}+\Err_{\rm mean}(q)
    \le
    2\|q\|_{Y_{\rm prs}}.
\]
Finally, the zero-mean subspace is the kernel of the bounded mean map, and is
therefore closed.
\end{proof}

\begin{lemma}[Vanishing of the mean residual under consistent normalization]
\label{lem:mean-residual-vanishes-under-normalization}
If \(q\in Y_{\rm prs}\) is already normalized so that
\[
    \int_{B_{1/2}}q(x,t)\,dx=0
    \qquad\text{for almost every }t\in I,
\]
then
\[
    \Err_{\rm mean}(q)=0.
\]
\end{lemma}

\begin{proof}
The stated normalization is precisely
\(\langle q\rangle_{B_{1/2}}(t)=0\) for almost every \(t\).  Hence the
constant-in-space mean function vanishes in \(Y_{\rm prs}\).
\end{proof}

\begin{remark}[Status of the pressure mean residual]
\Cref{lem:bounded-pressure-mean-residual} is a fixed-core normalization
statement in the default pressure observation norm.  It does not estimate any
nonconstant pressure component, does not prove smallness of the pressure, and
does not provide scale-uniform control under dyadic rescaling.  It only records
the cost of changing pressure representatives by a time-dependent spatial
constant on \(B_{1/2}\).
\end{remark}

\subsection{Pressure periodization residual}\label{subsec:pressure-periodization}

\begin{convention}[Default pressure periodization model]
\label{conv:default-pressure-periodization-model}
Fix a unit torus \(\T^3\) containing a copy of the normalized core \(B_{1/2}\)
inside one coordinate chart.  Let
\[
    Y_{\rm prs}^{\rm per}
    :=
    L^{3/2}(I;L^{3/2}(\T^3)).
\]
The default pressure periodization map is the zero-extension operator
\[
    \mathsf{Per}_{\rm prs}:Y_{\rm prs}\to Y_{\rm prs}^{\rm per},
\]
defined by
\[
    (\mathsf{Per}_{\rm prs}q)(x,t)
    =
    \begin{cases}
    q(x,t), & x\in B_{1/2},\\
    0, & x\in \T^3\setminus B_{1/2}.
    \end{cases}
\]
This is only a fixed pressure-observation periodization convention.  It is not
a claim that the zero extension is a smooth periodic pressure or that it solves
a periodic pressure equation.
\end{convention}

\begin{lemma}[Bounded pressure periodization]
\label{lem:bounded-pressure-periodization}
The map
\[
    \mathsf{Per}_{\rm prs}:Y_{\rm prs}\to Y_{\rm prs}^{\rm per}
\]
is linear and bounded, with
\[
    \|\mathsf{Per}_{\rm prs}q\|_{Y_{\rm prs}^{\rm per}}
    =
    \|q\|_{Y_{\rm prs}}.
\]
\end{lemma}

\begin{proof}
Linearity is immediate from the definition.  For almost every \(t\),
zero-extension gives
\[
    \|\mathsf{Per}_{\rm prs}q(\cdot,t)\|_{L^{3/2}(\T^3)}
    =
    \|q(\cdot,t)\|_{L^{3/2}(B_{1/2})}.
\]
Taking the \(L^{3/2}(I)\) norm proves the identity and hence boundedness.
\end{proof}

\begin{definition}[Pressure periodization residual]
\label{def:pressure-periodization-residual}
Given a local pressure observation \(q\in Y_{\rm prs}\) and a clean periodic
pressure representative \(q^{\rm per}\in Y_{\rm prs}^{\rm per}\), define
\[
    \Err_{\rm per}(q,q^{\rm per})
    :=
    \|\mathsf{Per}_{\rm prs}q-q^{\rm per}\|_{Y_{\rm prs}^{\rm per}}.
\]
\end{definition}

\begin{lemma}[Well-defined periodization residual]
\label{lem:periodization-residual-well-defined}
For \(q\in Y_{\rm prs}\) and \(q^{\rm per}\in Y_{\rm prs}^{\rm per}\), the
quantity \(\Err_{\rm per}(q,q^{\rm per})\) is a finite, well-defined pressure
error.  Moreover,
\[
    \Err_{\rm per}(q,q^{\rm per})
    \le
    \|q\|_{Y_{\rm prs}}
    +
    \|q^{\rm per}\|_{Y_{\rm prs}^{\rm per}}.
\]
If the clean periodic representative is chosen to be
\(q^{\rm per}=\mathsf{Per}_{\rm prs}q\), then
\[
    \Err_{\rm per}(q,q^{\rm per})=0.
\]
\end{lemma}

\begin{proof}
Boundedness of \(\mathsf{Per}_{\rm prs}\) gives
\(\mathsf{Per}_{\rm prs}q\in Y_{\rm prs}^{\rm per}\), so the difference
\(\mathsf{Per}_{\rm prs}q-q^{\rm per}\) belongs to \(Y_{\rm prs}^{\rm per}\).
The displayed estimate is the triangle inequality together with
\Cref{lem:bounded-pressure-periodization}.  The vanishing statement follows
directly from the definition when \(q^{\rm per}=\mathsf{Per}_{\rm prs}q\).
\end{proof}

\begin{remark}[Status of the periodization residual]
\Cref{lem:periodization-residual-well-defined} is a norm-level periodization
statement for pressure observations only.  It does not prove smooth
periodization, pressure-equation compatibility, boundary commutator control,
or scale-uniform leakage estimates.  Any mismatch between the zero-extended
observation and the chosen clean periodic pressure remains visible as
\(\Err_{\rm per}\).
\end{remark}

\begin{definition}[Pressure error budget]
\label{def:pressure-error-budget}
For a finite window \(\Lambda\), the pressure error budget associated with a
pressure-transfer datum is
\[
    \Err_{\rm prs}(\mathfrak D)
    :=
    \sum_{k\in\Lambda_{\rm sc}}\omega_{{\rm prs},k}
    \Bigl(
        \Err_{{\rm harm-tail},k}
        +
        \Err_{{\rm comm},k}
        +
        \Err_{{\rm active},k}
        +
        \Err_{{\rm proj},k}
        +
        \Err_{{\rm mean},k}
        +
        \Err_{{\rm per},k}
    \Bigr),
\]
where \(\omega_{{\rm prs},k}\ge0\) are fixed finite-window weights.  The
projection entry is now instantiated by
\[
    \Err_{{\rm proj},k}
    :=
    \Err_{{\rm proj},N}^{\cl}(q_k^{\cl})
    =
    \|q_k^{\cl}-\mathsf P_{{\rm prs},N}^{\cl}q_k^{\cl}\|_{Y_{{\rm prs},k}},
\]
where \(q_k^{\cl}\in Y_{{\rm prs},k}\) is the clean pressure observation being
projected at scale \(k\), and
\(Y_{{\rm prs},k}=L^{3/2}(I_k;L^{3/2}(B_{1/2}))\) in normalized coordinates.
The mean entry is instantiated by
\[
    \Err_{{\rm mean},k}
    :=
    \Err_{\rm mean}(q_k^{\rm prs})
    =
    \|\langle q_k^{\rm prs}\rangle_{B_{1/2}}\|_{Y_{{\rm prs},k}},
\]
where \(q_k^{\rm prs}\in Y_{{\rm prs},k}\) is the pressure representative whose
mean must be reconciled with the default zero-mean convention.  If the same
zero-mean convention is used throughout, then \(\Err_{{\rm mean},k}=0\).  The
periodization entry is instantiated by
\[
    \Err_{{\rm per},k}
    :=
    \Err_{\rm per}(q_k^{\rm prs},q_k^{\rm per})
    =
    \|\mathsf{Per}_{\rm prs}q_k^{\rm prs}-q_k^{\rm per}\|_{Y_{{\rm prs},k}^{\rm per}},
\]
where \(q_k^{\rm per}\in Y_{{\rm prs},k}^{\rm per}\) is the clean periodic
pressure representative.  If no periodization mismatch is introduced, or if
\(q_k^{\rm per}=\mathsf{Per}_{\rm prs}q_k^{\rm prs}\), then
\(\Err_{{\rm per},k}=0\).
\end{definition}

\begin{lemma}[Harmonic-tail budget instantiation in the normalized model]
\label{lem:harmonic-tail-budget-instantiation}
Consider a normalized scale component \(k\) with finite time interval \(I_k\).
Assume that its harmonic pressure tail satisfies
\[
    p_k^{\rm harm}\in L^2(I_k;L^2(B_{3/4}))
\]
and that \(p_k^{\rm harm}(t,\cdot)\) is harmonic in \(B_{3/4}\) for almost every
\(t\in I_k\).  If the harmonic-tail entry in
\Cref{def:pressure-error-budget} is instantiated by
\[
    \Err_{{\rm harm-tail},k}
    :=
    |I_k|^{1/6}|B_{1/2}|^{1/6}
    \left(\frac23\right)^{M+1}
    \|p_k^{\rm harm}\|_{L^2(I_k;L^2(B_{3/4}))},
\]
then
\[
    \|(I-\Pi_{{\rm harm},M})p_k^{\rm harm}\|_{Y_{{\rm prs},k}}
    \le
    \Err_{{\rm harm-tail},k},
\]
where \(Y_{{\rm prs},k}=L^{3/2}(I_k;L^{3/2}(B_{1/2}))\) in the normalized
coordinates.
Consequently, for any finite nonnegative pressure weights
\(\omega_{{\rm prs},k}\),
\[
    \sum_{k\in\Lambda_{\rm sc}}
    \omega_{{\rm prs},k}
    \|(I-\Pi_{{\rm harm},M})p_k^{\rm harm}\|_{Y_{{\rm prs},k}}
    \le
    \sum_{k\in\Lambda_{\rm sc}}
    \omega_{{\rm prs},k}\Err_{{\rm harm-tail},k}.
\]
\end{lemma}

\begin{proof}
The single-scale estimate is exactly
\Cref{cor:pressure-observation-harmonic-approximation-tail} applied on the
normalized component \(k\).  Multiplying by the nonnegative weights and summing
over the finite window gives the displayed finite-window bound.
\end{proof}

\begin{remark}[Status of the pressure budget]
\Cref{def:pressure-error-budget} is bookkeeping, not a global pressure estimate.
\Cref{lem:harmonic-tail-budget-instantiation} gives a fixed-geometry
instantiation of the harmonic-tail entry when the normalized harmonic pressure
tail is controlled on \(B_{3/4}\).  The commutator, active-source, projection,
mean, and periodization entries remain separate and are not estimated by the
harmonic-tail argument.  The budget definition does not assert that the total
pressure budget is small, absorbable, or uniform across scales.
\end{remark}

\begin{proposition}[Combined normalized pressure-component package]
\label{prop:combined-normalized-pressure-component-package}
Let \(\Lambda_{\rm sc}\) be a finite scale window.  For each \(k\in\Lambda_{\rm sc}\),
work in normalized coordinates on a finite time interval \(I_k\), with cutoff
\(\eta_k\), tensor source \(f_k=(f_{k,ij})\), clean tensor source
\(F_k^{\cl}=(F_{k,ij}^{\cl})\), and harmonic pressure tail \(p_k^{\rm harm}\).
Assume the following fixed-geometry hypotheses hold for each \(k\):
\begin{enumerate}[label=(\roman*)]
    \item \(p_k^{\rm harm}\in L^2(I_k;L^2(B_{3/4}))\), and
    \(p_k^{\rm harm}(t,\cdot)\) is harmonic in \(B_{3/4}\) for almost every
    \(t\in I_k\);
    \item \(g_{k,ij}:=(1-\eta_k)f_{k,ij}\) is supported in \(A_{3/4,1}\) and
    belongs to \(L^{3/2}(I_k;L^{3/2}(A_{3/4,1}))\);
    \item \(G_{k,ij}:=f_{k,ij}-F_{k,ij}^{\cl}\), extended by zero outside
    \(B_1\), belongs to \(L^{3/2}(I_k;L^{3/2}(\R^3))\).
\end{enumerate}
Define
\[
\begin{aligned}
    \mathcal B_{\rm prs}^{\rm norm}(\mathfrak D)
    :=
    \sum_{k\in\Lambda_{\rm sc}}\omega_{{\rm prs},k}
    \Bigg(
        &|I_k|^{1/6}|B_{1/2}|^{1/6}
        \left(\frac23\right)^{M+1}
        \|p_k^{\rm harm}\|_{L^2(I_k;L^2(B_{3/4}))}  \\
        &+
        C_{\rm nlp}\sum_{i,j=1}^3
        \|(1-\eta_k)f_{k,ij}\|_{L^{3/2}(I_k;L^{3/2}(A_{3/4,1}))} \\
        &+
        C_{\rm CZ}\sum_{i,j=1}^3
        \|f_{k,ij}-F_{k,ij}^{\cl}\|_{L^{3/2}(I_k;L^{3/2}(B_1))} \\
        &+
        \Err_{{\rm proj},k}
        +
        \Err_{{\rm mean},k}
        +
        \Err_{{\rm per},k}
    \Bigg).
\end{aligned}
\]
If the pressure budget in \Cref{def:pressure-error-budget} is instantiated by
the normalized harmonic-tail, commutator, and active-source entries above, then
\[
    \Err_{\rm prs}(\mathfrak D)
    \le
    \mathcal B_{\rm prs}^{\rm norm}(\mathfrak D).
\]
\end{proposition}

\begin{proof}
The harmonic-tail entry is controlled by
\Cref{lem:harmonic-tail-budget-instantiation}.  The commutator entry is
controlled by \Cref{lem:normalized-cutoff-riesz-commutator-estimate}.  The
active-source entry is controlled by
\Cref{lem:normalized-active-pressure-source-residual}.  The projection, mean,
and periodization terms are not estimated; they are retained as the explicit
residuals \(\Err_{{\rm proj},k}\), \(\Err_{{\rm mean},k}\), and
\(\Err_{{\rm per},k}\).  Substituting these component bounds into
\Cref{def:pressure-error-budget} and summing over the finite window with
nonnegative weights gives the claim.
\end{proof}

\begin{remark}[Status of the combined pressure package]
\Cref{prop:combined-normalized-pressure-component-package} is a fixed-geometry
component package.  It combines the pressure estimates that have actually been
proved in the normalized model and leaves the remaining residuals explicit.
It does not prove that \(\mathcal B_{\rm prs}^{\rm norm}\) is small, does not
absorb the pressure budget into the quotient distance, and does not provide
scale-uniform localization, truncation, or periodization control.
\end{remark}

\subsection{Energy--flux localization leakage}\label{subsec:energy-flux-localization}

\begin{convention}[Normalized energy--flux cutoff]
\label{conv:normalized-energy-flux-cutoff}
For the energy--flux localization leakage model, set
\[
    Q_1:=B_1\times(-1,0),
    \qquad
    Q_{1/2}:=B_{1/2}\times(-1/4,0).
\]
Fix a cutoff
\[
    \chi\in C_c^\infty(Q_1),
    \qquad
    0\le \chi\le1,
    \qquad
    \chi\equiv1\ \text{on }Q_{1/2},
\]
and define the fixed cutoff constant
\[
    C_\chi
    :=
    \|\partial_t\chi\|_{L^\infty(Q_1)}
    +
    \|\Delta\chi\|_{L^\infty(Q_1)}
    +
    \|\nabla\chi\|_{L^\infty(Q_1)}.
\]
This convention is fixed-scale and normalized.  No scale-uniform family of
cutoffs is asserted here.
\end{convention}

\begin{definition}[Energy--flux localization leakage errors]
\label{def:energy-flux-localization-leakage-errors}
For a velocity-pressure pair \((u,p)\) on \(Q_1\), define
\[
    \Err_{\rm loc}^{\rm en}(u)
    :=
    \int_{Q_1}|u|^2\bigl(|\partial_t\chi|+|\Delta\chi|\bigr),
\]
\[
    \Err_{\rm loc}^{\rm flux}(u)
    :=
    \int_{Q_1}|u|^2|u\cdot\nabla\chi|,
\]
and
\[
    \Err_{\rm loc}^{\rm prs}(u,p)
    :=
    \int_{Q_1}|p-(p)_{B_1}(t)|\,|u|\,|\nabla\chi|,
\]
where
\[
    (p)_{B_1}(t):=\frac1{|B_1|}\int_{B_1}p(x,t)\,dx.
\]
\end{definition}

\begin{lemma}[Fixed-scale energy--flux leakage estimates]
\label{lem:fixed-scale-energy-flux-leakage-estimates}
Let \(u\in L^3(Q_1)\cap L^2(Q_1)\), and let
\[
    p-(p)_{B_1}(\cdot)\in L^{3/2}(Q_1).
\]
Then
\[
    \Err_{\rm loc}^{\rm en}(u)
    \le
    C_\chi \|u\|_{L^2(Q_1)}^2,
\]
\[
    \Err_{\rm loc}^{\rm flux}(u)
    \le
    C_\chi \|u\|_{L^3(Q_1)}^3,
\]
and
\[
    \Err_{\rm loc}^{\rm prs}(u,p)
    \le
    C_\chi
    \|p-(p)_{B_1}(\cdot)\|_{L^{3/2}(Q_1)}
    \|u\|_{L^3(Q_1)}.
\]
\end{lemma}

\begin{proof}
The energy leakage estimate follows directly from the definition and the
boundedness of \(\partial_t\chi\) and \(\Delta\chi\):
\[
    \Err_{\rm loc}^{\rm en}(u)
    \le
    \bigl(\|\partial_t\chi\|_{L^\infty}
       +\|\Delta\chi\|_{L^\infty}\bigr)
    \int_{Q_1}|u|^2
    \le
    C_\chi\|u\|_{L^2(Q_1)}^2.
\]
For the flux leakage,
\[
    |u|^2|u\cdot\nabla\chi|
    \le
    \|\nabla\chi\|_{L^\infty}|u|^3,
\]
and integration over \(Q_1\) gives the claimed bound.  Finally, Holder's
inequality with exponents \(3/2\) and \(3\) gives
\[
\begin{aligned}
    \Err_{\rm loc}^{\rm prs}(u,p)
    &\le
    \|\nabla\chi\|_{L^\infty}
    \int_{Q_1}|p-(p)_{B_1}(t)|\,|u|  \\
    &\le
    C_\chi
    \|p-(p)_{B_1}(\cdot)\|_{L^{3/2}(Q_1)}
    \|u\|_{L^3(Q_1)}.
\end{aligned}
\]
\end{proof}

\begin{definition}[Normalized CKN quantities on \(Q_1\)]
\label{def:normalized-ckn-quantities}
On \(Q_1=B_1\times(-1,0)\), set
\[
    A:=\operatorname*{ess\,sup}_{-1<t<0}\int_{B_1}|u(x,t)|^2\,dx,
\]
\[
    C:=\int_{Q_1}|u|^3,
    \qquad
    D:=\int_{Q_1}|p-(p)_{B_1}(t)|^{3/2}.
\]
\end{definition}

\begin{corollary}[CKN-normalized leakage bounds]
\label{cor:ckn-normalized-leakage-bounds}
Under the hypotheses of
\Cref{lem:fixed-scale-energy-flux-leakage-estimates},
\[
    \Err_{\rm loc}^{\rm en}(u)
    \le
    C_\chi A,
\]
\[
    \Err_{\rm loc}^{\rm flux}(u)
    \le
    C_\chi C,
\]
and
\[
    \Err_{\rm loc}^{\rm prs}(u,p)
    \le
    C_\chi C^{1/3}D^{2/3}.
\]
\end{corollary}

\begin{proof}
Since the normalized time interval has length one,
\[
    \|u\|_{L^2(Q_1)}^2
    =
    \int_{-1}^{0}\int_{B_1}|u(x,t)|^2\,dx\,dt
    \le
    A.
\]
The flux estimate is exactly
\[
    \|u\|_{L^3(Q_1)}^3=C.
\]
For the pressure leakage, use
\[
    \|u\|_{L^3(Q_1)}=C^{1/3},
    \qquad
    \|p-(p)_{B_1}(\cdot)\|_{L^{3/2}(Q_1)}=D^{2/3}.
\]
Substitution into
\Cref{lem:fixed-scale-energy-flux-leakage-estimates} proves the three bounds.
\end{proof}

\begin{remark}[Status of the energy--flux leakage bounds]
\Cref{lem:fixed-scale-energy-flux-leakage-estimates} and
\Cref{cor:ckn-normalized-leakage-bounds} are fixed-scale localization estimates
for the energy, flux, and pressure-transport terms generated by the cutoff
\(\chi\).  They do not prove smallness of the leakage terms, do not prove
scale-uniform localization control, and do not address pressure-only or
momentum localization residuals.
\end{remark}

\subsection{Momentum residual localization leakage}

\begin{convention}[Default momentum residual test norm]
\label{conv:default-momentum-residual-test-norm}
For the momentum localization residual, use the same normalized cylinder
\(Q_1\) and cutoff \(\chi\) as in
\Cref{conv:normalized-energy-flux-cutoff}.  The default test norm is
\[
    \|\varphi\|_{\mathcal X_{\rm mom}}
    :=
    \|\varphi\|_{L^2(Q_1)}
    +
    \|\varphi\|_{L^3(Q_1)},
    \qquad
    \varphi\in C_c^\infty(Q_1;\R^3).
\]
For a vector-valued distribution \(T\), define
\[
    \|T\|_{\mathcal X_{\rm mom}^*}
    :=
    \sup_{\substack{\varphi\in C_c^\infty(Q_1;\R^3)\\
    \|\varphi\|_{\mathcal X_{\rm mom}}\le1}}
    |\langle T,\varphi\rangle|.
\]
This is a fixed-scale negative norm.  It is chosen only to record the
distributional size of cutoff-generated momentum errors.
\end{convention}

\begin{definition}[Cutoff momentum residual]
\label{def:cutoff-momentum-residual}
Let
\[
    p^\circ(x,t):=p(x,t)-(p)_{B_1}(t).
\]
For a velocity-pressure pair \((u,p)\) on \(Q_1\), define the cutoff momentum
residual
\[
    \mathcal R_{\chi}^{\rm mom}(u,p)
    :=
    \partial_t(\chi u)
    -
    \Delta(\chi u)
    +
    \nabla\cdot(\chi\,u\otimes u)
    +
    \nabla(\chi p^\circ)
\]
in the sense of distributions on \(Q_1\).  Its localization leakage size is
\[
    \Err_{\rm loc}^{\rm mom}(u,p)
    :=
    \|\mathcal R_{\chi}^{\rm mom}(u,p)\|_{\mathcal X_{\rm mom}^*}.
\]
\end{definition}

\begin{lemma}[Cutoff momentum residual identity]
\label{lem:cutoff-momentum-residual-identity}
Assume that \((u,p)\) satisfies the incompressible Navier--Stokes momentum
equation
\[
    \partial_tu-\Delta u+\nabla\cdot(u\otimes u)+\nabla p=0
\]
distributionally on \(Q_1\), with
\[
    u\in L^2(-1,0;H^1(B_1))\cap L^3(Q_1),
    \qquad
    p^\circ\in L^{3/2}(Q_1).
\]
Then, for every \(\varphi\in C_c^\infty(Q_1;\R^3)\),
\[
\begin{aligned}
    \langle \mathcal R_\chi^{\rm mom}(u,p),\varphi\rangle
    =
    \int_{Q_1}
    &(\partial_t\chi-\Delta\chi)\,u\cdot\varphi
    -
    2\sum_{j=1}^3(\partial_j\chi)\,\partial_j u\cdot\varphi     \\
    &+
    \bigl((u\otimes u)\nabla\chi\bigr)\cdot\varphi
    +
    p^\circ\,\nabla\chi\cdot\varphi .
\end{aligned}
\]
\end{lemma}

\begin{proof}
Since \(p^\circ=p-(p)_{B_1}(t)\), one has
\(\nabla p^\circ=\nabla p\).  Expanding the localized residual by the product
rule gives
\[
\begin{aligned}
    \mathcal R_\chi^{\rm mom}(u,p)
    =
    &\chi\bigl(\partial_tu-\Delta u+\nabla\cdot(u\otimes u)+\nabla p^\circ\bigr)\\
    &+
    (\partial_t\chi-\Delta\chi)u
    -
    2\sum_{j=1}^3(\partial_j\chi)\partial_j u
    +
    (u\otimes u)\nabla\chi
    +
    p^\circ\nabla\chi .
\end{aligned}
\]
The term on the first line vanishes distributionally by the momentum equation.
The remaining terms belong to \(L^1_{\rm loc}\) under the stated hypotheses,
so pairing with \(\varphi\) gives the displayed identity.
\end{proof}

\begin{lemma}[Fixed-scale momentum residual leakage estimate]
\label{lem:fixed-scale-momentum-residual-leakage-estimate}
Under the hypotheses of
\Cref{lem:cutoff-momentum-residual-identity},
\[
\begin{aligned}
    \Err_{\rm loc}^{\rm mom}(u,p)
    \le
    C_\chi\Bigl(
        \|u\|_{L^2(Q_1)}
        +
        \|\nabla u\|_{L^2(Q_1)}
        +
        \|u\|_{L^3(Q_1)}^2
        +
        \|p^\circ\|_{L^{3/2}(Q_1)}
    \Bigr).
\end{aligned}
\]
\end{lemma}

\begin{proof}
Let \(\varphi\in C_c^\infty(Q_1;\R^3)\).  By
\Cref{lem:cutoff-momentum-residual-identity},
\[
\begin{aligned}
    |\langle \mathcal R_\chi^{\rm mom}(u,p),\varphi\rangle|
    \le
    C_\chi
    (&\|u\|_{L^2(Q_1)}\|\varphi\|_{L^2(Q_1)}
    +
    \|\nabla u\|_{L^2(Q_1)}\|\varphi\|_{L^2(Q_1)} \\
    &+
    \|u\|_{L^3(Q_1)}^2\|\varphi\|_{L^3(Q_1)}
    +
    \|p^\circ\|_{L^{3/2}(Q_1)}\|\varphi\|_{L^3(Q_1)}).
\end{aligned}
\]
The first two terms use Cauchy's inequality.  The last two terms use Holder's
inequality with exponents \(3/2\) and \(3\).  If
\(\|\varphi\|_{\mathcal X_{\rm mom}}\le1\), the right-hand side is bounded by
the displayed quantity.  Taking the supremum over such \(\varphi\) proves the
estimate.
\end{proof}

\begin{corollary}[CKN-normalized momentum residual leakage bound]
\label{cor:ckn-normalized-momentum-residual-leakage-bound}
In addition to \(A\), \(C\), and \(D\) from
\Cref{def:normalized-ckn-quantities}, set
\[
    E:=\int_{Q_1}|\nabla u|^2.
\]
Then
\[
    \Err_{\rm loc}^{\rm mom}(u,p)
    \le
    C_\chi
    \bigl(
        A^{1/2}
        +
        E^{1/2}
        +
        C^{2/3}
        +
        D^{2/3}
    \bigr).
\]
\end{corollary}

\begin{proof}
On the normalized cylinder, \(\|u\|_{L^2(Q_1)}^2\le A\),
\(\|\nabla u\|_{L^2(Q_1)}=E^{1/2}\),
\(\|u\|_{L^3(Q_1)}^2=C^{2/3}\), and
\(\|p^\circ\|_{L^{3/2}(Q_1)}=D^{2/3}\).  Substituting these identities and
the preceding bound from
\Cref{lem:fixed-scale-momentum-residual-leakage-estimate} proves the claim.
\end{proof}

\begin{remark}[Status of the momentum residual leakage bound]
\Cref{lem:fixed-scale-momentum-residual-leakage-estimate} and
\Cref{cor:ckn-normalized-momentum-residual-leakage-bound} give a fixed-scale
distributional estimate for the cutoff-generated momentum residual.  The
residual uses the localized flux \(\chi u\otimes u\), not the nonlinear flux
\((\chi u)\otimes(\chi u)\); the latter would generate an additional nonlinear
mismatch term.  No smallness, scale-uniform localization control, absorption
into the clean residual norm, or full localized Navier--Stokes transfer theorem
is claimed here.
\end{remark}

\subsection{Finite-window localization budget}\label{subsec:localization-budget}

\begin{definition}[Concrete localization budget]
\label{def:concrete-localization-budget}
Let \(\Lambda_{\rm sc}\) be a finite scale window.  For each
\(k\in\Lambda_{\rm sc}\), let \((u_k,p_k)\) denote the normalized
velocity-pressure component on \(Q_1\), and define
\[
    p_k^\circ:=p_k-(p_k)_{B_1}(t).
\]
Fix nonnegative finite-window weights
\[
    \omega_{{\rm en},k},\quad
    \omega_{{\rm flux},k},\quad
    \omega_{{\rm prs},k}^{\loc},\quad
    \omega_{{\rm mom},k}\ge0.
\]
The concrete localization error budget is
\[
\begin{aligned}
    \Err_{\rm loc}(\mathfrak D)
    :=
    \sum_{k\in\Lambda_{\rm sc}}
    \Bigl(
        &\omega_{{\rm en},k}\Err_{\rm loc}^{\rm en}(u_k)
        +
        \omega_{{\rm flux},k}\Err_{\rm loc}^{\rm flux}(u_k)  \\
        &+
        \omega_{{\rm prs},k}^{\loc}
        \Err_{\rm loc}^{\rm prs}(u_k,p_k)
        +
        \omega_{{\rm mom},k}
        \Err_{\rm loc}^{\rm mom}(u_k,p_k)
    \Bigr).
\end{aligned}
\]
This definition uses only the localization leakage channels already introduced
above.
\end{definition}

\begin{definition}[Normalized localization budget]
\label{def:normalized-localization-budget}
For each \(k\in\Lambda_{\rm sc}\), define the normalized quantities
\[
    A_k:=\operatorname*{ess\,sup}_{-1<t<0}\int_{B_1}|u_k(x,t)|^2\,dx,
\]
\[
    C_k:=\int_{Q_1}|u_k|^3,
    \qquad
    D_k:=\int_{Q_1}|p_k^\circ|^{3/2},
    \qquad
    E_k:=\int_{Q_1}|\nabla u_k|^2 .
\]
The normalized localization budget is
\[
\begin{aligned}
    \mathcal B_{\rm loc}^{\rm norm}(\mathfrak D)
    :=
    C_\chi
    \sum_{k\in\Lambda_{\rm sc}}
    \Bigl(
        &\omega_{{\rm en},k} A_k
        +
        \omega_{{\rm flux},k} C_k
        +
        \omega_{{\rm prs},k}^{\loc} C_k^{1/3}D_k^{2/3} \\
        &+
        \omega_{{\rm mom},k}
        \bigl(
            A_k^{1/2}
            +
            E_k^{1/2}
            +
            C_k^{2/3}
            +
            D_k^{2/3}
        \bigr)
    \Bigr).
\end{aligned}
\]
\end{definition}

\begin{lemma}[Finite-window localization budget]
\label{lem:finite-window-localization-budget}
Assume that each normalized component \((u_k,p_k)\) satisfies the hypotheses of
\Cref{cor:ckn-normalized-leakage-bounds} and
\Cref{cor:ckn-normalized-momentum-residual-leakage-bound}.  Then
\[
    \Err_{\rm loc}(\mathfrak D)
    \le
    \mathcal B_{\rm loc}^{\rm norm}(\mathfrak D).
\]
\end{lemma}

\begin{proof}
For each fixed \(k\), \Cref{cor:ckn-normalized-leakage-bounds} gives
\[
    \Err_{\rm loc}^{\rm en}(u_k)\le C_\chi A_k,
    \qquad
    \Err_{\rm loc}^{\rm flux}(u_k)\le C_\chi C_k,
\]
and
\[
    \Err_{\rm loc}^{\rm prs}(u_k,p_k)
    \le
    C_\chi C_k^{1/3}D_k^{2/3}.
\]
Also, \Cref{cor:ckn-normalized-momentum-residual-leakage-bound} gives
\[
    \Err_{\rm loc}^{\rm mom}(u_k,p_k)
    \le
    C_\chi
    \bigl(
        A_k^{1/2}
        +
        E_k^{1/2}
        +
        C_k^{2/3}
        +
        D_k^{2/3}
    \bigr).
\]
Multiplying these four inequalities by the corresponding nonnegative weights
and summing over the finite set \(\Lambda_{\rm sc}\) gives the claimed bound.
\end{proof}

\begin{remark}[Status of the localization budget]
\Cref{lem:finite-window-localization-budget} only instantiates
\(\Err_{\rm loc}\) by summing the already proved fixed-scale localization
leakage components.  It does not introduce a new localization channel, does
not claim that \(\mathcal B_{\rm loc}^{\rm norm}\) is small, does not provide
scale-uniform localization control, and does not absorb \(\Err_{\rm loc}\)
into a term of the form
\(\eta_\Lambda\Dist_{\loc}+\Delta_\Lambda\) without an additional explicit
perturbative assumption.
\end{remark}

\subsection{Truncation leakage}\label{subsec:truncation-leakage}

\begin{convention}[Finite-window truncation datum]
\label{conv:finite-window-truncation-datum}
For each \(k\in\Lambda_{\rm sc}\), fix a Banach observation space
\[
    \calY_{{\rm tr},k}
\]
and a finite-dimensional clean truncation space
\[
    V_{{\rm tr},N,k}^{\cl}\subset \calY_{{\rm tr},k}.
\]
The object to be truncated is denoted by
\[
    z_k(\mathfrak D)\in\calY_{{\rm tr},k}.
\]
This datum is finite-window and norm-level.  It does not specify a Fourier
basis, does not impose a dyadic tail estimate, and does not assert that
truncation errors decay with \(N\).
\end{convention}

\begin{lemma}[Bounded finite-dimensional truncation projection]
\label{lem:bounded-finite-dimensional-truncation-projection}
For each \(k\in\Lambda_{\rm sc}\), there exists a bounded linear projection
\[
    \mathsf T_{{\rm tr},N,k}:\calY_{{\rm tr},k}\to V_{{\rm tr},N,k}^{\cl}.
\]
\end{lemma}

\begin{proof}
Fix \(k\).  Since \(V_{{\rm tr},N,k}^{\cl}\) is finite-dimensional, choose a
basis \(e_{1,k},\dots,e_{m_k,k}\).  The coordinate functionals on
\(V_{{\rm tr},N,k}^{\cl}\) are continuous because all norms on a
finite-dimensional space are equivalent.  By Hahn--Banach, extend these
coordinate functionals to bounded linear functionals
\(\ell_{1,k},\dots,\ell_{m_k,k}\) on \(\calY_{{\rm tr},k}\).  Define
\[
    \mathsf T_{{\rm tr},N,k}y
    :=
    \sum_{\alpha=1}^{m_k}\ell_{\alpha,k}(y)e_{\alpha,k}.
\]
This map is bounded and linear.  If \(y\in V_{{\rm tr},N,k}^{\cl}\), the
coordinate representation gives
\(\mathsf T_{{\rm tr},N,k}y=y\).  Hence \(\mathsf T_{{\rm tr},N,k}\) is a
projection onto \(V_{{\rm tr},N,k}^{\cl}\).
\end{proof}

\begin{definition}[Truncation residual]
\label{def:truncation-residual}
Given the bounded projections from
\Cref{lem:bounded-finite-dimensional-truncation-projection}, define the
single-scale truncation residual
\[
    \Err_{{\rm tr},k}(\mathfrak D)
    :=
    \|z_k(\mathfrak D)-\mathsf T_{{\rm tr},N,k}z_k(\mathfrak D)\|_{\calY_{{\rm tr},k}}.
\]
For nonnegative finite-window truncation weights \(\omega_{{\rm tr},k}\ge0\),
define
\[
    \Err_{\rm tr}(\mathfrak D)
    :=
    \sum_{k\in\Lambda_{\rm sc}}
    \omega_{{\rm tr},k}\Err_{{\rm tr},k}(\mathfrak D).
\]
\end{definition}

\begin{lemma}[Finite-window truncation leakage bound]
\label{lem:finite-window-truncation-leakage-bound}
Let
\[
    \mathcal B_{\rm tr}^{\rm norm}(\mathfrak D)
    :=
    \sum_{k\in\Lambda_{\rm sc}}
    \omega_{{\rm tr},k}
    \bigl(1+\|\mathsf T_{{\rm tr},N,k}\|_{\calY_{{\rm tr},k}\to\calY_{{\rm tr},k}}\bigr)
    \|z_k(\mathfrak D)\|_{\calY_{{\rm tr},k}}.
\]
Then
\[
    \Err_{\rm tr}(\mathfrak D)
    \le
    \mathcal B_{\rm tr}^{\rm norm}(\mathfrak D).
\]
\end{lemma}

\begin{proof}
For each \(k\),
\[
\begin{aligned}
    \Err_{{\rm tr},k}(\mathfrak D)
    &=
    \|z_k(\mathfrak D)-\mathsf T_{{\rm tr},N,k}z_k(\mathfrak D)\|_{\calY_{{\rm tr},k}} \\
    &\le
    \|z_k(\mathfrak D)\|_{\calY_{{\rm tr},k}}
    +
    \|\mathsf T_{{\rm tr},N,k}z_k(\mathfrak D)\|_{\calY_{{\rm tr},k}} \\
    &\le
    \bigl(1+\|\mathsf T_{{\rm tr},N,k}\|_{\calY_{{\rm tr},k}\to\calY_{{\rm tr},k}}\bigr)
    \|z_k(\mathfrak D)\|_{\calY_{{\rm tr},k}}.
\end{aligned}
\]
Multiplying by \(\omega_{{\rm tr},k}\ge0\) and summing over the finite window
gives the result.
\end{proof}

\begin{remark}[Status of truncation leakage]
\Cref{lem:finite-window-truncation-leakage-bound} is a norm-level
finite-window truncation estimate.  It proves that the truncation residual is
well defined and bounded by the untruncated observation size and the projection
operator norm.  It does not prove that \(\Err_{\rm tr}\) is small, does not
prove any decay as \(N\to\infty\), and does not provide scale-uniform tail
control.  Any such approximation theorem must be supplied as an additional
input in the chosen truncation space.
\end{remark}

\subsection{Nonlinear cutoff-flux mismatch}\label{subsec:nonlinear-cutoff}

\begin{convention}[Default nonlinear mismatch test norm]
\label{conv:default-nonlinear-mismatch-test-norm}
For the nonlinear cutoff-flux mismatch, use the fixed test norm
\[
    \|\varphi\|_{\mathcal X_{\rm nl}}
    :=
    \|\varphi\|_{L^3(Q_1)}
    +
    \|\nabla\varphi\|_{L^3(Q_1)},
    \qquad
    \varphi\in C_c^\infty(Q_1;\R^3).
\]
For a vector-valued distribution \(T\), set
\[
    \|T\|_{\mathcal X_{\rm nl}^*}
    :=
    \sup_{\substack{\varphi\in C_c^\infty(Q_1;\R^3)\\
    \|\varphi\|_{\mathcal X_{\rm nl}}\le1}}
    |\langle T,\varphi\rangle|.
\]
The gradient term is part of the norm because the mismatch is measured after
taking a divergence.
\end{convention}

\begin{definition}[Nonlinear cutoff-flux mismatch]
\label{def:nonlinear-cutoff-flux-mismatch}
Define
\[
    \mathcal R_\chi^{\rm nl}(u)
    :=
    \nabla\cdot\left((\chi u)\otimes(\chi u)-\chi\,u\otimes u\right)
\]
as a distribution on \(Q_1\).  Its fixed-scale mismatch size is
\[
    \Err_{\rm nl}^{\rm cut}(u)
    :=
    \|\mathcal R_\chi^{\rm nl}(u)\|_{\mathcal X_{\rm nl}^*}.
\]
\end{definition}

\begin{lemma}[Fixed-scale nonlinear cutoff-flux mismatch]
\label{lem:fixed-scale-nonlinear-cutoff-flux-mismatch}
If \(u\in L^3(Q_1)\), then
\[
    \Err_{\rm nl}^{\rm cut}(u)
    \le
    C_\chi\|u\|_{L^3(Q_1)}^2.
\]
In normalized CKN variables, this is
\[
    \Err_{\rm nl}^{\rm cut}(u)
    \le
    C_\chi C^{2/3}.
\]
\end{lemma}

\begin{proof}
Let \(\varphi\in C_c^\infty(Q_1;\R^3)\).  Since
\[
    (\chi u)\otimes(\chi u)-\chi\,u\otimes u
    =
    \chi(\chi-1)u\otimes u,
\]
distributional integration by parts gives
\[
\begin{aligned}
    |\langle \mathcal R_\chi^{\rm nl}(u),\varphi\rangle|
    &=
    \left|
    \int_{Q_1}
    \chi(\chi-1)(u\otimes u):\nabla\varphi
    \right|  \\
    &\le
    \|\chi(\chi-1)\|_{L^\infty(Q_1)}
    \|u\otimes u\|_{L^{3/2}(Q_1)}
    \|\nabla\varphi\|_{L^3(Q_1)}       \\
    &\le
    C_\chi \|u\|_{L^3(Q_1)}^2
    \|\varphi\|_{\mathcal X_{\rm nl}}.
\end{aligned}
\]
Taking the supremum over
\(\|\varphi\|_{\mathcal X_{\rm nl}}\le1\) proves the first estimate.  The
normalized form follows from \(\|u\|_{L^3(Q_1)}^2=C^{2/3}\).
\end{proof}

\begin{definition}[Finite-window nonlinear mismatch budget]
\label{def:finite-window-nonlinear-mismatch-budget}
For nonnegative finite-window weights \(\omega_{{\rm nl},k}^{\rm cut}\ge0\),
define
\[
    \Err_{\rm nl}^{\rm cut}(\mathfrak D)
    :=
    \sum_{k\in\Lambda_{\rm sc}}
    \omega_{{\rm nl},k}^{\rm cut}\Err_{\rm nl}^{\rm cut}(u_k),
\]
and
\[
    \mathcal B_{\rm nl}^{\rm cut}(\mathfrak D)
    :=
    C_\chi\sum_{k\in\Lambda_{\rm sc}}
    \omega_{{\rm nl},k}^{\rm cut}C_k^{2/3}.
\]
\end{definition}

\begin{lemma}[Finite-window nonlinear cutoff budget]
\label{lem:finite-window-nonlinear-cutoff-budget}
If every normalized component satisfies \(u_k\in L^3(Q_1)\), then
\[
    \Err_{\rm nl}^{\rm cut}(\mathfrak D)
    \le
    \mathcal B_{\rm nl}^{\rm cut}(\mathfrak D).
\]
\end{lemma}

\begin{proof}
Apply \Cref{lem:fixed-scale-nonlinear-cutoff-flux-mismatch} to each
normalized component, multiply by the nonnegative weights, and sum over the
finite scale window.
\end{proof}

\begin{remark}[Status of the nonlinear cutoff-flux mismatch]
The estimates in this subsection only account for the algebraic mismatch
between the flux \(\chi u\otimes u\) used in the cutoff momentum residual and
the flux \((\chi u)\otimes(\chi u)\) associated with the localized velocity
\(\chi u\).  They do not prove smallness, scale-uniform control, or absorption
into the clean nonlinear residual.  The test norm \(\mathcal X_{\rm nl}\) is
stronger than the previous momentum residual norm because a divergence is being
estimated.
\end{remark}

\begin{definition}[Finite-window nonlinear remainder model]
\label{def:finite-window-nonlinear-remainder-model}
For each \(k\in\Lambda_{\rm sc}\), let
\(\mathcal Z_{{\rm nl},k}\) be a finite-dimensional normed residual space and
let
\[
    N_{{\rm nl},k}^{\rm rem}:
    \calD_k^{\loc}\to\mathcal Z_{{\rm nl},k}
\]
be a bounded linear map representing the nonlinear residual not accounted for
by the cutoff-flux mismatch.  For nonnegative finite-window weights
\(\omega_{{\rm nl},k}^{\rm rem}\ge0\), define
\[
    \Err_{\rm nl}^{\rm rem}(\mathfrak D)
    :=
    \left(
    \sum_{k\in\Lambda_{\rm sc}}
    (\omega_{{\rm nl},k}^{\rm rem})^2
    \|N_{{\rm nl},k}^{\rm rem}D_k^{\loc}\|_{\mathcal Z_{{\rm nl},k}}^2
    \right)^{1/2}.
\]
This models only the finite-window remainder left after the already estimated
cutoff-flux mismatch.
\end{definition}

\begin{lemma}[Finite-dimensional nonlinear remainder bound]
\label{lem:finite-dimensional-nonlinear-remainder-bound}
Under \Cref{def:finite-window-nonlinear-remainder-model}, set
\[
    C_{\rm nl}^{\rm rem}
    :=
    \max_{k\in\Lambda_{\rm sc}}
    \omega_{{\rm nl},k}^{\rm rem}
    \|N_{{\rm nl},k}^{\rm rem}\|_{\calD_k^{\loc}\to\mathcal Z_{{\rm nl},k}},
\]
with \(C_{\rm nl}^{\rm rem}=0\) if \(\Lambda_{\rm sc}\) is empty.  Then
\[
    \Err_{\rm nl}^{\rm rem}(\mathfrak D)
    \le
    C_{\rm nl}^{\rm rem}
    \left(
    \sum_{k\in\Lambda_{\rm sc}}
    \|D_k^{\loc}\|_{\calD_k^{\loc}}^2
    \right)^{1/2}.
\]
\end{lemma}

\begin{proof}
For each \(k\in\Lambda_{\rm sc}\), boundedness of
\(N_{{\rm nl},k}^{\rm rem}\) gives
\[
    \omega_{{\rm nl},k}^{\rm rem}
    \|N_{{\rm nl},k}^{\rm rem}D_k^{\loc}\|_{\mathcal Z_{{\rm nl},k}}
    \le
    C_{\rm nl}^{\rm rem}
    \|D_k^{\loc}\|_{\calD_k^{\loc}}.
\]
Squaring, summing over the finite scale window, and taking the square root
proves the claim.
\end{proof}

\begin{remark}[Status of the nonlinear remainder model]
\Cref{lem:finite-dimensional-nonlinear-remainder-bound} is only a
finite-dimensional boundedness statement for the explicitly chosen remainder
operators \(N_{{\rm nl},k}^{\rm rem}\).  It does not identify the full
Navier--Stokes nonlinear residual, does not prove smallness, does not prove
scale-uniform control, and does not absorb
\(\Err_{\rm nl}^{\rm rem}\) into
\(\eta_\Lambda\Dist_{\loc}+\Delta_\Lambda\).
\end{remark}

\subsection{Reproduction drift budget}\label{subsec:reproduction-drift}

\begin{definition}[Abstract finite-window reproduction maps]
\label{def:abstract-finite-window-reproduction-maps}
Let
\[
    \Lambda_{\rm adj}
    :=
    \{k\in\Lambda_{\rm sc}: k+1\in\Lambda_{\rm sc}\}
\]
be the adjacent-scale index set.  For each \(k\in\Lambda_{\rm adj}\), fix
finite-dimensional normed local and clean component spaces
\[
    \calD_k^{\loc},\quad \calD_{k+1}^{\loc},
    \qquad
    \calD_k^{\cl},\quad \calD_{k+1}^{\cl},
\]
bounded linear local-to-clean component charts
\[
    \Theta_k:\calD_k^{\loc}\to\calD_k^{\cl},
    \qquad
    \Theta_{k+1}:\calD_{k+1}^{\loc}\to\calD_{k+1}^{\cl},
\]
and bounded linear reproduction maps
\[
    R_k^{\loc}:\calD_k^{\loc}\to\calD_{k+1}^{\loc},
    \qquad
    R_k^{\cl}:\calD_k^{\cl}\to\calD_{k+1}^{\cl}.
\]
This is the default abstract finite-window reproduction model.  The bounded
linearity of the component charts is part of the datum and is used in the drift
identity below.  The definition does not assert that either reproduction map is
generated by dyadic Navier--Stokes rescaling.
\end{definition}

\begin{definition}[Localized and clean reproduction residuals]
\label{def:localized-clean-reproduction-residuals}
For a finite-window localized package
\[
    \mathfrak D^{\loc}=(D_k^{\loc})_{k\in\Lambda_{\rm sc}},
    \qquad
    D_k^{\loc}\in\calD_k^{\loc},
\]
define
\[
    \Rep_{\rm loc}(\mathfrak D)^2
    :=
    \sum_{k\in\Lambda_{\rm adj}}
    \|D_{k+1}^{\loc}-R_k^{\loc}D_k^{\loc}\|_{\calD_{k+1}^{\loc}}^2.
\]
The clean reproduction residual of the charted package is
\[
    \Rep_{\rm cl}(\Theta\mathfrak D)^2
    :=
    \sum_{k\in\Lambda_{\rm adj}}
    \|\Theta_{k+1}D_{k+1}^{\loc}
      -R_k^{\cl}\Theta_kD_k^{\loc}\|_{\calD_{k+1}^{\cl}}^2.
\]
\end{definition}

\begin{definition}[Chart-commutation drift]
\label{def:chart-commutation-drift}
For \(k\in\Lambda_{\rm adj}\), define
\[
\begin{aligned}
    {\rm Drift}_k(\mathfrak D)
    &:=
    \Theta_{k+1}D_{k+1}^{\loc}
    -
    R_k^{\cl}\Theta_kD_k^{\loc}
    -
    \Theta_{k+1}\bigl(D_{k+1}^{\loc}-R_k^{\loc}D_k^{\loc}\bigr) \\
    &=
    \Theta_{k+1}R_k^{\loc}D_k^{\loc}
    -
    R_k^{\cl}\Theta_kD_k^{\loc}.
\end{aligned}
\]
The reproduction drift budget is
\[
    \Err_{\rm rep}(\mathfrak D)
    :=
    \left(
    \sum_{k\in\Lambda_{\rm adj}}
    \|{\rm Drift}_k(\mathfrak D)\|_{\calD_{k+1}^{\cl}}^2
    \right)^{1/2}.
\]
\end{definition}

\begin{lemma}[Finite-window reproduction residual comparison]
\label{lem:finite-window-reproduction-residual-comparison}
Let
\[
    C_{\rm rep}
    :=
    \max_{k\in\Lambda_{\rm adj}}
    \|\Theta_{k+1}\|_{\calD_{k+1}^{\loc}\to\calD_{k+1}^{\cl}},
\]
with \(C_{\rm rep}=0\) if \(\Lambda_{\rm adj}\) is empty.  Then
\[
    \Rep_{\rm cl}(\Theta\mathfrak D)
    \le
    C_{\rm rep}\Rep_{\rm loc}(\mathfrak D)
    +
    \Err_{\rm rep}(\mathfrak D).
\]
\end{lemma}

\begin{proof}
For each \(k\in\Lambda_{\rm adj}\), the definition of the drift gives
\[
\begin{aligned}
    \Theta_{k+1}D_{k+1}^{\loc}
    -
    R_k^{\cl}\Theta_kD_k^{\loc}
    =
    \Theta_{k+1}\bigl(D_{k+1}^{\loc}-R_k^{\loc}D_k^{\loc}\bigr)
    +
    {\rm Drift}_k(\mathfrak D).
\end{aligned}
\]
Taking the \(\ell^2(\Lambda_{\rm adj})\) norm of these clean residuals and
using the triangle inequality gives
\[
\begin{aligned}
    \Rep_{\rm cl}(\Theta\mathfrak D)
    &\le
    \left(
    \sum_{k\in\Lambda_{\rm adj}}
    \|\Theta_{k+1}(D_{k+1}^{\loc}-R_k^{\loc}D_k^{\loc})\|_{\calD_{k+1}^{\cl}}^2
    \right)^{1/2}
    +
    \Err_{\rm rep}(\mathfrak D) \\
    &\le
    C_{\rm rep}
    \left(
    \sum_{k\in\Lambda_{\rm adj}}
    \|D_{k+1}^{\loc}-R_k^{\loc}D_k^{\loc}\|_{\calD_{k+1}^{\loc}}^2
    \right)^{1/2}
    +
    \Err_{\rm rep}(\mathfrak D).
\end{aligned}
\]
The last displayed expression is
\(C_{\rm rep}\Rep_{\rm loc}(\mathfrak D)+\Err_{\rm rep}(\mathfrak D)\).
\end{proof}

\begin{remark}[Status of reproduction drift]
\Cref{lem:finite-window-reproduction-residual-comparison} is an algebraic
finite-window comparison for abstract reproduction maps.  It does not assert
that \(R_k^{\loc}\) or \(R_k^{\cl}\) is the dyadic Navier--Stokes rescaling
map, does not prove that the chart-commutation drift is small, and does not
provide scale-uniform reproduction control.  Any dyadic reproduction model or
small-drift theorem is a separate input.
\end{remark}

\subsection{Dyadic shift reproduction specialization}

\begin{convention}[Normalized adjacent dyadic rescaling]
\label{conv:normalized-adjacent-dyadic-rescaling}
In this subsection only, fix adjacent radii
\[
    r_{k+1}=2^{-1}r_k .
\]
If \((y,s)\) denotes normalized coordinates at scale \(k+1\), then the
corresponding normalized coordinates at scale \(k\) are
\[
    \iota_{k\to k+1}(y,s)
    =
    (2^{-1}y,2^{-2}s).
\]
Equivalently, the inverse expansion from the \(k\)-window coordinates to the
\((k+1)\)-window coordinates is \((Y,S)\mapsto(2Y,4S)\).  For a
velocity--pressure representative \((u,p)\) at scale \(k\), the normalized
dyadic zoom into the adjacent scale is
\[
    \mathcal Z_{k\to k+1}(u,p)(y,s)
    =
    \left(
    2^{-1}u(2^{-1}y,2^{-2}s),
    2^{-2}p(2^{-1}y,2^{-2}s)
    \right).
\]
This convention fixes only the model-level adjacent rescaling.  It does not
assert that the finite-window defect coordinates are generated by an actual
Navier--Stokes solution across scales.
\end{convention}

\begin{definition}[Local and clean dyadic shift operators]
\label{def:local-clean-dyadic-shift-operators}
A dyadic shift reproduction datum is the abstract reproduction datum of
\Cref{def:abstract-finite-window-reproduction-maps} together with the
rescaling convention of \Cref{conv:normalized-adjacent-dyadic-rescaling} and
bounded linear operators
\[
    S_k^{\loc}:\calD_k^{\loc}\to\calD_{k+1}^{\loc},
    \qquad
    S_k^{\cl}:\calD_k^{\cl}\to\calD_{k+1}^{\cl},
    \qquad k\in\Lambda_{\rm adj}.
\]
These operators are interpreted as the finite-window local and clean coordinate
realizations of the normalized adjacent dyadic zoom, after whatever projection
or coordinate extraction is part of the chosen local and clean models.  In the
dyadic specialization we set
\[
    R_k^{\loc}=S_k^{\loc},
    \qquad
    R_k^{\cl}=S_k^{\cl}.
\]
Thus this subsection specializes the abstract reproduction maps; it does not
replace the abstract finite-window framework.
\end{definition}

\begin{definition}[Dyadic reproduction residual and drift]
\label{def:dyadic-reproduction-residual-and-drift}
For a finite-window localized package
\(\mathfrak D^{\loc}=(D_k^{\loc})_{k\in\Lambda_{\rm sc}}\), define the local
dyadic reproduction residual by
\[
    \Rep_{\rm loc}^{\rm dy}(\mathfrak D)^2
    :=
    \sum_{k\in\Lambda_{\rm adj}}
    \|D_{k+1}^{\loc}
      -S_k^{\loc}D_k^{\loc}\|_{\calD_{k+1}^{\loc}}^2.
\]
The corresponding clean dyadic residual of the charted package is
\[
    \Rep_{\rm cl}^{\rm dy}(\Theta\mathfrak D)^2
    :=
    \sum_{k\in\Lambda_{\rm adj}}
    \|\Theta_{k+1}D_{k+1}^{\loc}
      -S_k^{\cl}\Theta_kD_k^{\loc}\|_{\calD_{k+1}^{\cl}}^2.
\]
The dyadic chart-commutation drift is
\[
    {\rm Drift}_k^{\rm dy}(\mathfrak D)
    :=
    \Theta_{k+1}S_k^{\loc}D_k^{\loc}
    -
    S_k^{\cl}\Theta_kD_k^{\loc}
    =
    (\Theta_{k+1}S_k^{\loc}-S_k^{\cl}\Theta_k)D_k^{\loc}.
\]
Set
\[
    \Err_{\rm rep}^{\rm dy}(\mathfrak D)
    :=
    \left(
    \sum_{k\in\Lambda_{\rm adj}}
    \|{\rm Drift}_k^{\rm dy}(\mathfrak D)\|_{\calD_{k+1}^{\cl}}^2
    \right)^{1/2}.
\]
\end{definition}

\begin{lemma}[Finite-window dyadic reproduction comparison]
\label{lem:finite-window-dyadic-reproduction-comparison}
For a fixed dyadic shift reproduction datum, let
\[
    C_{\Theta,{\rm dy}}
    :=
    \max_{k\in\Lambda_{\rm adj}}
    \|\Theta_{k+1}\|_{\calD_{k+1}^{\loc}\to\calD_{k+1}^{\cl}},
\]
with \(C_{\Theta,{\rm dy}}=0\) if \(\Lambda_{\rm adj}\) is empty.  Then
\[
    \Rep_{\rm cl}^{\rm dy}(\Theta\mathfrak D)
    \le
    C_{\Theta,{\rm dy}}\Rep_{\rm loc}^{\rm dy}(\mathfrak D)
    +
    \Err_{\rm rep}^{\rm dy}(\mathfrak D).
\]
\end{lemma}

\begin{proof}
For each \(k\in\Lambda_{\rm adj}\), add and subtract
\(\Theta_{k+1}S_k^{\loc}D_k^{\loc}\).  This gives
\[
\begin{aligned}
    \Theta_{k+1}D_{k+1}^{\loc}
    -
    S_k^{\cl}\Theta_kD_k^{\loc}
    &=
    \Theta_{k+1}
    \bigl(D_{k+1}^{\loc}-S_k^{\loc}D_k^{\loc}\bigr)
    +
    {\rm Drift}_k^{\rm dy}(\mathfrak D).
\end{aligned}
\]
Taking the \(\ell^2(\Lambda_{\rm adj})\) norm and using the triangle
inequality yields
\[
\begin{aligned}
    \Rep_{\rm cl}^{\rm dy}(\Theta\mathfrak D)
    &\le
    \left(
    \sum_{k\in\Lambda_{\rm adj}}
    \|\Theta_{k+1}(D_{k+1}^{\loc}
      -S_k^{\loc}D_k^{\loc})\|_{\calD_{k+1}^{\cl}}^2
    \right)^{1/2}
    +
    \Err_{\rm rep}^{\rm dy}(\mathfrak D) \\
    &\le
    C_{\Theta,{\rm dy}}
    \left(
    \sum_{k\in\Lambda_{\rm adj}}
    \|D_{k+1}^{\loc}
      -S_k^{\loc}D_k^{\loc}\|_{\calD_{k+1}^{\loc}}^2
    \right)^{1/2}
    +
    \Err_{\rm rep}^{\rm dy}(\mathfrak D).
\end{aligned}
\]
The last expression is the asserted bound.
\end{proof}

\begin{remark}[Status of the dyadic specialization]
\Cref{lem:finite-window-dyadic-reproduction-comparison} is a conditional
model specialization of the abstract reproduction comparison.  It proves only
a finite-window algebraic estimate after the local and clean dyadic shift
operators have been fixed.  It does not prove scale-uniform reproduction, does
not prove that the dyadic drift is small, and does not imply any
Navier--Stokes regularity statement.
\end{remark}

\subsection{Finite-dimensional dyadic drift bound}

\begin{definition}[Dyadic commutator operator]
\label{def:dyadic-commutator-operator}
In the linear finite-window dyadic shift datum, define
\[
    K_k^{\rm dy}
    :=
    \Theta_{k+1}S_k^{\loc}
    -
    S_k^{\cl}\Theta_k
    :
    \calD_k^{\loc}\to\calD_{k+1}^{\cl},
    \qquad k\in\Lambda_{\rm adj}.
\]
Thus the dyadic drift from
\Cref{def:dyadic-reproduction-residual-and-drift} is
\[
    {\rm Drift}_k^{\rm dy}(\mathfrak D)
    =
    K_k^{\rm dy}D_k^{\loc}.
\]
\end{definition}

\begin{lemma}[Finite-dimensional dyadic drift bound]
\label{lem:finite-dimensional-dyadic-drift-bound}
Assume that the component spaces
\(\calD_k^{\loc}\) and \(\calD_{k+1}^{\cl}\) are finite-dimensional normed
spaces and that the maps
\[
    \Theta_k,\quad \Theta_{k+1},\quad S_k^{\loc},\quad S_k^{\cl}
\]
appearing in the dyadic shift datum are bounded linear maps on their stated
domains.  Then, for every \(k\in\Lambda_{\rm adj}\),
\[
    \|K_k^{\rm dy}\|_{\calD_k^{\loc}\to\calD_{k+1}^{\cl}}<\infty .
\]
Consequently, if
\[
    C_{\rm rep}^{\rm dy}
    :=
    \max_{k\in\Lambda_{\rm adj}}
    \|K_k^{\rm dy}\|_{\calD_k^{\loc}\to\calD_{k+1}^{\cl}},
\]
with \(C_{\rm rep}^{\rm dy}=0\) when \(\Lambda_{\rm adj}\) is empty, then
\[
    \Err_{\rm rep}^{\rm dy}(\mathfrak D)
    \le
    C_{\rm rep}^{\rm dy}
    \left(
    \sum_{k\in\Lambda_{\rm adj}}
    \|D_k^{\loc}\|_{\calD_k^{\loc}}^2
    \right)^{1/2}.
\]
\end{lemma}

\begin{proof}
The operator \(K_k^{\rm dy}\) is the difference of the two linear maps
\[
    \Theta_{k+1}S_k^{\loc},
    \qquad
    S_k^{\cl}\Theta_k
    :
    \calD_k^{\loc}\to\calD_{k+1}^{\cl}.
\]
Each composition is bounded because it is a composition of bounded linear maps.
Hence \(K_k^{\rm dy}\) is a bounded linear map.  Since the spaces are
finite-dimensional, its operator norm is finite.

For each \(k\in\Lambda_{\rm adj}\), the definition of
\(C_{\rm rep}^{\rm dy}\) gives
\[
    \|{\rm Drift}_k^{\rm dy}(\mathfrak D)\|_{\calD_{k+1}^{\cl}}
    =
    \|K_k^{\rm dy}D_k^{\loc}\|_{\calD_{k+1}^{\cl}}
    \le
    C_{\rm rep}^{\rm dy}\|D_k^{\loc}\|_{\calD_k^{\loc}}.
\]
Squaring this inequality, summing over \(k\in\Lambda_{\rm adj}\), and taking
the square root gives the asserted finite-window bound for
\(\Err_{\rm rep}^{\rm dy}\).
\end{proof}

\begin{remark}[Status of the finite-dimensional dyadic drift bound]
\Cref{lem:finite-dimensional-dyadic-drift-bound} proves only that the dyadic
drift is bounded by the size of the finite-window localized package with a
constant depending on the chosen finite window, norms, charts, and dyadic shift
operators.  It does not prove that \(C_{\rm rep}^{\rm dy}\) is small, does not
prove scale-uniform control, and does not absorb \(\Err_{\rm rep}^{\rm dy}\)
into \(\eta_\Lambda\Dist_{\loc}+\Delta_\Lambda\).
\end{remark}

\subsection{Residual, cleaning, observability, reproduction, and profit}

The localized residual map \(\calE_\Lambda^{\loc}\) measures projected
divergence, momentum, pressure, ledger, and gluing compatibility.  The cleaning
map
\[
    G_\Lambda^{\loc}:\calC_\Lambda^{\loc}\to\calD_\Lambda^{\loc}
\]
removes only chosen cutoff, harmonic-gauge, recentring, coarse-graining, and
truncation artifacts.  The quotient distance is
\[
    \Dist_{\loc}(\mathfrak D,\Image G_\Lambda^{\loc})
    =
    \inf_{a\in\calC_\Lambda^{\loc}}
    \|\mathfrak D-G_\Lambda^{\loc}a\|_{\calD_\Lambda^{\loc}}.
\]
The reproduction residual and profit functional are denoted by
\(\Rep_\Lambda^{\loc}\) and \(\Prof_\Lambda^{\loc}\).

\section{Local-to-Clean Charts}\label{sec:chart}

\subsection{Chart datum}

\begin{definition}[Local-to-clean chart datum]
A local-to-clean chart datum consists of the following finite-window choices:
\begin{enumerate}[label=(\roman*),leftmargin=2em]
    \item a core restriction operator;
    \item a cutoff-flattening convention;
    \item a harmonic-gauge projection;
    \item a periodization operator;
    \item a clean Fourier/time projection;
    \item coordinate extraction maps into \(\calD_\Lambda^{\cl}\).
\end{enumerate}
A chart datum also fixes a chart domain
\[
    \calU_\Lambda^{\loc}\subseteq\calD_\Lambda^{\loc}
\]
and finite-dimensional normed intermediate spaces so that these operations
compose.  We write the successive maps schematically as
\[
\calU_\Lambda^{\loc}
\xrightarrow{\mathsf R_\Lambda}
X_1
\xrightarrow{\mathsf F_\Lambda}
X_2
\xrightarrow{\mathsf H_\Lambda}
X_3
\xrightarrow{\mathsf{Per}_\Lambda}
X_4
\xrightarrow{\mathsf P_\Lambda^{\cl}}
X_5
\xrightarrow{\mathsf E_\Lambda}
\calD_\Lambda^{\cl}.
\]
Here \(\mathsf R_\Lambda\) denotes core restriction,
\(\mathsf F_\Lambda\) cutoff flattening, \(\mathsf H_\Lambda\) harmonic-gauge
projection, \(\mathsf{Per}_\Lambda\) periodization,
\(\mathsf P_\Lambda^{\cl}\) clean finite-window projection, and
\(\mathsf E_\Lambda\) clean coordinate extraction.  Each map is part of the
chosen datum and is assumed continuous on its stated domain.  The resulting
chart is
\[
    \Theta_\Lambda:
    \calU_\Lambda^{\loc}
    \longrightarrow
    \calD_\Lambda^{\cl},
    \qquad
    \Theta_\Lambda
    =
    \mathsf E_\Lambda
    \mathsf P_\Lambda^{\cl}
    \mathsf{Per}_\Lambda
    \mathsf H_\Lambda
    \mathsf F_\Lambda
    \mathsf R_\Lambda .
\]
The datum is called linear if all six component maps are linear.
\end{definition}

\begin{lemma}[Bounded chart map]
\label[lemma]{lem:bounded-chart-map}
For a fixed local-to-clean chart datum, the map
\(\Theta_\Lambda:\calU_\Lambda^{\loc}\to\calD_\Lambda^{\cl}\) is a
well-defined finite-dimensional map on its chart domain.  If the chart datum
is linear and \(\calU_\Lambda^{\loc}=\calD_\Lambda^{\loc}\), then
\(\Theta_\Lambda\) is a bounded linear map.  More precisely, there is a
finite constant \(C_{\Theta,\Lambda}<\infty\), depending only on the fixed
chart datum and the chosen norms, such that
\[
    \|\Theta_\Lambda\mathfrak D\|_{\cl}
    \le
    C_{\Theta,\Lambda}
    \|\mathfrak D\|_{\loc}
    \qquad
    \text{for all }\mathfrak D\in\calD_\Lambda^{\loc}.
\]
\end{lemma}

\begin{proof}
The chart datum includes the chart domain, the intermediate spaces, and the
component maps with compatible domains and codomains.  Therefore the
composition
\[
    \Theta_\Lambda
    =
    \mathsf E_\Lambda
    \mathsf P_\Lambda^{\cl}
    \mathsf{Per}_\Lambda
    \mathsf H_\Lambda
    \mathsf F_\Lambda
    \mathsf R_\Lambda
\]
is defined for every \(\mathfrak D\in\calU_\Lambda^{\loc}\), and its values
lie in \(\calD_\Lambda^{\cl}\).  This proves well-definedness.

Assume now that the datum is linear and that
\(\calU_\Lambda^{\loc}=\calD_\Lambda^{\loc}\).  Since all spaces in the
finite-window chart are finite-dimensional normed spaces, every linear
component map has a finite operator norm.  Let
\[
    M_R=\|\mathsf R_\Lambda\|,
    \quad
    M_F=\|\mathsf F_\Lambda\|,
    \quad
    M_H=\|\mathsf H_\Lambda\|,
    \quad
    M_{\rm Per}=\|\mathsf{Per}_\Lambda\|,
    \quad
    M_P=\|\mathsf P_\Lambda^{\cl}\|,
    \quad
    M_E=\|\mathsf E_\Lambda\|.
\]
Then, for every \(\mathfrak D\in\calD_\Lambda^{\loc}\),
\[
\begin{aligned}
    \|\Theta_\Lambda\mathfrak D\|_{\cl}
    &=
    \|
    \mathsf E_\Lambda
    \mathsf P_\Lambda^{\cl}
    \mathsf{Per}_\Lambda
    \mathsf H_\Lambda
    \mathsf F_\Lambda
    \mathsf R_\Lambda
    \mathfrak D
    \|_{\cl}  \\
    &\le
    M_E M_P M_{\rm Per} M_H M_F M_R
    \|\mathfrak D\|_{\loc}.
\end{aligned}
\]
Thus the estimate holds with
\[
    C_{\Theta,\Lambda}
    =
    M_E M_P M_{\rm Per} M_H M_F M_R<\infty .
\]
The composition of linear maps is linear, so \(\Theta_\Lambda\) is a bounded
linear map.  This proves only boundedness of the fixed chart; it does not
prove quotient-distance stability, near-isometry, or small localized error.
\end{proof}

\subsection{Stable chart hypotheses}

\begin{assumption}[Stable chart bound]
\label[assumption]{ass:stable-chart-bound}
There are constants \(C_{\Theta,\Lambda}<\infty\),
\(\varepsilon_{\rm chart}\ge0\), and \(\Delta_{\rm chart}\ge0\) such that, for
all localized packages in the chart domain,
\[
    \|\Theta_\Lambda\mathfrak D\|_{\cl}
    \le
    C_{\Theta,\Lambda}\|\mathfrak D\|_{\loc}.
\]
When a perturbative normalization is available, this may be strengthened to
\[
    \|\Theta_\Lambda\mathfrak D\|_{\cl}
    \le
    (1+\varepsilon_{\rm chart})\|\mathfrak D\|_{\loc}
    +
    \Delta_{\rm chart}.
\]
\end{assumption}

\begin{lemma}[Bounded chart estimates and perturbative normalization]
\label[lemma]{lem:chart-bound-versus-normalization}
Assume the chart datum is linear and
\(\calU_\Lambda^{\loc}=\calD_\Lambda^{\loc}\).  Then the first estimate in
\Cref{ass:stable-chart-bound} holds with the constant
\(C_{\Theta,\Lambda}\) obtained in \Cref{lem:bounded-chart-map}.  The sharper
perturbative estimate
\[
    \|\Theta_\Lambda\mathfrak D\|_{\cl}
    \le
    (1+\varepsilon_{\rm chart})\|\mathfrak D\|_{\loc}
    +
    \Delta_{\rm chart}
\]
is an additional structural input unless it is proved from the concrete chart
geometry.  In particular, on a full vector-space chart domain, if the above
estimate holds for all \(\mathfrak D\), then
\[
    \|\Theta_\Lambda\|_{\loc\to\cl}
    \le
    1+\varepsilon_{\rm chart}.
\]
\end{lemma}

\begin{proof}
The bounded estimate is precisely the conclusion of
\Cref{lem:bounded-chart-map}.  It gives
\[
    \|\Theta_\Lambda\mathfrak D\|_{\cl}
    \le
    C_{\Theta,\Lambda}\|\mathfrak D\|_{\loc}
\]
for all \(\mathfrak D\in\calD_\Lambda^{\loc}\).

Now suppose that the perturbative estimate with additive constant
\(\Delta_{\rm chart}\) holds on the full vector space
\(\calD_\Lambda^{\loc}\).  Fix any nonzero direction
\(\mathfrak D\in\calD_\Lambda^{\loc}\) and apply the estimate to
\(t\mathfrak D\), \(t>0\).  Since the chart is linear,
\[
    t\|\Theta_\Lambda\mathfrak D\|_{\cl}
    =
    \|\Theta_\Lambda(t\mathfrak D)\|_{\cl}
    \le
    (1+\varepsilon_{\rm chart})t\|\mathfrak D\|_{\loc}
    +
    \Delta_{\rm chart}.
\]
Dividing by \(t\) and sending \(t\to\infty\) gives
\[
    \|\Theta_\Lambda\mathfrak D\|_{\cl}
    \le
    (1+\varepsilon_{\rm chart})\|\mathfrak D\|_{\loc}.
\]
Taking the supremum over \(\|\mathfrak D\|_{\loc}=1\) yields
\[
    \|\Theta_\Lambda\|_{\loc\to\cl}
    \le
    1+\varepsilon_{\rm chart}.
\]
Finite-dimensional boundedness alone only gives the finite constant
\(C_{\Theta,\Lambda}\); it does not imply that this constant is close to
\(1\).  Thus the perturbative normalization is a genuine structural
hypothesis unless supplied by a separate argument.
\end{proof}

\subsection{Gauge compatibility defect}

The chart must also be compared with the chosen local and clean cleaning
spaces.  Fix a clean finite-window cleaning map
\[
    G_\Lambda^{\cl}:\calC_\Lambda^{\cl}\to\calD_\Lambda^{\cl}.
\]
For every localized cleaning direction \(a\in\calC_\Lambda^{\loc}\) such that
\(G_\Lambda^{\loc}a\in\calU_\Lambda^{\loc}\), define the gauge compatibility
defect by
\begin{equation}\label{eq:gauge-compatibility-defect}
    \Gamma_\Lambda(a)
    :=
    \Dist_{\cl}
    \left(
    \Theta_\Lambda G_\Lambda^{\loc}a,
    \Image G_\Lambda^{\cl}
    \right).
\end{equation}
Thus \(\Gamma_\Lambda(a)\) measures how much of a localized cleaning direction
survives as a non-clean direction after applying the local-to-clean chart.  A
perfectly compatible chart would satisfy \(\Gamma_\Lambda(a)=0\) for all such
\(a\).

Define the unit-scale gauge compatibility constant
\begin{equation}\label{eq:gauge-compatibility-constant}
    \delta_{\rm gauge,\Lambda}
    :=
    \sup_{\substack{a\in\calC_\Lambda^{\loc}\\
    \|a\|_{\calC_\Lambda^{\loc}}\le1\\
    G_\Lambda^{\loc}a\in\calU_\Lambda^{\loc}}}
    \Gamma_\Lambda(a).
\end{equation}

\begin{lemma}[Finite gauge compatibility defect]
\label[lemma]{lem:gauge-compatibility-finite}
Assume that the chart datum is continuous on its chart domain, that
\(G_\Lambda^{\loc}\) and \(G_\Lambda^{\cl}\) are finite-dimensional linear
cleaning maps, and that the set
\[
    B_{\rm gauge}
    :=
    \left\{
    a\in\calC_\Lambda^{\loc}:
    \|a\|_{\calC_\Lambda^{\loc}}\le1,\ 
    G_\Lambda^{\loc}a\in\calU_\Lambda^{\loc}
    \right\}
\]
is closed in \(\calC_\Lambda^{\loc}\).  Then
\(\delta_{\rm gauge,\Lambda}<\infty\).  If \(B_{\rm gauge}\) is nonempty,
the supremum in \eqref{eq:gauge-compatibility-constant} is attained.
\end{lemma}

\begin{proof}
The space \(\calC_\Lambda^{\loc}\) is finite-dimensional.  Therefore the
closed unit ball is compact, and the closed subset \(B_{\rm gauge}\) is compact.
The localized cleaning map \(G_\Lambda^{\loc}\) is continuous, the chart
\(\Theta_\Lambda\) is continuous on its chart domain, and the clean cleaning
image \(\Image G_\Lambda^{\cl}\) is a finite-dimensional linear subspace of
\(\calD_\Lambda^{\cl}\), hence closed.  The distance to a closed subspace is a
continuous function.  Consequently \(a\mapsto\Gamma_\Lambda(a)\) is continuous
on \(B_{\rm gauge}\).

If \(B_{\rm gauge}\) is empty, the supremum is taken over the empty set and the
constant may be set to \(0\) by convention.  If \(B_{\rm gauge}\) is nonempty,
the extreme value theorem gives a finite maximum of \(\Gamma_\Lambda\) on
\(B_{\rm gauge}\).  This proves finiteness and attainment.
\end{proof}

\begin{remark}[Status of gauge compatibility]
The constant \(\delta_{\rm gauge,\Lambda}\) is only a finite-window defect
constant.  The lemma does not prove that it is small, scale-uniform, or
absorbable into the localized transfer budget.  Smallness of this quantity is a
separate structural input.
\end{remark}

\begin{definition}[Finite-dimensional gauge mismatch operator]
\label{def:finite-dimensional-gauge-mismatch-operator}
In the linear finite-window gauge datum, assume
\(\calU_\Lambda^{\loc}=\calD_\Lambda^{\loc}\), the chart is linear, and the
clean image \(\Image G_\Lambda^{\cl}\) is equipped with the quotient norm on
\[
    \calD_\Lambda^{\cl}/\Image G_\Lambda^{\cl}.
\]
Let
\[
    Q_\Lambda^{\cl}:
    \calD_\Lambda^{\cl}
    \to
    \calD_\Lambda^{\cl}/\Image G_\Lambda^{\cl}
\]
be the quotient map.  Define
\[
    K_{\rm gauge}
    :=
    Q_\Lambda^{\cl}\Theta_\Lambda G_\Lambda^{\loc}
    :
    \calC_\Lambda^{\loc}
    \to
    \calD_\Lambda^{\cl}/\Image G_\Lambda^{\cl}.
\]
Then, for each localized cleaning direction \(a\in\calC_\Lambda^{\loc}\),
\[
    \Gamma_\Lambda(a)
    =
    \|K_{\rm gauge}a\|_{\calD_\Lambda^{\cl}/\Image G_\Lambda^{\cl}}.
\]
\end{definition}

\begin{lemma}[Finite-dimensional selected gauge mismatch bound]
\label{lem:finite-dimensional-selected-gauge-mismatch-bound}
Assume the linear finite-window gauge datum of
\Cref{def:finite-dimensional-gauge-mismatch-operator}.  Then
\[
    \|K_{\rm gauge}\|_{\calC_\Lambda^{\loc}
    \to\calD_\Lambda^{\cl}/\Image G_\Lambda^{\cl}}<\infty .
\]
If, in addition, a bounded linear selected-cleaning map
\[
    A_\Lambda:\calD_\Lambda^{\loc}\to\calC_\Lambda^{\loc}
\]
is fixed and
\[
    \Err_{\rm gauge}^{\rm sel}(\mathfrak D)
    :=
    \Gamma_\Lambda(A_\Lambda\mathfrak D),
\]
then
\[
    \Err_{\rm gauge}^{\rm sel}(\mathfrak D)
    \le
    C_{\rm gauge}^{\rm fd}\|\mathfrak D\|_{\loc},
    \qquad
    C_{\rm gauge}^{\rm fd}
    :=
    \|K_{\rm gauge}\|\,\|A_\Lambda\|.
\]
\end{lemma}

\begin{proof}
The maps \(G_\Lambda^{\loc}\), \(\Theta_\Lambda\), and \(Q_\Lambda^{\cl}\) are
linear on finite-dimensional normed spaces in the stated datum.  Therefore the
composition
\[
    K_{\rm gauge}=Q_\Lambda^{\cl}\Theta_\Lambda G_\Lambda^{\loc}
\]
is a linear map between finite-dimensional normed spaces, and hence has finite
operator norm.

The quotient norm is by definition the distance to the clean gauge image:
\[
    \|Q_\Lambda^{\cl}d\|_{\calD_\Lambda^{\cl}/\Image G_\Lambda^{\cl}}
    =
    \Dist_{\cl}(d,\Image G_\Lambda^{\cl}).
\]
Taking \(d=\Theta_\Lambda G_\Lambda^{\loc}a\) gives
\[
    \|K_{\rm gauge}a\|
    =
    \Dist_{\cl}(\Theta_\Lambda G_\Lambda^{\loc}a,\Image G_\Lambda^{\cl})
    =
    \Gamma_\Lambda(a).
\]
For the selected cleaning \(a=A_\Lambda\mathfrak D\), boundedness of
\(K_{\rm gauge}\) and \(A_\Lambda\) gives
\[
    \Err_{\rm gauge}^{\rm sel}(\mathfrak D)
    =
    \|K_{\rm gauge}A_\Lambda\mathfrak D\|
    \le
    \|K_{\rm gauge}\|\,\|A_\Lambda\|\,\|\mathfrak D\|_{\loc}.
\]
This is the claimed finite-dimensional selected-gauge bound.
\end{proof}

\begin{remark}[Status of the selected gauge bound]
\Cref{lem:finite-dimensional-selected-gauge-mismatch-bound} bounds only the
gauge mismatch generated by a chosen selected-cleaning map.  It does not prove
that the full detection error \(\Err_{\rm gauge}\) is equal to
\(\Err_{\rm gauge}^{\rm sel}\), does not prove smallness, does not give
scale-uniform control, and does not absorb the gauge term into
\(\eta_\Lambda\Dist_{\loc}+\Delta_\Lambda\).
\end{remark}

The relation between the finite gauge defect and the later error budget is not
automatic.  The following elementary absorption lemma records one clean way in
which the constant \(\delta_{\rm gauge,\Lambda}\) may enter the normalized
gauge error.  It is included only to keep the bookkeeping honest; the
hypotheses below are structural inputs, not consequences of finite-dimensionality.

\begin{lemma}[Gauge defect absorption into the normalized budget]
\label[lemma]{lem:gauge-defect-absorption}
Assume that the chart domain is compatible with localized cleanings and that
there are constants \(L_{\rm gauge,\Lambda},M_{\rm gauge,\Lambda}\ge0\) and a
choice of localized cleaning parameter \(a_{\mathfrak D}\in\calC_\Lambda^{\loc}\)
for every \(\mathfrak D\in\calU_\Lambda^{\loc}\) such that
\[
    G_\Lambda^{\loc}a_{\mathfrak D}\in\calU_\Lambda^{\loc},
    \qquad
    \|a_{\mathfrak D}\|_{\calC_\Lambda^{\loc}}
    \le
    L_{\rm gauge,\Lambda}
    \Dist_{\loc}(\mathfrak D,\Image G_\Lambda^{\loc})
    +
    M_{\rm gauge,\Lambda}.
\]
Assume also that the gauge contribution in the detection comparison satisfies
\[
    \Err_{\rm gauge}(\mathfrak D)
    \le
    \Gamma_\Lambda(a_{\mathfrak D})
    +
    \eta_{\rm gauge}^{0}
    \Dist_{\loc}(\mathfrak D,\Image G_\Lambda^{\loc})
    +
    \Delta_{\rm gauge}^{0} .
\]
If, in addition, the gauge defect is homogeneous along the selected cleaning
parameters, so that
\[
    \Gamma_\Lambda(a_{\mathfrak D})
    \le
    \delta_{\rm gauge,\Lambda}
    \|a_{\mathfrak D}\|_{\calC_\Lambda^{\loc}},
\]
then
\[
    \Err_{\rm gauge}(\mathfrak D)
    \le
    \eta_{\rm gauge}
    \Dist_{\loc}(\mathfrak D,\Image G_\Lambda^{\loc})
    +
    \Delta_{\rm gauge},
\]
where
\[
    \eta_{\rm gauge}
    =
    \delta_{\rm gauge,\Lambda}L_{\rm gauge,\Lambda}
    +
    \eta_{\rm gauge}^{0},
    \qquad
    \Delta_{\rm gauge}
    =
    \delta_{\rm gauge,\Lambda}M_{\rm gauge,\Lambda}
    +
    \Delta_{\rm gauge}^{0}.
\]
\end{lemma}

\begin{proof}
The assumptions give
\[
\begin{aligned}
    \Err_{\rm gauge}(\mathfrak D)
    &\le
    \Gamma_\Lambda(a_{\mathfrak D})
    +
    \eta_{\rm gauge}^{0}
    \Dist_{\loc}(\mathfrak D,\Image G_\Lambda^{\loc})
    +
    \Delta_{\rm gauge}^{0}  \\
    &\le
    \delta_{\rm gauge,\Lambda}
    \|a_{\mathfrak D}\|_{\calC_\Lambda^{\loc}}
    +
    \eta_{\rm gauge}^{0}
    \Dist_{\loc}(\mathfrak D,\Image G_\Lambda^{\loc})
    +
    \Delta_{\rm gauge}^{0}  \\
    &\le
    (\delta_{\rm gauge,\Lambda}L_{\rm gauge,\Lambda}
    +
    \eta_{\rm gauge}^{0})
    \Dist_{\loc}(\mathfrak D,\Image G_\Lambda^{\loc})
    +
    \delta_{\rm gauge,\Lambda}M_{\rm gauge,\Lambda}
    +
    \Delta_{\rm gauge}^{0}.
\end{aligned}
\]
This is exactly the claimed normalized gauge-error bound.
\end{proof}

\begin{remark}[Gauge constants versus quotient lifting]
The constants \(\delta_{\rm gauge,\Lambda}\), \(\eta_{\rm gauge}\), and
\(\delta_G\) play different roles.  The first measures how localized cleaning
directions survive under the chart; the second is the part of the detection
error budget assigned to that mismatch; the third is an additive loss in the
quotient-distance comparison.  None of them is automatically small.  In a
concrete localized Navier--Stokes model one must prove separate estimates
showing that they are either absorbable in \(\eta_\Lambda\) or harmless inside
\(\Delta_\Lambda'\).
\end{remark}

\begin{assumption}[Quotient-lifting stability]
\label[assumption]{ass:quotient-lifting-stability}
There are constants \(0\le\varepsilon_G<1\) and \(\delta_G\ge0\) such that the
following lifting property holds.  For every
\(\mathfrak D\in\calU_\Lambda^{\loc}\) and every
\(b\in\calC_\Lambda^{\cl}\), there exists \(a\in\calC_\Lambda^{\loc}\) with
\[
    \|\mathfrak D-G_\Lambda^{\loc}a\|_{\loc}
    \le
    \frac{1}{1-\varepsilon_G}
    \|\Theta_\Lambda\mathfrak D-G_\Lambda^{\cl}b\|_{\cl}
    +
    \frac{\delta_G}{1-\varepsilon_G}.
\]
\end{assumption}

\begin{lemma}[Quotient-distance comparison from lifting]
\label[lemma]{lem:quotient-distance-comparison}
Assume \Cref{ass:quotient-lifting-stability}.  Then every
\(\mathfrak D\in\calU_\Lambda^{\loc}\) satisfies
\[
    \Dist_{\cl}
    (
    \Theta_\Lambda\mathfrak D,\Image G_\Lambda^{\cl}
    )
    \ge
    (1-\varepsilon_G)
    \Dist_{\loc}
    (
    \mathfrak D,\Image G_\Lambda^{\loc}
    )
    -
    \delta_G .
\]
\end{lemma}

\begin{proof}
Fix \(\mathfrak D\in\calU_\Lambda^{\loc}\) and
\(b\in\calC_\Lambda^{\cl}\).  By
\Cref{ass:quotient-lifting-stability}, there is
\(a\in\calC_\Lambda^{\loc}\) such that
\[
    \|\mathfrak D-G_\Lambda^{\loc}a\|_{\loc}
    \le
    \frac{1}{1-\varepsilon_G}
    \|\Theta_\Lambda\mathfrak D-G_\Lambda^{\cl}b\|_{\cl}
    +
    \frac{\delta_G}{1-\varepsilon_G}.
\]
Since the localized quotient distance is the infimum over all localized
cleaning corrections,
\[
    \Dist_{\loc}(\mathfrak D,\Image G_\Lambda^{\loc})
    \le
    \frac{1}{1-\varepsilon_G}
    \|\Theta_\Lambda\mathfrak D-G_\Lambda^{\cl}b\|_{\cl}
    +
    \frac{\delta_G}{1-\varepsilon_G}.
\]
Multiplying by \(1-\varepsilon_G>0\) gives
\[
    (1-\varepsilon_G)
    \Dist_{\loc}(\mathfrak D,\Image G_\Lambda^{\loc})
    -
    \delta_G
    \le
    \|\Theta_\Lambda\mathfrak D-G_\Lambda^{\cl}b\|_{\cl}.
\]
Taking the infimum over \(b\in\calC_\Lambda^{\cl}\) proves the stated
quotient-distance comparison.
\end{proof}

\begin{remark}[Status of quotient lifting]
The lifting assumption is the central structural input behind quotient-distance
stability.  It is not a consequence of boundedness of
\(\Theta_\Lambda\) or of finiteness of the gauge defect.  It asserts that clean
gauge corrections can be represented, up to the explicit loss
\((\varepsilon_G,\delta_G)\), by localized cleaning corrections.  Concrete
Navier--Stokes estimates for this lifting remain outside the finite-window
argument.
\end{remark}

\section{Detection Comparisons and Error Budgets}\label{sec:detection-error}

\subsection{Pressure component comparison}

\begin{assumption}[Pressure observation compatibility]
\label{ass:pressure-observation-compatibility}
For each localized package \(\mathfrak D\in\calU_\Lambda^{\loc}\), the
pressure observation channel is compatible with the pressure-transfer datum in
the following sense:
\[
    \|O_{\Lambda,{\rm prs}}^{\loc}\mathfrak D
      -
      O_{\Lambda,{\rm prs}}^{\cl}\Theta_\Lambda\mathfrak D\|
    \le
    \Err_{\rm prs}(\mathfrak D),
\]
where \(\Err_{\rm prs}\) is the budget from
\Cref{def:pressure-error-budget}.  This is a structural compatibility
assumption between the selected pressure observations, pressure normalization,
harmonic cleaning convention, and local-to-clean chart.
\end{assumption}

\begin{lemma}[Pressure detection comparison]
\label{lem:pressure-detection-comparison}
Assume \Cref{ass:pressure-observation-compatibility}.  Then each
\(\mathfrak D\in\calU_\Lambda^{\loc}\) satisfies
\[
    \|O_{\Lambda,{\rm prs}}^{\loc}\mathfrak D\|
    \ge
    \|O_{\Lambda,{\rm prs}}^{\cl}\Theta_\Lambda\mathfrak D\|
    -
    \Err_{\rm prs}(\mathfrak D).
\]
\end{lemma}

\begin{proof}
The triangle inequality gives
\[
    \|O_{\Lambda,{\rm prs}}^{\cl}\Theta_\Lambda\mathfrak D\|
    \le
    \|O_{\Lambda,{\rm prs}}^{\loc}\mathfrak D\|
    +
    \|O_{\Lambda,{\rm prs}}^{\loc}\mathfrak D
      -
      O_{\Lambda,{\rm prs}}^{\cl}\Theta_\Lambda\mathfrak D\|.
\]
Using \Cref{ass:pressure-observation-compatibility} and rearranging proves the
claim.
\end{proof}

\begin{assumption}[Pressure component normalized bounds]
\label{ass:pressure-component-normalized-bounds}
For every pressure component
\[
    r\in
    \{
    {\rm harm\mbox{-}tail},
    {\rm comm},
    {\rm active},
    {\rm proj},
    {\rm mean},
    {\rm per}
    \}
\]
and every \(k\in\Lambda_{\rm sc}\), there are constants
\(\eta_{{\rm prs},r,k}\ge0\) and \(\Delta_{{\rm prs},r,k}\ge0\) such that
\[
    \omega_{{\rm prs},k}\Err_{r,k}(\mathfrak D)
    \le
    \eta_{{\rm prs},r,k}
    \Dist_{\loc}(\mathfrak D,\Image G_\Lambda^{\loc})
    +
    \Delta_{{\rm prs},r,k}.
\]
\end{assumption}

\begin{lemma}[Pressure error absorption]
\label{lem:pressure-error-absorption}
Assume \Cref{ass:pressure-component-normalized-bounds}.  Define
\[
    \eta_{\rm prs}
    :=
    \sum_{k\in\Lambda_{\rm sc}}\sum_r\eta_{{\rm prs},r,k},
    \qquad
    \Delta_{\rm prs}
    :=
    \sum_{k\in\Lambda_{\rm sc}}\sum_r\Delta_{{\rm prs},r,k}.
\]
Then
\[
    \Err_{\rm prs}(\mathfrak D)
    \le
    \eta_{\rm prs}
    \Dist_{\loc}(\mathfrak D,\Image G_\Lambda^{\loc})
    +
    \Delta_{\rm prs}.
\]
\end{lemma}

\begin{proof}
Insert the pressure budget from \Cref{def:pressure-error-budget} and apply
\Cref{ass:pressure-component-normalized-bounds} to each weighted component.
Summing the resulting inequalities gives the displayed bound.
\end{proof}

\begin{remark}[Pressure obstruction alternatives]
The pressure module leaves three honest possibilities.  First, harmonic and
commutator directions may be compatible with the chosen cleaning space and
enter the budget as absorbable errors.  Second, they may be controlled but not
small, in which case they remain in \(\Delta_{\rm prs}\) rather than in an
absorbed coefficient.  Third, they may produce pressure-observable near-kernel
directions; those directions should then be treated as genuine pressure
phantom candidates, not hidden as gauge.
\end{remark}

\subsection{Componentwise comparison}

For a localized package in the chart domain, set
\[
    \Err_\Lambda(\mathfrak D)
    =
    \Err_{\rm prs}
    +
    \Err_{\rm loc}
    +
    \Err_{\rm tr}
    +
    \Err_{\rm nl}
    +
    \Err_{\rm rep}
    +
    \Err_{\rm gauge}
    +
    \Err_{\rm prof}.
\]
The individual terms record, respectively, pressure splitting error,
localization leakage, trace loss, nonlinear or residual mismatch, reproduction
drift, gauge mismatch, and profit discrepancy.

\begin{definition}[Concrete residual-budget instantiation]
\label{def:concrete-residual-budget-instantiation}
In the concrete finite-window model developed above, instantiate the nonlinear
entry as
\[
    \Err_{\rm nl}(\mathfrak D)
    :=
    \Err_{\rm nl}^{\rm cut}(\mathfrak D)
    +
    \Err_{\rm nl}^{\rm rem}(\mathfrak D),
\]
where \(\Err_{\rm nl}^{\rm rem}\ge0\) records nonlinear residuals not accounted
for by the cutoff-flux mismatch.  Define the concrete normalized residual
budget by
\[
\begin{aligned}
    \mathcal B_\Lambda^{\rm conc}(\mathfrak D)
    :=
    &\mathcal B_{\rm prs}^{\rm norm}(\mathfrak D)
    +
    \mathcal B_{\rm loc}^{\rm norm}(\mathfrak D)
    +
    \mathcal B_{\rm tr}^{\rm norm}(\mathfrak D)
    +
    \mathcal B_{\rm nl}^{\rm cut}(\mathfrak D) \\
    &+
    \Err_{\rm nl}^{\rm rem}(\mathfrak D)
    +
    \Err_{\rm rep}(\mathfrak D)
    +
    \Err_{\rm gauge}(\mathfrak D)
    +
    \Err_{\rm prof}(\mathfrak D).
\end{aligned}
\]
Here \(\mathcal B_{\rm prs}^{\rm norm}\),
\(\mathcal B_{\rm loc}^{\rm norm}\), \(\mathcal B_{\rm tr}^{\rm norm}\), and
\(\mathcal B_{\rm nl}^{\rm cut}\) are the pressure, localization,
truncation, and cutoff-nonlinear budgets defined in the preceding sections.
\end{definition}

\begin{lemma}[Concrete residual-budget synchronization]
\label{lem:concrete-residual-budget-synchronization}
Assume that the pressure, localization, truncation, nonlinear cutoff, and
reproduction entries are instantiated by
\[
    \mathcal B_{\rm prs}^{\rm norm},\qquad
    \mathcal B_{\rm loc}^{\rm norm},\qquad
    \mathcal B_{\rm tr}^{\rm norm},\qquad
    \mathcal B_{\rm nl}^{\rm cut},\qquad
    \Err_{\rm rep},
\]
as in
\Cref{prop:combined-normalized-pressure-component-package},
\Cref{lem:finite-window-localization-budget},
\Cref{lem:finite-window-truncation-leakage-bound},
\Cref{lem:finite-window-nonlinear-cutoff-budget}, and
\Cref{lem:finite-window-reproduction-residual-comparison}.  Then
\[
    \Err_\Lambda(\mathfrak D)
    \le
    \mathcal B_\Lambda^{\rm conc}(\mathfrak D).
\]
\end{lemma}

\begin{proof}
The pressure component is bounded by
\(\mathcal B_{\rm prs}^{\rm norm}\) by
\Cref{prop:combined-normalized-pressure-component-package}.  The localization,
truncation, and cutoff-nonlinear components are bounded by
\Cref{lem:finite-window-localization-budget},
\Cref{lem:finite-window-truncation-leakage-bound}, and
\Cref{lem:finite-window-nonlinear-cutoff-budget}, respectively.  The
reproduction entry is the drift budget appearing in
\Cref{lem:finite-window-reproduction-residual-comparison}.  By definition,
\[
    \Err_{\rm nl}
    =
    \Err_{\rm nl}^{\rm cut}
    +
    \Err_{\rm nl}^{\rm rem}.
\]
Substituting these bounds into the definition of
\(\Err_\Lambda(\mathfrak D)\) and leaving the gauge and profit entries
explicit gives exactly
\[
    \Err_\Lambda(\mathfrak D)
    \le
    \mathcal B_\Lambda^{\rm conc}(\mathfrak D).
\]
\end{proof}

\begin{remark}[Status of the concrete residual budget]
\Cref{lem:concrete-residual-budget-synchronization} is only a synchronization
of already defined error components.  It does not prove that
\(\mathcal B_\Lambda^{\rm conc}\) is small, does not absorb it into the
anti-phantom gap, and does not remove the remaining nonlinear, gauge, or profit
errors.  Any estimate of
\(\mathcal B_\Lambda^{\rm conc}\le
\eta_\Lambda\Dist_{\loc}+\Delta_\Lambda\) is an additional perturbative input.
\end{remark}

\begin{assumption}[Bounded finite-window package]
\label{ass:bounded-finite-window-package}
There is a finite constant \(R_\Lambda<\infty\) such that every localized
package under consideration satisfies the package-size bounds
\[
    \|\mathfrak D\|_{\loc}\le R_\Lambda,
    \qquad
    \left(
    \sum_{k\in\Lambda_{\rm sc}}\|D_k^{\loc}\|_{\calD_k^{\loc}}^2
    \right)^{1/2}
    \le R_\Lambda,
\]
and
\[
    \left(
    \sum_{k\in\Lambda_{\rm adj}}\|D_k^{\loc}\|_{\calD_k^{\loc}}^2
    \right)^{1/2}
    \le R_\Lambda .
\]
\end{assumption}

\begin{definition}[Finite-dimensional additive residual budget]
\label{def:finite-dimensional-additive-residual-budget}
In the finite-dimensional residual model where the reproduction, gauge,
profit, and nonlinear leftover terms are represented by
\[
    \Err_{\rm rep}^{\rm dy},\qquad
    \Err_{\rm gauge}^{\rm sel},\qquad
    \Err_{\rm prof},\qquad
    \Err_{\rm nl}^{\rm rem},
\]
define
\[
    \Delta_\Lambda^{\rm fd}
    :=
    R_\Lambda
    \left(
    C_{\rm rep}^{\rm dy}
    +
    C_{\rm gauge}^{\rm fd}
    +
    C_{\rm prof}^{\rm fd}
    +
    C_{\rm nl}^{\rm rem}
    \right).
\]
\end{definition}

\begin{lemma}[Bounded-window additive residual audit]
\label{lem:bounded-window-additive-residual-audit}
Assume \Cref{ass:bounded-finite-window-package} and use the finite-dimensional
residual models from
\Cref{lem:finite-dimensional-dyadic-drift-bound},
\Cref{lem:finite-dimensional-selected-gauge-mismatch-bound},
\Cref{lem:finite-dimensional-profit-discrepancy-bound}, and
\Cref{lem:finite-dimensional-nonlinear-remainder-bound}.  Then
\[
    \Err_{\rm rep}^{\rm dy}(\mathfrak D)
    +
    \Err_{\rm gauge}^{\rm sel}(\mathfrak D)
    +
    \Err_{\rm prof}(\mathfrak D)
    +
    \Err_{\rm nl}^{\rm rem}(\mathfrak D)
    \le
    \Delta_\Lambda^{\rm fd}.
\]
\end{lemma}

\begin{proof}
By \Cref{lem:finite-dimensional-dyadic-drift-bound} and
\Cref{ass:bounded-finite-window-package},
\[
    \Err_{\rm rep}^{\rm dy}(\mathfrak D)
    \le
    C_{\rm rep}^{\rm dy}R_\Lambda .
\]
By \Cref{lem:finite-dimensional-selected-gauge-mismatch-bound},
\[
    \Err_{\rm gauge}^{\rm sel}(\mathfrak D)
    \le
    C_{\rm gauge}^{\rm fd}R_\Lambda .
\]
By \Cref{lem:finite-dimensional-profit-discrepancy-bound},
\[
    \Err_{\rm prof}(\mathfrak D)
    \le
    C_{\rm prof}^{\rm fd}R_\Lambda .
\]
Finally, by \Cref{lem:finite-dimensional-nonlinear-remainder-bound} and
\Cref{ass:bounded-finite-window-package},
\[
    \Err_{\rm nl}^{\rm rem}(\mathfrak D)
    \le
    C_{\rm nl}^{\rm rem}R_\Lambda .
\]
Summing these four estimates gives the displayed bound.
\end{proof}

\begin{remark}[Status of the additive residual audit]
\Cref{lem:bounded-window-additive-residual-audit} places the selected
finite-dimensional residuals into an additive finite-window budget under the
explicit bounded-package hypothesis.  It does not prove that these terms are
small, does not make the constants scale-uniform, and does not absorb any term
into \(\eta_\Lambda\Dist_{\loc}+\Delta_\Lambda\).  In particular, replacing
the selected gauge term by the full \(\Err_{\rm gauge}\) or replacing the
model nonlinear remainder by the full Navier--Stokes nonlinear residual would
require additional hypotheses.
\end{remark}

\begin{assumption}[Componentwise detection estimates]
\label[assumption]{ass:componentwise-detection-estimates}
For each localized package \(\mathfrak D\in\calU_\Lambda^{\loc}\), the clean
and localized detection channels satisfy
\[
    \|O_\Lambda^{\loc}\mathfrak D\|
    \ge
    \|O_\Lambda^{\cl}\Theta_\Lambda\mathfrak D\|
    -
    (
    \Err_{\rm prs}
    +
    \Err_{\rm loc}
    +
    \Err_{\rm tr}
    +
    \Err_{\rm gauge}
    ),
\]
\[
    C_E\|\calE_\Lambda^{\loc}(\mathfrak D)\|
    \ge
    C_E\|\calE_\Lambda^{\cl}(\Theta_\Lambda\mathfrak D)\|
    -
    \Err_{\rm nl},
\]
\[
    C_R\Rep_\Lambda^{\loc}(\mathfrak D)
    \ge
    C_R\Rep_\Lambda^{\cl}(\Theta_\Lambda\mathfrak D)
    -
    \Err_{\rm rep},
\]
and
\[
    [\Prof_\Lambda^{\loc}(\mathfrak D)]_+
    \ge
    [\Prof_\Lambda^{\cl}(\Theta_\Lambda\mathfrak D)]_+
    -
    \Err_{\rm prof}.
\]
\end{assumption}

\begin{lemma}[Detection-map comparison]
\label[lemma]{lem:detection-map-comparison}
Assume \Cref{ass:componentwise-detection-estimates}.  Then each
\(\mathfrak D\in\calU_\Lambda^{\loc}\) satisfies
\[
    \|\calF_\Lambda^{\loc}(\mathfrak D)\|_{\loc}
    \ge
    \|\calF_\Lambda^{\cl}(\Theta_\Lambda\mathfrak D)\|_{\cl}
    -
    \Err_\Lambda(\mathfrak D).
\]
\end{lemma}

\begin{proof}
The detection norm is the weighted sum of the observability, residual,
reproduction, and positive-profit components appearing in
\(\calF_\Lambda\).  Summing the four inequalities in
\Cref{ass:componentwise-detection-estimates} gives
\[
\begin{aligned}
    \|\calF_\Lambda^{\loc}(\mathfrak D)\|_{\loc}
    &\ge
    \|\calF_\Lambda^{\cl}(\Theta_\Lambda\mathfrak D)\|_{\cl}
    -
    \Err_{\rm prs}
    -
    \Err_{\rm loc}
    -
    \Err_{\rm tr}
    -
    \Err_{\rm gauge}      \\
    &\qquad
    -
    \Err_{\rm nl}
    -
    \Err_{\rm rep}
    -
    \Err_{\rm prof}.
\end{aligned}
\]
The error on the right-hand side is exactly
\(\Err_\Lambda(\mathfrak D)\).  This proves the comparison.
\end{proof}

\begin{remark}[Status of the detection comparison]
\Cref{lem:detection-map-comparison} is an algebraic consequence of the
component estimates.  It does not prove the pressure, localization, trace,
nonlinear, reproduction, gauge, or profit estimates themselves.
\end{remark}

\subsection{Profit comparison}\label{subsec:profit-comparison}

\begin{definition}[Profit discrepancy residual]
\label{def:profit-discrepancy-residual}
For a localized package \(\mathfrak D\in\calU_\Lambda^{\loc}\), define the
finite-window profit discrepancy by
\[
    \Err_{\rm prof}(\mathfrak D)
    :=
    |\Prof_\Lambda^{\loc}(\mathfrak D)
    -
    \Prof_\Lambda^{\cl}(\Theta_\Lambda\mathfrak D)|.
\]
This is the exact scalar mismatch between the localized profit assigned to
\(\mathfrak D\) and the clean profit assigned to its charted image.
\end{definition}

\begin{lemma}[Positive-part profit comparison]
\label[lemma]{lem:profit-positive-part}
With \(\Err_{\rm prof}\) defined as in
\Cref{def:profit-discrepancy-residual},
\[
    |[\Prof_\Lambda^{\loc}(\mathfrak D)]_+
    -
    [\Prof_\Lambda^{\cl}(\Theta_\Lambda\mathfrak D)]_+|
    \le
    \Err_{\rm prof}(\mathfrak D).
\]
\end{lemma}

\begin{proof}
The map \(x\mapsto[x]_+=\max\{x,0\}\) is \(1\)-Lipschitz on \(\R\).  Applying
this to the two scalar profit values gives the claim.
\end{proof}

\begin{remark}[Status of the profit discrepancy]
\Cref{def:profit-discrepancy-residual} only instantiates the profit error as
the exact finite-window scalar mismatch.  It does not prove that this mismatch
is small, scale-uniform, or absorbable into the normalized error budget.  Any
bound of the form
\[
    \Err_{\rm prof}(\mathfrak D)
    \le
    \eta_{\rm prof}\Dist_{\loc}(\mathfrak D,\Image G_\Lambda^{\loc})
    +
    \Delta_{\rm prof}
\]
is a separate structural or PDE-facing estimate.
\end{remark}

\begin{definition}[Finite-dimensional profit discrepancy operator]
\label{def:finite-dimensional-profit-discrepancy-operator}
In the linear finite-window profit datum, assume
\(\calU_\Lambda^{\loc}=\calD_\Lambda^{\loc}\), the chart
\(\Theta_\Lambda:\calD_\Lambda^{\loc}\to\calD_\Lambda^{\cl}\) is linear, and
the localized and clean profit maps are bounded linear scalar functionals
\[
    \Prof_\Lambda^{\loc}:\calD_\Lambda^{\loc}\to\R,
    \qquad
    \Prof_\Lambda^{\cl}:\calD_\Lambda^{\cl}\to\R .
\]
Define the profit discrepancy operator
\[
    K_{\rm prof}
    :=
    \Prof_\Lambda^{\loc}
    -
    \Prof_\Lambda^{\cl}\Theta_\Lambda
    :
    \calD_\Lambda^{\loc}\to\R.
\]
Then
\[
    \Err_{\rm prof}(\mathfrak D)
    =
    |K_{\rm prof}\mathfrak D|.
\]
\end{definition}

\begin{lemma}[Finite-dimensional profit discrepancy bound]
\label{lem:finite-dimensional-profit-discrepancy-bound}
Under the linear finite-window profit datum of
\Cref{def:finite-dimensional-profit-discrepancy-operator},
\[
    \|K_{\rm prof}\|_{\calD_\Lambda^{\loc}\to\R}<\infty .
\]
Consequently, with
\[
    C_{\rm prof}^{\rm fd}
    :=
    \|K_{\rm prof}\|_{\calD_\Lambda^{\loc}\to\R},
\]
every localized package \(\mathfrak D\in\calD_\Lambda^{\loc}\) satisfies
\[
    \Err_{\rm prof}(\mathfrak D)
    \le
    C_{\rm prof}^{\rm fd}\|\mathfrak D\|_{\loc}.
\]
In particular,
\[
    C_{\rm prof}^{\rm fd}
    \le
    \|\Prof_\Lambda^{\loc}\|_{\calD_\Lambda^{\loc}\to\R}
    +
    \|\Prof_\Lambda^{\cl}\|_{\calD_\Lambda^{\cl}\to\R}
    \|\Theta_\Lambda\|_{\calD_\Lambda^{\loc}\to\calD_\Lambda^{\cl}} .
\]
\end{lemma}

\begin{proof}
The map \(K_{\rm prof}\) is the difference of the two linear scalar maps
\[
    \Prof_\Lambda^{\loc},
    \qquad
    \Prof_\Lambda^{\cl}\Theta_\Lambda:
    \calD_\Lambda^{\loc}\to\R .
\]
The first is bounded by assumption.  The second is bounded because it is the
composition of the bounded linear chart \(\Theta_\Lambda\) and the bounded
linear clean profit functional \(\Prof_\Lambda^{\cl}\).  Hence
\(K_{\rm prof}\) is a bounded linear scalar functional.  Since the domain is
finite-dimensional, its operator norm is finite.

For every \(\mathfrak D\in\calD_\Lambda^{\loc}\),
\[
    \Err_{\rm prof}(\mathfrak D)
    =
    |K_{\rm prof}\mathfrak D|
    \le
    \|K_{\rm prof}\|_{\calD_\Lambda^{\loc}\to\R}
    \|\mathfrak D\|_{\loc}.
\]
This proves the displayed bound with
\(C_{\rm prof}^{\rm fd}=\|K_{\rm prof}\|\).  The triangle inequality and
submultiplicativity of operator norms also give
\[
\begin{aligned}
    \|K_{\rm prof}\|
    &\le
    \|\Prof_\Lambda^{\loc}\|
    +
    \|\Prof_\Lambda^{\cl}\Theta_\Lambda\|  \\
    &\le
    \|\Prof_\Lambda^{\loc}\|
    +
    \|\Prof_\Lambda^{\cl}\|\,\|\Theta_\Lambda\|.
\end{aligned}
\]
\end{proof}

\begin{remark}[Status of the finite-dimensional profit bound]
\Cref{lem:finite-dimensional-profit-discrepancy-bound} is a boundedness result
for a fixed finite-window linear profit datum.  It does not prove that
\(C_{\rm prof}^{\rm fd}\) is small, does not prove scale-uniform control, and
does not absorb \(\Err_{\rm prof}\) into
\(\eta_\Lambda\Dist_{\loc}+\Delta_\Lambda\).  It also does not identify the
profit discrepancy with a Navier--Stokes energy-production mechanism; that
would require a separate PDE-facing definition of the profit functional.
\end{remark}

\subsection{Normalized error budget}

\begin{assumption}[Component normalized error bounds]
\label[assumption]{ass:component-normalized-error-bounds}
For each component
\[
    q\in
    \{
    {\rm prs},{\rm loc},{\rm tr},{\rm nl},{\rm rep},{\rm gauge},{\rm prof}
    \},
\]
there are constants \(\eta_q\ge0\) and \(\Delta_q\ge0\) such that
\[
    \Err_q(\mathfrak D)
    \le
    \eta_q
    \Dist_{\loc}(\mathfrak D,\Image G_\Lambda^{\loc})
    +
    \Delta_q
\]
for every localized package in the chart domain.  Define
\[
    \eta_\Lambda
    :=
    \sum_q\eta_q,
    \qquad
    \Delta_\Lambda
    :=
    \sum_q\Delta_q .
\]
\end{assumption}

\begin{lemma}[Error-budget normalization]
\label[lemma]{lem:error-budget-normalization}
Assume \Cref{ass:component-normalized-error-bounds}.  Then
\[
    \Err_\Lambda(\mathfrak D)
    \le
    \eta_\Lambda
    \Dist_{\loc}
    (
    \mathfrak D,\Image G_\Lambda^{\loc}
    )
    +
    \Delta_\Lambda .
\]
\end{lemma}

\begin{proof}
By definition,
\[
    \Err_\Lambda(\mathfrak D)
    =
    \sum_q\Err_q(\mathfrak D),
    \qquad
    q\in
    \{
    {\rm prs},{\rm loc},{\rm tr},{\rm nl},{\rm rep},{\rm gauge},{\rm prof}
    \}.
\]
Applying \Cref{ass:component-normalized-error-bounds} to each component and
summing gives
\[
\begin{aligned}
    \Err_\Lambda(\mathfrak D)
    &\le
    \left(\sum_q\eta_q\right)
    \Dist_{\loc}(\mathfrak D,\Image G_\Lambda^{\loc})
    +
    \sum_q\Delta_q  \\
    &=
    \eta_\Lambda
    \Dist_{\loc}(\mathfrak D,\Image G_\Lambda^{\loc})
    +
    \Delta_\Lambda .
\end{aligned}
\]
This is the claimed normalized error budget.
\end{proof}

\begin{remark}[Status of error normalization]
The normalization lemma is algebraic once the component bounds are assumed.
The genuine localized PDE work is the proof of those component bounds in
concrete norms.  In particular, no scale-uniform pressure-tail, cutoff leakage,
truncation-tail, nonlinear-remainder, reproduction-drift, gauge-mismatch, or
profit estimate is proved here.
\end{remark}

\subsection{Normalized absorption audit}\label{sec:absorption-audit}

\begin{proposition}[Concrete-budget absorption criterion]
\label{prop:concrete-budget-absorption-criterion}
Assume the hypotheses of
\Cref{thm:main-local-to-clean-transfer} except replace
\Cref{ass:component-normalized-error-bounds} by the concrete residual-budget
synchronization hypotheses of
\Cref{lem:concrete-residual-budget-synchronization} and by the following
concrete-budget condition: there are constants \(\eta_{\rm conc}\ge0\) and
\(\Delta_{\rm conc}\ge0\) such that every localized package in the chart domain
satisfies
\[
    \mathcal B_\Lambda^{\rm conc}(\mathfrak D)
    \le
    \eta_{\rm conc}
    \Dist_{\loc}(\mathfrak D,\Image G_\Lambda^{\loc})
    +
    \Delta_{\rm conc}.
\]
Assume also that
\[
    \eta_{\rm conc}<c_\Lambda^{\cl}(1-\varepsilon_G).
\]
Then the algebraic local-to-clean anti-phantom transfer theorem applies with
\[
    \eta_\Lambda=\eta_{\rm conc},
    \qquad
    \Delta_\Lambda=\Delta_{\rm conc}.
\]
In particular,
\[
    c_\Lambda^{\loc}
    =
    c_\Lambda^{\cl}(1-\varepsilon_G)-\eta_{\rm conc}
    >0,
    \qquad
    \Delta_\Lambda'
    =
    \Delta_{\rm conc}+c_\Lambda^{\cl}\delta_G .
\]
\end{proposition}

\begin{proof}
By \Cref{lem:concrete-residual-budget-synchronization},
\[
    \Err_\Lambda(\mathfrak D)
    \le
    \mathcal B_\Lambda^{\rm conc}(\mathfrak D).
\]
The concrete-budget condition therefore gives
\[
    \Err_\Lambda(\mathfrak D)
    \le
    \eta_{\rm conc}
    \Dist_{\loc}(\mathfrak D,\Image G_\Lambda^{\loc})
    +
    \Delta_{\rm conc}.
\]
Thus the normalized error-budget input required by
\Cref{thm:main-local-to-clean-transfer} holds with
\(\eta_\Lambda=\eta_{\rm conc}\) and
\(\Delta_\Lambda=\Delta_{\rm conc}\).  The strict inequality
\(\eta_{\rm conc}<c_\Lambda^{\cl}(1-\varepsilon_G)\) is exactly the
threshold condition in the main transfer theorem.  Applying
\Cref{thm:main-local-to-clean-transfer} gives the stated constants.
\end{proof}

\begin{remark}[Absorption audit of the concrete components]
\Cref{prop:concrete-budget-absorption-criterion} is an audit criterion, not a
proof of normalized absorption.  The current status of the concrete components
is as follows.
\begin{enumerate}[label=(\arabic*),leftmargin=2em]
    \item \emph{Pressure budget.}  The pressure budget has fixed-scale
    commutator, active-source, harmonic-tail, projection, mean, and
    periodization controls.  Absorption into
    \(\eta_{\rm conc}\Dist_{\loc}+\Delta_{\rm conc}\) remains conditional.
    \item \emph{Localization budget.}  The localization budget is bounded by
    normalized CKN quantities such as \(A\), \(C\), \(D\), and \(E\).
    Absorption remains conditional.
    \item \emph{Truncation leakage.}  The truncation leakage is bounded at
    finite-window level by a projection residual norm.  Decay in the
    truncation dimension \(N\) remains conditional.
    \item \emph{Nonlinear cutoff mismatch.}  The cutoff-flux mismatch is
    bounded by \(C^{2/3}\) in the fixed normalized model.  Absorption remains
    conditional.
    \item \emph{Dyadic reproduction drift.}  The dyadic reproduction drift is
    bounded by a finite-dimensional operator norm.  Smallness of that drift
    remains conditional.
    \item \emph{Gauge mismatch.}  A finite-dimensional selected-gauge bound is
    available.  Identifying it with the full gauge error and proving smallness
    remain conditional.
    \item \emph{Profit discrepancy.}  The profit discrepancy is an exact scalar
    residual, and the positive-part Lipschitz comparison and finite-dimensional
    profit bound are proved.  Absorption remains conditional.
\end{enumerate}
No component in this list is currently proved small, scale-uniform, or
absorbable.
\end{remark}

\begin{remark}[Pressure target selected from the audit]
The next honest theorem target is a concrete sufficient condition for one
component of the absorption criterion.  The target selected from the audit is
the pressure budget: prove that, under an explicit pressure-tail compatibility
hypothesis,
\[
    \mathcal B_{\rm prs}^{\rm norm}(\mathfrak D)
    \le
    \eta_{\rm prs}^{\rm conc}
    \Dist_{\loc}(\mathfrak D,\Image G_\Lambda^{\loc})
    +
    \Delta_{\rm prs}^{\rm conc}.
\]
This would not prove the full normalized absorption criterion, but it would
turn the largest current pressure placeholder into a precise sufficient
condition.
\end{remark}

\subsection{Pressure-budget sufficient condition}

The pressure audit can now be converted into a precise conditional statement.
No pressure smallness is asserted.  The point is only to identify a concrete
finite-window hypothesis under which the already defined pressure budget has
the normalized absorbed-plus-additive form required by the transfer theorem.

\begin{definition}[Concrete pressure-budget entries]
\label{def:concrete-pressure-budget-entries}
For each \(k\in\Lambda_{\rm sc}\), write the summands in
\(\mathcal B_{\rm prs}^{\rm norm}\) from
\Cref{prop:combined-normalized-pressure-component-package} as
\[
\begin{aligned}
    T_{{\rm harm},k}(\mathfrak D)
    &:=
    |I_k|^{1/6}|B_{1/2}|^{1/6}
    \left(\frac23\right)^{M+1}
    \|p_k^{\rm harm}\|_{L^2(I_k;L^2(B_{3/4}))},\\
    T_{{\rm comm},k}(\mathfrak D)
    &:=
    C_{\rm nlp}\sum_{i,j=1}^3
    \|(1-\eta_k)f_{k,ij}\|_{L^{3/2}(I_k;L^{3/2}(A_{3/4,1}))},\\
    T_{{\rm active},k}(\mathfrak D)
    &:=
    C_{\rm CZ}\sum_{i,j=1}^3
    \|f_{k,ij}-F_{k,ij}^{\cl}\|_{L^{3/2}(I_k;L^{3/2}(B_1))},\\
    T_{{\rm proj},k}(\mathfrak D)
    &:=
    \Err_{{\rm proj},k},\\
    T_{{\rm mean},k}(\mathfrak D)
    &:=
    \Err_{{\rm mean},k},\\
    T_{{\rm per},k}(\mathfrak D)
    &:=
    \Err_{{\rm per},k}.
\end{aligned}
\]
Thus
\[
    \mathcal B_{\rm prs}^{\rm norm}(\mathfrak D)
    =
    \sum_{k\in\Lambda_{\rm sc}}\omega_{{\rm prs},k}
    \sum_{\rho\in\mathcal R_{\rm prs}}
    T_{\rho,k}(\mathfrak D),
    \qquad
    \mathcal R_{\rm prs}
    :=
    \{
    {\rm harm},{\rm comm},{\rm active},{\rm proj},{\rm mean},{\rm per}
    \}.
\]
\end{definition}

\begin{assumption}[Pressure-tail compatibility with the localized quotient]
\label{ass:pressure-tail-quotient-compatibility}
For every pressure entry
\(\rho\in\mathcal R_{\rm prs}\) and every scale
\(k\in\Lambda_{\rm sc}\), there are constants
\(\eta_{{\rm prs},\rho,k}^{\rm conc}\ge0\) and
\(\Delta_{{\rm prs},\rho,k}^{\rm conc}\ge0\) such that every localized package
in the chart domain satisfies
\[
    \omega_{{\rm prs},k}T_{\rho,k}(\mathfrak D)
    \le
    \eta_{{\rm prs},\rho,k}^{\rm conc}
    \Dist_{\loc}(\mathfrak D,\Image G_\Lambda^{\loc})
    +
    \Delta_{{\rm prs},\rho,k}^{\rm conc}.
\]
\end{assumption}

\begin{proposition}[Pressure-budget sufficient condition for absorption]
\label{prop:pressure-budget-sufficient-condition}
Assume the fixed-geometry hypotheses of
\Cref{prop:combined-normalized-pressure-component-package} and the pressure
compatibility hypothesis
\Cref{ass:pressure-tail-quotient-compatibility}.  Define
\[
    \eta_{\rm prs}^{\rm conc}
    :=
    \sum_{k\in\Lambda_{\rm sc}}
    \sum_{\rho\in\mathcal R_{\rm prs}}
    \eta_{{\rm prs},\rho,k}^{\rm conc},
    \qquad
    \Delta_{\rm prs}^{\rm conc}
    :=
    \sum_{k\in\Lambda_{\rm sc}}
    \sum_{\rho\in\mathcal R_{\rm prs}}
    \Delta_{{\rm prs},\rho,k}^{\rm conc}.
\]
Then
\[
    \mathcal B_{\rm prs}^{\rm norm}(\mathfrak D)
    \le
    \eta_{\rm prs}^{\rm conc}
    \Dist_{\loc}(\mathfrak D,\Image G_\Lambda^{\loc})
    +
    \Delta_{\rm prs}^{\rm conc}.
\]
Consequently, whenever the pressure budget is instantiated as in
\Cref{prop:combined-normalized-pressure-component-package},
\[
    \Err_{\rm prs}(\mathfrak D)
    \le
    \eta_{\rm prs}^{\rm conc}
    \Dist_{\loc}(\mathfrak D,\Image G_\Lambda^{\loc})
    +
    \Delta_{\rm prs}^{\rm conc}.
\]
\end{proposition}

\begin{proof}
By \Cref{def:concrete-pressure-budget-entries},
\[
    \mathcal B_{\rm prs}^{\rm norm}(\mathfrak D)
    =
    \sum_{k\in\Lambda_{\rm sc}}\omega_{{\rm prs},k}
    \sum_{\rho\in\mathcal R_{\rm prs}}
    T_{\rho,k}(\mathfrak D).
\]
Applying
\Cref{ass:pressure-tail-quotient-compatibility} to each weighted summand gives
\[
\begin{aligned}
    \mathcal B_{\rm prs}^{\rm norm}(\mathfrak D)
    &\le
    \sum_{k\in\Lambda_{\rm sc}}
    \sum_{\rho\in\mathcal R_{\rm prs}}
    \left(
    \eta_{{\rm prs},\rho,k}^{\rm conc}
    \Dist_{\loc}(\mathfrak D,\Image G_\Lambda^{\loc})
    +
    \Delta_{{\rm prs},\rho,k}^{\rm conc}
    \right)\\
    &=
    \eta_{\rm prs}^{\rm conc}
    \Dist_{\loc}(\mathfrak D,\Image G_\Lambda^{\loc})
    +
    \Delta_{\rm prs}^{\rm conc}.
\end{aligned}
\]
The final pressure-error estimate follows from
\Cref{prop:combined-normalized-pressure-component-package}, which gives
\(\Err_{\rm prs}\le\mathcal B_{\rm prs}^{\rm norm}\) under the stated
fixed-geometry pressure instantiation.
\end{proof}

\begin{definition}[Non-pressure concrete remainder budget]
\label{def:non-pressure-concrete-remainder-budget}
The concrete budget left after the pressure component is
\[
\begin{aligned}
    \mathcal B_{\Lambda,{\rm nonprs}}^{\rm conc}(\mathfrak D)
    :=
    &\mathcal B_{\rm loc}^{\rm norm}(\mathfrak D)
    +
    \mathcal B_{\rm tr}^{\rm norm}(\mathfrak D)
    +
    \mathcal B_{\rm nl}^{\rm cut}(\mathfrak D)\\
    &+
    \Err_{\rm nl}^{\rm rem}(\mathfrak D)
    +
    \Err_{\rm rep}(\mathfrak D)
    +
    \Err_{\rm gauge}(\mathfrak D)
    +
    \Err_{\rm prof}(\mathfrak D).
\end{aligned}
\]
Thus
\[
    \mathcal B_\Lambda^{\rm conc}
    =
    \mathcal B_{\rm prs}^{\rm norm}
    +
    \mathcal B_{\Lambda,{\rm nonprs}}^{\rm conc}.
\]
\end{definition}

\begin{corollary}[Pressure-isolated concrete-budget criterion]
\label{cor:pressure-isolated-concrete-budget-criterion}
Assume the hypotheses of
\Cref{prop:pressure-budget-sufficient-condition}.  Suppose, in addition, that
the non-pressure concrete remainder satisfies
\[
    \mathcal B_{\Lambda,{\rm nonprs}}^{\rm conc}(\mathfrak D)
    \le
    \eta_{\rm nonprs}^{\rm conc}
    \Dist_{\loc}(\mathfrak D,\Image G_\Lambda^{\loc})
    +
    \Delta_{\rm nonprs}^{\rm conc}.
\]
If
\[
    \eta_{\rm prs}^{\rm conc}
    +
    \eta_{\rm nonprs}^{\rm conc}
    <
    c_\Lambda^{\cl}(1-\varepsilon_G),
\]
then the concrete-budget absorption criterion
\Cref{prop:concrete-budget-absorption-criterion} applies with
\[
    \eta_{\rm conc}
    =
    \eta_{\rm prs}^{\rm conc}
    +
    \eta_{\rm nonprs}^{\rm conc},
    \qquad
    \Delta_{\rm conc}
    =
    \Delta_{\rm prs}^{\rm conc}
    +
    \Delta_{\rm nonprs}^{\rm conc}.
\]
\end{corollary}

\begin{proof}
By \Cref{def:non-pressure-concrete-remainder-budget},
\[
    \mathcal B_\Lambda^{\rm conc}
    =
    \mathcal B_{\rm prs}^{\rm norm}
    +
    \mathcal B_{\Lambda,{\rm nonprs}}^{\rm conc}.
\]
Use \Cref{prop:pressure-budget-sufficient-condition} on the pressure term and
the displayed hypothesis on the non-pressure remainder.  Summing the two
estimates gives
\[
    \mathcal B_\Lambda^{\rm conc}(\mathfrak D)
    \le
    \left(
    \eta_{\rm prs}^{\rm conc}
    +
    \eta_{\rm nonprs}^{\rm conc}
    \right)
    \Dist_{\loc}(\mathfrak D,\Image G_\Lambda^{\loc})
    +
    \Delta_{\rm prs}^{\rm conc}
    +
    \Delta_{\rm nonprs}^{\rm conc}.
\]
The strict coefficient inequality is exactly the threshold needed in
\Cref{prop:concrete-budget-absorption-criterion}.
\end{proof}

\begin{remark}[Status of pressure-budget compatibility]
\Cref{prop:pressure-budget-sufficient-condition} proves only that the explicit
pressure entries imply a pressure-level absorbed-plus-additive bound once the
component compatibility estimates are supplied.  It does not prove any of
those compatibility estimates, does not make
\(\eta_{\rm prs}^{\rm conc}\) small, and does not establish scale-uniform
pressure-tail decay.  The remaining honest absorption work is now localized:
one must prove component compatibility for the pressure entries or move to the
next component of
\(\mathcal B_{\Lambda,{\rm nonprs}}^{\rm conc}\).
\end{remark}

\subsection{Localization-budget sufficient condition}

The next component in the concrete budget is the localization leakage budget.
The fixed-scale estimates above already show that the localization leakage is
bounded by CKN-normalized quantities.  The remaining absorption question is
whether those normalized quantities are compatible with the localized quotient
distance.  This subsection records that compatibility as an explicit
hypothesis and proves only the resulting finite-window summation statement.

\begin{definition}[Concrete localization-budget entries]
\label{def:concrete-localization-budget-entries}
For each \(k\in\Lambda_{\rm sc}\), write the summands in
\(\mathcal B_{\rm loc}^{\rm norm}\) from
\Cref{def:normalized-localization-budget} as
\[
\begin{aligned}
    L_{{\rm en},k}(\mathfrak D)
    &:=
    C_\chi\omega_{{\rm en},k} A_k,\\
    L_{{\rm flux},k}(\mathfrak D)
    &:=
    C_\chi\omega_{{\rm flux},k} C_k,\\
    L_{{\rm prs},k}(\mathfrak D)
    &:=
    C_\chi\omega_{{\rm prs},k}^{\loc} C_k^{1/3}D_k^{2/3},\\
    L_{{\rm mom},k}(\mathfrak D)
    &:=
    C_\chi\omega_{{\rm mom},k}
    \bigl(
        A_k^{1/2}
        +
        E_k^{1/2}
        +
        C_k^{2/3}
        +
        D_k^{2/3}
    \bigr).
\end{aligned}
\]
Thus
\[
    \mathcal B_{\rm loc}^{\rm norm}(\mathfrak D)
    =
    \sum_{k\in\Lambda_{\rm sc}}
    \sum_{\sigma\in\mathcal R_{\rm loc}}
    L_{\sigma,k}(\mathfrak D),
    \qquad
    \mathcal R_{\rm loc}
    :=
    \{
    {\rm en},{\rm flux},{\rm prs},{\rm mom}
    \}.
\]
\end{definition}

\begin{assumption}[Localization--CKN compatibility with the localized quotient]
\label{ass:localization-ckn-quotient-compatibility}
For every localization entry
\(\sigma\in\mathcal R_{\rm loc}\) and every scale
\(k\in\Lambda_{\rm sc}\), there are constants
\(\eta_{{\rm loc},\sigma,k}^{\rm conc}\ge0\) and
\(\Delta_{{\rm loc},\sigma,k}^{\rm conc}\ge0\) such that every localized
package in the chart domain satisfies
\[
    L_{\sigma,k}(\mathfrak D)
    \le
    \eta_{{\rm loc},\sigma,k}^{\rm conc}
    \Dist_{\loc}(\mathfrak D,\Image G_\Lambda^{\loc})
    +
    \Delta_{{\rm loc},\sigma,k}^{\rm conc}.
\]
\end{assumption}

\begin{proposition}[Localization-budget sufficient condition for absorption]
\label{prop:localization-budget-sufficient-condition}
Assume the hypotheses of
\Cref{lem:finite-window-localization-budget} and the compatibility hypothesis
\Cref{ass:localization-ckn-quotient-compatibility}.  Define
\[
    \eta_{\rm loc}^{\rm conc}
    :=
    \sum_{k\in\Lambda_{\rm sc}}
    \sum_{\sigma\in\mathcal R_{\rm loc}}
    \eta_{{\rm loc},\sigma,k}^{\rm conc},
    \qquad
    \Delta_{\rm loc}^{\rm conc}
    :=
    \sum_{k\in\Lambda_{\rm sc}}
    \sum_{\sigma\in\mathcal R_{\rm loc}}
    \Delta_{{\rm loc},\sigma,k}^{\rm conc}.
\]
Then
\[
    \mathcal B_{\rm loc}^{\rm norm}(\mathfrak D)
    \le
    \eta_{\rm loc}^{\rm conc}
    \Dist_{\loc}(\mathfrak D,\Image G_\Lambda^{\loc})
    +
    \Delta_{\rm loc}^{\rm conc}.
\]
Consequently,
\[
    \Err_{\rm loc}(\mathfrak D)
    \le
    \eta_{\rm loc}^{\rm conc}
    \Dist_{\loc}(\mathfrak D,\Image G_\Lambda^{\loc})
    +
    \Delta_{\rm loc}^{\rm conc}.
\]
\end{proposition}

\begin{proof}
By \Cref{def:concrete-localization-budget-entries},
\[
    \mathcal B_{\rm loc}^{\rm norm}(\mathfrak D)
    =
    \sum_{k\in\Lambda_{\rm sc}}
    \sum_{\sigma\in\mathcal R_{\rm loc}}
    L_{\sigma,k}(\mathfrak D).
\]
Applying
\Cref{ass:localization-ckn-quotient-compatibility} to each summand and then
summing gives
\[
\begin{aligned}
    \mathcal B_{\rm loc}^{\rm norm}(\mathfrak D)
    &\le
    \sum_{k\in\Lambda_{\rm sc}}
    \sum_{\sigma\in\mathcal R_{\rm loc}}
    \left(
    \eta_{{\rm loc},\sigma,k}^{\rm conc}
    \Dist_{\loc}(\mathfrak D,\Image G_\Lambda^{\loc})
    +
    \Delta_{{\rm loc},\sigma,k}^{\rm conc}
    \right)\\
    &=
    \eta_{\rm loc}^{\rm conc}
    \Dist_{\loc}(\mathfrak D,\Image G_\Lambda^{\loc})
    +
    \Delta_{\rm loc}^{\rm conc}.
\end{aligned}
\]
The final estimate follows from
\Cref{lem:finite-window-localization-budget}, which gives
\(\Err_{\rm loc}\le\mathcal B_{\rm loc}^{\rm norm}\) under the stated
fixed-scale localization hypotheses.
\end{proof}

\begin{definition}[Post-pressure-localization concrete remainder budget]
\label{def:post-pressure-localization-remainder-budget}
After the pressure and localization components have been isolated, define
\[
\begin{aligned}
    \mathcal B_{\Lambda,{\rm rem2}}^{\rm conc}(\mathfrak D)
    :=
    &\mathcal B_{\rm tr}^{\rm norm}(\mathfrak D)
    +
    \mathcal B_{\rm nl}^{\rm cut}(\mathfrak D)
    +
    \Err_{\rm nl}^{\rm rem}(\mathfrak D)\\
    &+
    \Err_{\rm rep}(\mathfrak D)
    +
    \Err_{\rm gauge}(\mathfrak D)
    +
    \Err_{\rm prof}(\mathfrak D).
\end{aligned}
\]
Hence
\[
    \mathcal B_\Lambda^{\rm conc}
    =
    \mathcal B_{\rm prs}^{\rm norm}
    +
    \mathcal B_{\rm loc}^{\rm norm}
    +
    \mathcal B_{\Lambda,{\rm rem2}}^{\rm conc}.
\]
\end{definition}

\begin{corollary}[Pressure-localization concrete-budget criterion]
\label{cor:pressure-localization-concrete-budget-criterion}
Assume the hypotheses of
\Cref{prop:pressure-budget-sufficient-condition} and
\Cref{prop:localization-budget-sufficient-condition}.  Suppose, in addition,
that the remaining concrete budget satisfies
\[
    \mathcal B_{\Lambda,{\rm rem2}}^{\rm conc}(\mathfrak D)
    \le
    \eta_{\rm rem2}^{\rm conc}
    \Dist_{\loc}(\mathfrak D,\Image G_\Lambda^{\loc})
    +
    \Delta_{\rm rem2}^{\rm conc}.
\]
If
\[
    \eta_{\rm prs}^{\rm conc}
    +
    \eta_{\rm loc}^{\rm conc}
    +
    \eta_{\rm rem2}^{\rm conc}
    <
    c_\Lambda^{\cl}(1-\varepsilon_G),
\]
then the concrete-budget absorption criterion applies with
\[
    \eta_{\rm conc}
    =
    \eta_{\rm prs}^{\rm conc}
    +
    \eta_{\rm loc}^{\rm conc}
    +
    \eta_{\rm rem2}^{\rm conc},
    \qquad
    \Delta_{\rm conc}
    =
    \Delta_{\rm prs}^{\rm conc}
    +
    \Delta_{\rm loc}^{\rm conc}
    +
    \Delta_{\rm rem2}^{\rm conc}.
\]
\end{corollary}

\begin{proof}
Use the decomposition in
\Cref{def:post-pressure-localization-remainder-budget}.  Apply
\Cref{prop:pressure-budget-sufficient-condition} to the pressure component,
\Cref{prop:localization-budget-sufficient-condition} to the localization
component, and the displayed hypothesis to the remaining component.  Summing
the three inequalities gives the concrete-budget bound with the stated
\(\eta_{\rm conc}\) and \(\Delta_{\rm conc}\).  The strict coefficient
condition is the threshold in
\Cref{prop:concrete-budget-absorption-criterion}.
\end{proof}

\begin{remark}[Status of localization-budget compatibility]
\Cref{prop:localization-budget-sufficient-condition} proves only that the
already derived CKN-normalized localization entries yield an
absorbed-plus-additive localization bound once those entries are compatible
with the localized quotient distance.  It does not prove smallness of
\(A_k\), \(C_k\), \(D_k\), or \(E_k\), does not give scale-uniform
localization control, and does not absorb the remaining truncation, nonlinear,
reproduction, gauge, or profit terms.
\end{remark}

\subsection{Truncation-budget sufficient condition}

The truncation residual is already a finite-window projection residual in
Banach observation spaces.  The bounded projection lemma makes the residual
well defined, but it does not imply approximation decay in \(N\).  The honest
absorption input is therefore a truncation-tail compatibility hypothesis,
recorded below.

\begin{definition}[Concrete truncation-budget entries]
\label{def:concrete-truncation-budget-entries}
For each \(k\in\Lambda_{\rm sc}\), define
\[
    U_{{\rm tr},k}(\mathfrak D)
    :=
    \omega_{{\rm tr},k}
    \bigl(
        1+
        \|\mathsf T_{{\rm tr},N,k}\|_{\calY_{{\rm tr},k}\to\calY_{{\rm tr},k}}
    \bigr)
    \|z_k(\mathfrak D)\|_{\calY_{{\rm tr},k}} .
\]
Then
\[
    \mathcal B_{\rm tr}^{\rm norm}(\mathfrak D)
    =
    \sum_{k\in\Lambda_{\rm sc}}U_{{\rm tr},k}(\mathfrak D).
\]
\end{definition}

\begin{assumption}[Truncation-tail compatibility with the localized quotient]
\label{ass:truncation-tail-quotient-compatibility}
For every scale \(k\in\Lambda_{\rm sc}\), there are constants
\(\eta_{{\rm tr},k}^{\rm conc}\ge0\) and
\(\Delta_{{\rm tr},k}^{\rm conc}\ge0\) such that every localized package in the
chart domain satisfies
\[
    U_{{\rm tr},k}(\mathfrak D)
    \le
    \eta_{{\rm tr},k}^{\rm conc}
    \Dist_{\loc}(\mathfrak D,\Image G_\Lambda^{\loc})
    +
    \Delta_{{\rm tr},k}^{\rm conc}.
\]
\end{assumption}

\begin{proposition}[Truncation-budget sufficient condition for absorption]
\label{prop:truncation-budget-sufficient-condition}
Assume the finite-window truncation datum of
\Cref{conv:finite-window-truncation-datum}, the bounded projections of
\Cref{lem:bounded-finite-dimensional-truncation-projection}, and the
compatibility hypothesis
\Cref{ass:truncation-tail-quotient-compatibility}.  Define
\[
    \eta_{\rm tr}^{\rm conc}
    :=
    \sum_{k\in\Lambda_{\rm sc}}\eta_{{\rm tr},k}^{\rm conc},
    \qquad
    \Delta_{\rm tr}^{\rm conc}
    :=
    \sum_{k\in\Lambda_{\rm sc}}\Delta_{{\rm tr},k}^{\rm conc}.
\]
Then
\[
    \mathcal B_{\rm tr}^{\rm norm}(\mathfrak D)
    \le
    \eta_{\rm tr}^{\rm conc}
    \Dist_{\loc}(\mathfrak D,\Image G_\Lambda^{\loc})
    +
    \Delta_{\rm tr}^{\rm conc}.
\]
Consequently,
\[
    \Err_{\rm tr}(\mathfrak D)
    \le
    \eta_{\rm tr}^{\rm conc}
    \Dist_{\loc}(\mathfrak D,\Image G_\Lambda^{\loc})
    +
    \Delta_{\rm tr}^{\rm conc}.
\]
\end{proposition}

\begin{proof}
By \Cref{def:concrete-truncation-budget-entries},
\[
    \mathcal B_{\rm tr}^{\rm norm}(\mathfrak D)
    =
    \sum_{k\in\Lambda_{\rm sc}}U_{{\rm tr},k}(\mathfrak D).
\]
Applying \Cref{ass:truncation-tail-quotient-compatibility} to each summand
and summing over the finite window gives
\[
\begin{aligned}
    \mathcal B_{\rm tr}^{\rm norm}(\mathfrak D)
    &\le
    \sum_{k\in\Lambda_{\rm sc}}
    \left(
    \eta_{{\rm tr},k}^{\rm conc}
    \Dist_{\loc}(\mathfrak D,\Image G_\Lambda^{\loc})
    +
    \Delta_{{\rm tr},k}^{\rm conc}
    \right)\\
    &=
    \eta_{\rm tr}^{\rm conc}
    \Dist_{\loc}(\mathfrak D,\Image G_\Lambda^{\loc})
    +
    \Delta_{\rm tr}^{\rm conc}.
\end{aligned}
\]
The final estimate follows from
\Cref{lem:finite-window-truncation-leakage-bound}, which gives
\(\Err_{\rm tr}\le\mathcal B_{\rm tr}^{\rm norm}\).
\end{proof}

\begin{definition}[Post-truncation concrete remainder budget]
\label{def:post-truncation-concrete-remainder-budget}
After the pressure, localization, and truncation components have been
isolated, define
\[
\begin{aligned}
    \mathcal B_{\Lambda,{\rm rem3}}^{\rm conc}(\mathfrak D)
    :=
    &\mathcal B_{\rm nl}^{\rm cut}(\mathfrak D)
    +
    \Err_{\rm nl}^{\rm rem}(\mathfrak D)
    +
    \Err_{\rm rep}(\mathfrak D)\\
    &+
    \Err_{\rm gauge}(\mathfrak D)
    +
    \Err_{\rm prof}(\mathfrak D).
\end{aligned}
\]
Hence
\[
    \mathcal B_\Lambda^{\rm conc}
    =
    \mathcal B_{\rm prs}^{\rm norm}
    +
    \mathcal B_{\rm loc}^{\rm norm}
    +
    \mathcal B_{\rm tr}^{\rm norm}
    +
    \mathcal B_{\Lambda,{\rm rem3}}^{\rm conc}.
\]
\end{definition}

\begin{corollary}[Pressure-localization-truncation concrete-budget criterion]
\label{cor:pressure-localization-truncation-concrete-budget-criterion}
Assume the hypotheses of
\Cref{prop:pressure-budget-sufficient-condition},
\Cref{prop:localization-budget-sufficient-condition}, and
\Cref{prop:truncation-budget-sufficient-condition}.  Suppose, in addition,
that
\[
    \mathcal B_{\Lambda,{\rm rem3}}^{\rm conc}(\mathfrak D)
    \le
    \eta_{\rm rem3}^{\rm conc}
    \Dist_{\loc}(\mathfrak D,\Image G_\Lambda^{\loc})
    +
    \Delta_{\rm rem3}^{\rm conc}.
\]
If
\[
    \eta_{\rm prs}^{\rm conc}
    +
    \eta_{\rm loc}^{\rm conc}
    +
    \eta_{\rm tr}^{\rm conc}
    +
    \eta_{\rm rem3}^{\rm conc}
    <
    c_\Lambda^{\cl}(1-\varepsilon_G),
\]
then the concrete-budget absorption criterion applies with
\[
    \eta_{\rm conc}
    =
    \eta_{\rm prs}^{\rm conc}
    +
    \eta_{\rm loc}^{\rm conc}
    +
    \eta_{\rm tr}^{\rm conc}
    +
    \eta_{\rm rem3}^{\rm conc},
\]
and
\[
    \Delta_{\rm conc}
    =
    \Delta_{\rm prs}^{\rm conc}
    +
    \Delta_{\rm loc}^{\rm conc}
    +
    \Delta_{\rm tr}^{\rm conc}
    +
    \Delta_{\rm rem3}^{\rm conc}.
\]
\end{corollary}

\begin{proof}
Use the decomposition in
\Cref{def:post-truncation-concrete-remainder-budget}.  Apply the pressure,
localization, and truncation sufficient conditions to the first three
components and the displayed hypothesis to the remaining component.  Summing
the four inequalities gives the concrete-budget bound with the stated
\(\eta_{\rm conc}\) and \(\Delta_{\rm conc}\).  The strict coefficient
condition is the threshold in
\Cref{prop:concrete-budget-absorption-criterion}.
\end{proof}

\begin{remark}[Status of truncation-budget compatibility]
\Cref{prop:truncation-budget-sufficient-condition} proves only that a stated
truncation-tail compatibility hypothesis implies the normalized
absorbed-plus-additive truncation bound.  The bounded projection construction
does not by itself imply decay as \(N\to\infty\), and no such decay,
smallness, or scale-uniform truncation control is claimed here.
\end{remark}

\subsection{Nonlinear-budget sufficient condition}

The nonlinear part of the concrete budget has two pieces: the fixed-scale
cutoff-flux mismatch already bounded by \(C_\chi C_k^{2/3}\), and the
finite-dimensional nonlinear remainder model.  These are kept separate because
the first is a concrete cutoff calculation, while the second is only a chosen
finite-window residual model.

\begin{definition}[Concrete nonlinear-budget entries]
\label{def:concrete-nonlinear-budget-entries}
For each \(k\in\Lambda_{\rm sc}\), define the cutoff nonlinear entry
\[
    W_{{\rm cut},k}(\mathfrak D)
    :=
    C_\chi\omega_{{\rm nl},k}^{\rm cut}C_k^{2/3}.
\]
The concrete nonlinear budget used in this subsection is
\[
    \mathcal B_{\rm nl}^{\rm conc}(\mathfrak D)
    :=
    \mathcal B_{\rm nl}^{\rm cut}(\mathfrak D)
    +
    \Err_{\rm nl}^{\rm rem}(\mathfrak D)
    =
    \sum_{k\in\Lambda_{\rm sc}}W_{{\rm cut},k}(\mathfrak D)
    +
    \Err_{\rm nl}^{\rm rem}(\mathfrak D).
\]
\end{definition}

\begin{assumption}[Nonlinear compatibility with the localized quotient]
\label{ass:nonlinear-quotient-compatibility}
For each \(k\in\Lambda_{\rm sc}\), there are constants
\(\eta_{{\rm cut},k}^{\rm conc}\ge0\) and
\(\Delta_{{\rm cut},k}^{\rm conc}\ge0\) such that
\[
    W_{{\rm cut},k}(\mathfrak D)
    \le
    \eta_{{\rm cut},k}^{\rm conc}
    \Dist_{\loc}(\mathfrak D,\Image G_\Lambda^{\loc})
    +
    \Delta_{{\rm cut},k}^{\rm conc}.
\]
Assume also that there are constants
\(\eta_{\rm rem,nl}^{\rm conc}\ge0\) and
\(\Delta_{\rm rem,nl}^{\rm conc}\ge0\) such that
\[
    \Err_{\rm nl}^{\rm rem}(\mathfrak D)
    \le
    \eta_{\rm rem,nl}^{\rm conc}
    \Dist_{\loc}(\mathfrak D,\Image G_\Lambda^{\loc})
    +
    \Delta_{\rm rem,nl}^{\rm conc}.
\]
\end{assumption}

\begin{proposition}[Nonlinear-budget sufficient condition for absorption]
\label{prop:nonlinear-budget-sufficient-condition}
Assume the hypotheses of
\Cref{lem:finite-window-nonlinear-cutoff-budget}, the finite-window nonlinear
remainder model of \Cref{def:finite-window-nonlinear-remainder-model}, and
\Cref{ass:nonlinear-quotient-compatibility}.  Define
\[
    \eta_{\rm nl}^{\rm conc}
    :=
    \sum_{k\in\Lambda_{\rm sc}}\eta_{{\rm cut},k}^{\rm conc}
    +
    \eta_{\rm rem,nl}^{\rm conc},
    \qquad
    \Delta_{\rm nl}^{\rm conc}
    :=
    \sum_{k\in\Lambda_{\rm sc}}\Delta_{{\rm cut},k}^{\rm conc}
    +
    \Delta_{\rm rem,nl}^{\rm conc}.
\]
Then
\[
    \mathcal B_{\rm nl}^{\rm conc}(\mathfrak D)
    \le
    \eta_{\rm nl}^{\rm conc}
    \Dist_{\loc}(\mathfrak D,\Image G_\Lambda^{\loc})
    +
    \Delta_{\rm nl}^{\rm conc}.
\]
Consequently,
\[
    \Err_{\rm nl}^{\rm cut}(\mathfrak D)
    +
    \Err_{\rm nl}^{\rm rem}(\mathfrak D)
    \le
    \eta_{\rm nl}^{\rm conc}
    \Dist_{\loc}(\mathfrak D,\Image G_\Lambda^{\loc})
    +
    \Delta_{\rm nl}^{\rm conc}.
\]
\end{proposition}

\begin{proof}
By \Cref{def:concrete-nonlinear-budget-entries},
\[
    \mathcal B_{\rm nl}^{\rm conc}(\mathfrak D)
    =
    \sum_{k\in\Lambda_{\rm sc}}W_{{\rm cut},k}(\mathfrak D)
    +
    \Err_{\rm nl}^{\rm rem}(\mathfrak D).
\]
Applying \Cref{ass:nonlinear-quotient-compatibility} to each cutoff entry and
to the nonlinear remainder gives
\[
\begin{aligned}
    \mathcal B_{\rm nl}^{\rm conc}(\mathfrak D)
    &\le
    \sum_{k\in\Lambda_{\rm sc}}
    \left(
    \eta_{{\rm cut},k}^{\rm conc}
    \Dist_{\loc}(\mathfrak D,\Image G_\Lambda^{\loc})
    +
    \Delta_{{\rm cut},k}^{\rm conc}
    \right)\\
    &\qquad
    +
    \eta_{\rm rem,nl}^{\rm conc}
    \Dist_{\loc}(\mathfrak D,\Image G_\Lambda^{\loc})
    +
    \Delta_{\rm rem,nl}^{\rm conc}\\
    &=
    \eta_{\rm nl}^{\rm conc}
    \Dist_{\loc}(\mathfrak D,\Image G_\Lambda^{\loc})
    +
    \Delta_{\rm nl}^{\rm conc}.
\end{aligned}
\]
The final estimate follows from
\Cref{lem:finite-window-nonlinear-cutoff-budget}, which gives
\(\Err_{\rm nl}^{\rm cut}\le\mathcal B_{\rm nl}^{\rm cut}\), together with the
definition of \(\mathcal B_{\rm nl}^{\rm conc}\).
\end{proof}

\begin{definition}[Post-nonlinear concrete remainder budget]
\label{def:post-nonlinear-concrete-remainder-budget}
After the pressure, localization, truncation, and nonlinear components have
been isolated, define
\[
    \mathcal B_{\Lambda,{\rm rem4}}^{\rm conc}(\mathfrak D)
    :=
    \Err_{\rm rep}(\mathfrak D)
    +
    \Err_{\rm gauge}(\mathfrak D)
    +
    \Err_{\rm prof}(\mathfrak D).
\]
Then
\[
    \mathcal B_\Lambda^{\rm conc}
    =
    \mathcal B_{\rm prs}^{\rm norm}
    +
    \mathcal B_{\rm loc}^{\rm norm}
    +
    \mathcal B_{\rm tr}^{\rm norm}
    +
    \mathcal B_{\rm nl}^{\rm conc}
    +
    \mathcal B_{\Lambda,{\rm rem4}}^{\rm conc}.
\]
\end{definition}

\begin{corollary}[Pressure-localization-truncation-nonlinear criterion]
\label{cor:pressure-localization-truncation-nonlinear-criterion}
Assume the hypotheses of
\Cref{prop:pressure-budget-sufficient-condition},
\Cref{prop:localization-budget-sufficient-condition},
\Cref{prop:truncation-budget-sufficient-condition}, and
\Cref{prop:nonlinear-budget-sufficient-condition}.  Suppose, in addition, that
\[
    \mathcal B_{\Lambda,{\rm rem4}}^{\rm conc}(\mathfrak D)
    \le
    \eta_{\rm rem4}^{\rm conc}
    \Dist_{\loc}(\mathfrak D,\Image G_\Lambda^{\loc})
    +
    \Delta_{\rm rem4}^{\rm conc}.
\]
If
\[
    \eta_{\rm prs}^{\rm conc}
    +
    \eta_{\rm loc}^{\rm conc}
    +
    \eta_{\rm tr}^{\rm conc}
    +
    \eta_{\rm nl}^{\rm conc}
    +
    \eta_{\rm rem4}^{\rm conc}
    <
    c_\Lambda^{\cl}(1-\varepsilon_G),
\]
then the concrete-budget absorption criterion applies with the corresponding
sum of the four displayed component coefficients and the remaining coefficient
\(\eta_{\rm rem4}^{\rm conc}\), and with the corresponding sum of the
\(\Delta\)-terms.
\end{corollary}

\begin{proof}
Use the decomposition in
\Cref{def:post-nonlinear-concrete-remainder-budget}.  Apply the pressure,
localization, truncation, and nonlinear sufficient conditions to the first four
components and the displayed hypothesis to the remaining component.  Summing
these inequalities gives the concrete-budget bound.  The strict coefficient
condition is the threshold in
\Cref{prop:concrete-budget-absorption-criterion}.
\end{proof}

\begin{remark}[Status of nonlinear-budget compatibility]
\Cref{prop:nonlinear-budget-sufficient-condition} does not identify the
finite-dimensional nonlinear remainder with the full Navier--Stokes nonlinear
residual.  It proves only that the cutoff mismatch budget and the chosen
remainder model are absorbed once their explicit compatibility estimates are
assumed.  No nonlinear smallness, scale-uniform closure, or full localized
Navier--Stokes transfer is claimed.
\end{remark}

\section{Main local-to-clean transfer theorem}\label{sec:main-reduction}

\subsection{Clean finite-window gap}

\begin{assumption}[Clean finite-window gap]
\label[assumption]{ass:clean-finite-window-gap}
The clean detection map satisfies
\[
    \|\calF_\Lambda^{\cl}(d)\|_{\cl}
    \ge
    c_\Lambda^{\cl}
    \Dist_{\cl}(d,\Image G_\Lambda^{\cl}),
    \qquad
    c_\Lambda^{\cl}>0,
\]
for all clean packages in the chart image.
\end{assumption}

\begin{remark}[Gauge compatibility of the clean gap]
The proof of the transfer theorem uses only the displayed lower bound in
\Cref{ass:clean-finite-window-gap}.  When this bound is interpreted as a
quotient anti-phantom gap, one should additionally require that the clean
detection map be compatible with the clean gauge directions, for example that
\(\calF_\Lambda^{\cl}\) descends to
\(\calD_\Lambda^{\cl}/\Image G_\Lambda^{\cl}\), or at least that the displayed
coercivity is certified after passing to the quotient.  This prevents the
``clean gap'' from measuring a gauge artifact rather than a genuine non-gauge
defect.
\end{remark}

\subsection{Split transfer inputs}

\begin{proposition}[Split hypotheses imply the transfer inputs]
\label[proposition]{prop:split-transfer-inputs}
Assume
\Cref{ass:quotient-lifting-stability,ass:componentwise-detection-estimates,ass:component-normalized-error-bounds}.
Then every localized package \(\mathfrak D\in\calU_\Lambda^{\loc}\) satisfies
the following three estimates:
\[
    \Dist_{\cl}
    (
    \Theta_\Lambda\mathfrak D,\Image G_\Lambda^{\cl}
    )
    \ge
    (1-\varepsilon_G)
    \Dist_{\loc}
    (
    \mathfrak D,\Image G_\Lambda^{\loc}
    )
    -
    \delta_G ,
\]
\[
    \|\calF_\Lambda^{\loc}(\mathfrak D)\|_{\loc}
    \ge
    \|\calF_\Lambda^{\cl}(\Theta_\Lambda\mathfrak D)\|_{\cl}
    -
    \Err_\Lambda(\mathfrak D),
\]
and
\[
    \Err_\Lambda(\mathfrak D)
    \le
    \eta_\Lambda
    \Dist_{\loc}
    (
    \mathfrak D,\Image G_\Lambda^{\loc}
    )
    +
    \Delta_\Lambda .
\]
\end{proposition}

\begin{proof}
The first estimate is \Cref{lem:quotient-distance-comparison}.  The second is
\Cref{lem:detection-map-comparison}.  The third is
\Cref{lem:error-budget-normalization}.  These results use exactly the three
split hypotheses listed in the statement.
\end{proof}

\subsection{Algebraic transfer theorem}

\begin{theorem}[Algebraic local-to-clean anti-phantom transfer]
\label[theorem]{thm:main-local-to-clean-transfer}
Assume
\Cref{ass:quotient-lifting-stability,ass:componentwise-detection-estimates,ass:component-normalized-error-bounds,ass:clean-finite-window-gap}.
If
\[
    \eta_\Lambda<c_\Lambda^{\cl}(1-\varepsilon_G),
\]
then every \(\mathfrak D\in\calU_\Lambda^{\loc}\) satisfies
\[
\begin{aligned}
    &\|O_\Lambda^{\loc}\mathfrak D\|
    +
    C_E\|\calE_\Lambda^{\loc}(\mathfrak D)\|
    +
    C_R\Rep_\Lambda^{\loc}(\mathfrak D)
    +
    [\Prof_\Lambda^{\loc}(\mathfrak D)]_+  \\
    &\qquad\ge
    c_\Lambda^{\loc}
    \Dist_{\loc}
    (
    \mathfrak D,\Image G_\Lambda^{\loc}
    )
    -
    \Delta_\Lambda',
\end{aligned}
\]
where
\[
    c_\Lambda^{\loc}
    =
    c_\Lambda^{\cl}(1-\varepsilon_G)-\eta_\Lambda>0,
    \qquad
    \Delta_\Lambda'
    =
    \Delta_\Lambda+c_\Lambda^{\cl}\delta_G.
\]
Equivalently,
\[
\begin{aligned}
    &\|O_\Lambda^{\loc}\mathfrak D\|
    +
    C_E\|\calE_\Lambda^{\loc}(\mathfrak D)\|
    +
    C_R\Rep_\Lambda^{\loc}(\mathfrak D)  \\
    &\qquad\ge
    c_\Lambda^{\loc}
    \Dist_{\loc}
    (
    \mathfrak D,\Image G_\Lambda^{\loc}
    )
    -
    [\Prof_\Lambda^{\loc}(\mathfrak D)]_+
    -
    \Delta_\Lambda'.
\end{aligned}
\]
\end{theorem}

\begin{proof}
Fix a localized package \(\mathfrak D\in\calU_\Lambda^{\loc}\).  By
\Cref{prop:split-transfer-inputs},
\[
    \|\calF_\Lambda^{\loc}(\mathfrak D)\|_{\loc}
    \ge
    \|\calF_\Lambda^{\cl}(\Theta_\Lambda\mathfrak D)\|_{\cl}
    -
    \Err_\Lambda(\mathfrak D).
\]
The clean finite-window gap gives
\[
    \|\calF_\Lambda^{\cl}(\Theta_\Lambda\mathfrak D)\|_{\cl}
    \ge
    c_\Lambda^{\cl}
    \Dist_{\cl}
    (
    \Theta_\Lambda\mathfrak D,\Image G_\Lambda^{\cl}
    ).
\]
Using again \Cref{prop:split-transfer-inputs} for the quotient comparison,
\[
    \Dist_{\cl}
    (
    \Theta_\Lambda\mathfrak D,\Image G_\Lambda^{\cl}
    )
    \ge
    (1-\varepsilon_G)
    \Dist_{\loc}
    (
    \mathfrak D,\Image G_\Lambda^{\loc}
    )
    -
    \delta_G .
\]
Combining these three inequalities gives
\[
\begin{aligned}
    \|\calF_\Lambda^{\loc}(\mathfrak D)\|_{\loc}
    &\ge
    c_\Lambda^{\cl}(1-\varepsilon_G)
    \Dist_{\loc}
    (
    \mathfrak D,\Image G_\Lambda^{\loc}
    )
    -
    c_\Lambda^{\cl}\delta_G
    -
    \Err_\Lambda(\mathfrak D).
\end{aligned}
\]
The normalized error estimate in \Cref{prop:split-transfer-inputs} implies
\[
    \Err_\Lambda(\mathfrak D)
    \le
    \eta_\Lambda
    \Dist_{\loc}
    (
    \mathfrak D,\Image G_\Lambda^{\loc}
    )
    +
    \Delta_\Lambda .
\]
Therefore
\[
\begin{aligned}
    \|\calF_\Lambda^{\loc}(\mathfrak D)\|_{\loc}
    &\ge
    \bigl(c_\Lambda^{\cl}(1-\varepsilon_G)-\eta_\Lambda\bigr)
    \Dist_{\loc}
    (
    \mathfrak D,\Image G_\Lambda^{\loc}
    )
    -
    \Delta_\Lambda
    -
    c_\Lambda^{\cl}\delta_G  \\
    &=
    c_\Lambda^{\loc}
    \Dist_{\loc}
    (
    \mathfrak D,\Image G_\Lambda^{\loc}
    )
    -
    \Delta_\Lambda' .
\end{aligned}
\]
The smallness condition gives \(c_\Lambda^{\loc}>0\).  Expanding the weighted
localized detection norm yields the first displayed inequality in the theorem.
Moving the nonnegative positive-profit term
\([\Prof_\Lambda^{\loc}(\mathfrak D)]_+\) to the right-hand side gives the
equivalent controlled ledger-profit form.
\end{proof}

\begin{remark}[Status of the theorem]
The theorem is a conditional finite-window perturbative transfer statement.
It is algebraic: it applies to any element of the localized finite-window chart
domain satisfying the stated structural hypotheses.  It does not prove
Navier--Stokes regularity, does not exclude singularities, and does not by
itself construct a Navier--Stokes-derived package.  The latter step is separated
in \Cref{cor:ns-derived-transfer}.
\end{remark}

\begin{corollary}[NS-derived localized packages]
\label[corollary]{cor:ns-derived-transfer}
Assume, in addition to the hypotheses of
\Cref{thm:main-local-to-clean-transfer}, the localized package realization
hypothesis \Cref{ass:localized-package-realization}.  Let \((u,p)\) be a
suitable weak solution whose associated localized finite-window package
\(\mathfrak D_\Lambda^{\loc}\) lies in the chart domain
\(\calU_\Lambda^{\loc}\).  Then the two inequalities of
\Cref{thm:main-local-to-clean-transfer} hold with
\(\mathfrak D=\mathfrak D_\Lambda^{\loc}\).
\end{corollary}

\begin{proof}
By \Cref{ass:localized-package-realization}, the suitable weak solution
produces a finite-dimensional localized package
\(\mathfrak D_\Lambda^{\loc}\in\calD_\Lambda^{\loc}\).  The extra hypothesis
places this package in the chart domain.  Applying
\Cref{thm:main-local-to-clean-transfer} to
\(\mathfrak D_\Lambda^{\loc}\) gives the claim.
\end{proof}

\section{Final synthesis and closure targets}\label{sec:final-synthesis}

We finish by combining the conceptual audit with the local-to-clean transfer
result.  This final statement is best read as a synthesis corollary: the main
finite-window coercive estimate is \Cref{thm:main-local-to-clean-transfer}, while
the corollary below records how that estimate fits into the larger singularity
audit.  A persistent non-CKN branch must pass the payment test, the
combined-observability test, the NS-realizability test, the reproduction test,
and the localized residual-budget test.  If any of these tests closes uniformly,
the branch must enter a CKN scale.  If none closes, the remaining object is no
longer a generic concentration of energy or pressure; it is an NS-realizable,
cleaned, combined-invisible, profitable recurrence.

\begin{corollary}[Unified singularity-audit and transfer alternative]\label[corollary]{cor:unified-audit-transfer}
Let \((u,p)\) be a suitable weak solution near a point \(z_0\).  Suppose a dyadic
branch based at \(z_0\) remains outside the CKN smallness regime.  Assume dyadic
defect extraction, the finite-scale ledger, observable depletion or visible-supply
taxation with summable errors and a lower-bounded depletion budget, moving-window
growth control, local-to-clean quotient lifting, componentwise detection
comparison, concrete residual-budget absorption, and a near-kernel extraction
principle that converts persistent cleaned near-invisibility with small
reproduction residual into an NS-realizable PRV cascade.  Then one of the following
alternatives occurs:
\begin{enumerate}[label=(\alph*),leftmargin=2em]
\item accumulated localization/window leakage is non-negligible;
\item positive average untaxed supply persists after the available taxes are
charged;
\item moving-window observability constants or cleaning constants grow too fast
to close the argument;
\item a nonzero NS-realizable, cleaned, scale-critical, combined-invisible,
profitable, reproducible defect cascade exists;
\item in a one-component degeneration, the obstruction appears as a vertical-duality
or anti-phantom failure for the active residual quotient.
\end{enumerate}
Consequently, if leakage is negligible, visible supply is depleted or taxed,
moving-window constants are controlled, normalized absorption holds, and such
PRV cascades are excluded, then the branch must enter a CKN scale and \(z_0\) is
regular.  The theorem is conditional in precisely the hypotheses listed above;
it does not assert that the depletion, absorption, or near-kernel extraction
inputs have been proved in this paper.
\end{corollary}

\begin{proof}
The proof combines the earlier conditional alternatives.  The finite-scale ledger gives the payment alternative: persistent reservoir badness forces untaxed supply or leakage.  If leakage is non-negligible, (a) holds.  If positive average supply remains after the available taxes are charged, (b) holds.  Otherwise the visible part of the supply is taxed through the pressure, flux, energy, or trace channels.  The observable-depletion input, its lower-bounded budget, and the summability of the error terms forbid an infinite controlled branch with a non-summable amount of visible observation.  If the moving-window constants lose effectiveness before this conclusion is reached, (c) holds.  If not, the branch has a subsequence of cleaned localized packages approaching the combined-invisible NS-realizable kernel.  The local-to-clean transfer theorem supplies the localized coercive comparison under normalized absorption, and the reproduction comparison supplies the scale-to-scale residual control.  The near-kernel extraction principle then produces the PRV cascade in (d), with profitability inherited from the ledger.  In the one-component channel, failure of the required vertical-duality closure is recorded as (e).  If all alternatives are excluded, the original non-CKN branch cannot persist, and therefore a CKN scale occurs.
\end{proof}

\begin{problem}[Localized quotient compatibility]
The concrete theorem target left by this paper is to prove component
compatibility estimates of the form
\[
    \Err_{\bullet}(D)
    \le
    \eta_{\bullet}\dist_{\mathrm{loc}}(D,\operatorname{Im}G^{\mathrm{loc}}_\Lambda)
    +\Delta_{\bullet},
\]
with coefficients summing below the clean gap margin.  This requires either a
sharper localized quotient distance or a structural observability theorem.
\end{problem}

\begin{problem}[NS-realizable anti-phantom closure]
Prove that the combined-invisible finite-window kernel contains no nonzero
direction in the closure of cleaned NS-realizable packages, or classify the
remaining kernel directions as genuine PRV mechanisms.
\end{problem}

\begin{problem}[Moving-window growth control]
Control the scale dependence of observability, cleaning, pressure-transfer,
truncation, and reproduction constants strongly enough to prevent non-summable
escape along selected moving windows.
\end{problem}

\section{Conclusion}

The paper develops a finite-window language for separating genuine
Navier--Stokes obstructions from artifacts of cutoff, pressure normalization,
finite projection, and clean-model representation.  The unconditional PDE part
is the critical ledger: persistent reservoir badness must be paid for by
untaxed supply or leakage.  The main algebraic part is the local-to-clean
anti-phantom transfer theorem: a clean quotient gap becomes a localized coercive
estimate once quotient lifting, componentwise detection comparison, and residual
absorption are available. The remaining blocker is therefore sharply localized.  One must either prove
component compatibility estimates against the localized quotient distance, with
coefficients below the clean gap margin, or construct an NS-realizable recurrence
that survives those tests.  In this sense a possible singularity, within the
present framework, is forced to be highly organized: it must be cleaned,
combined-invisible, profitable, reproducible, and compatible with the momentum,
pressure, flux, gauge, and residual ledgers.  Closing or realizing that final
object is the next theorem-level step.

\end{document}